\documentclass[10pt, a4paper]{amsart}
\usepackage{amscd,amsmath,amssymb,amsfonts,amsthm,ascmac}
\usepackage{ytableau,enumerate,mathrsfs,stmaryrd,latexsym} 
\usepackage[all]{xy}
\usepackage{color, graphicx}
\usepackage[top=27truemm,bottom=21truemm,left=26truemm,right=26truemm]{geometry}
\def\r{\mathbb{R}}
\def\c{\mathbb{C}}

\def\q{\mathbb{Q}}

\def\z{\mathbb{Z}}
\def\p{\mathbb{P}}

\def\a{\alpha}

\newtheorem{thm}{Theorem}[section]
\newtheorem{defi}[thm]{Definition}

\newtheorem{rem}[thm]{Remark}
\newtheorem{prop}[thm]{Proposition}
\newtheorem{ex}[thm]{Example}

\newtheorem{cor}[thm]{Corollary}
\newtheorem{lem}[thm]{Lemma}

\newtheorem*{ack}{Acknowledgements}
\newtheorem*{thm1}{Theorem 1}
\newtheorem*{prop1}{Proposition}
\newtheorem*{cor1}{Corollary}
\newtheorem*{thm2}{Theorem 2}
\title[NEWTON-OKOUNKOV BODIES FOR BOTT-SAMELSON VARIETIES AND STRING POLYTOPES]{\fontsize{11pt}{11pt}\selectfont NEWTON-OKOUNKOV BODIES FOR BOTT-SAMELSON VARIETIES AND STRING POLYTOPES FOR GENERALIZED DEMAZURE MODULES}
\date{}
\author[N. Fujita]{\fontsize{10pt}{10pt}\selectfont Naoki Fujita}
\begin{document}
\ytableausetup{smalltableaux}
\address{Department of Mathematics, Tokyo Institute of Technology, 2-12-1 Oh-okayama, Meguro-ku, Tokyo 152-8551, Japan}
\email{fujita.n.ac@m.titech.ac.jp}
\begin{abstract}
A Newton-Okounkov convex body is a convex body constructed from a projective variety with a valuation on its homogeneous coordinate ring; this generalizes a Newton polytope for a toric variety. This convex body has various kinds of information about the original projective variety; for instance, Kaveh showed that the string polytopes from representation theory are examples of Newton-Okounkov bodies for Schubert varieties. In this paper, we extend the notion of string polytopes for Demazure modules to generalized Demazure modules, and prove that the resulting generalized string polytopes are identical to the Newton-Okounkov bodies for Bott-Samelson varieties with respect to a specific valuation. As an application of this result, we show that these are indeed polytopes.
\end{abstract}
\maketitle
\setcounter{tocdepth}{1}
\tableofcontents
\section[Introduction]{Introduction}
\indent
The theory of Newton-Okounkov bodies is introduced by Okounkov in order to study multiplicity functions for representations of a reductive group (\cite{O1}, \cite{O2}), and afterward developed independently by Kaveh-Khovanskii (\cite{KK1}) and by Lazarsfeld-Mustata (\cite{LM}). Let $X$ be a normal projective variety of complex dimension $r$, and $\mathcal{L}$ a very ample line bundle on $X$. Take a valuation $v$ on the function field $\c(X)$ with values in $\z^r$, and fix a nonzero section $\tau \in H^0(X, \mathcal{L})$; in addition, we assume that $v$ has one-dimensional leaves (see \S\S 2.1 for the definition). From these data, we construct a semigroup $S(X, \mathcal{L}, v, \tau) \subset \z_{>0} \times \z^r$ (see Definition \ref{Newton-Okounkov body}). If we denote by $C(X, \mathcal{L}, v, \tau) \subset \r_{\ge 0} \times \r^r$ the smallest real closed cone containing $S(X, \mathcal{L}, v, \tau)$, then the Newton-Okounkov body $\Delta(X, \mathcal{L}, v, \tau) \subset \r^r$ is defined to be the intersection of $C(X, \mathcal{L}, v, \tau)$ and $\{1\} \times \r^r$. The theory of Newton-Okounkov bodies is a generalization of that of Newton polytopes for toric varieties, and these convex bodies have important geometric information about $X$; for instance, we can systematically construct a series of toric degenerations of $X$ by the following result.

\vspace{2mm}\begin{prop1}[{\cite[Corollary 3.14]{HK} and \cite[Theorem 1]{A}}]
Assume that the semigroup $S(X, \mathcal{L}, v, \tau)$ is finitely generated, and hence the Newton-Okounkov body $\Delta(X, \mathcal{L}, v, \tau)$ is a rational convex polytope. Then, there exists a flat degeneration of $X$ to a $($not necessarily normal$)$ toric variety $X_0 := {\rm Proj} (\c[S(X, \mathcal{L}, v, \tau)]);$ note that the normalization of $X_0$ is the normal toric variety corresponding to the polytope $\Delta(X, \mathcal{L}, v, \tau)$.
\end{prop1}\vspace{2mm}

\noindent However, the semigroup $S(X, \mathcal{L}, v, \tau)$ is not finitely generated in general, and the Newton-Okounkov body $\Delta(X, \mathcal{L}, v, \tau)$ may not be a rational convex polytope. Hence the following is a fundamental question: ``when is the semigroup $S(X, \mathcal{L}, v, \tau)$ finitely generated?'' Although it is difficult to give a complete answer to this question, there exist some partial results in this direction (see, for instance, \cite{A}, \cite{SS}). In this paper, we provide a series of examples, in which the semigroup $S(X, \mathcal{L}, v, \tau)$ is indeed finitely generated, and hence the Newton-Okounkov body $\Delta(X, \mathcal{L}, v, \tau)$ is indeed a rational convex polytope.

A remarkable fact is that the theory of Newton-Okounkov bodies is deeply connected with representation theory; for instance, Kaveh (\cite{Kav}) proved that the Newton-Okounkov body of a Schubert variety with respect to a specific valuation is identical to the string polytope constructed from the string parameterization for a Demazure crystal. The purpose of this paper is to extend this result to Bott-Samelson varieties. 

To be more precise, let $G$ be a connected, simply-connected semisimple algebraic group over $\c$, $\mathfrak{g}$ its Lie algebra, $B \subset G$ a Borel subgroup, and $I$ an index set for the vertices of the Dynkin diagram. For simplicity, we deal with only finite-dimensional Lie algebra $\mathfrak{g}$; but, our result can be extended to a symmetrizable Kac-Moody Lie algebra without much difficulty. 

An arbitrary word ${\bf i} = (i_1, \ldots, i_r) \in I^r$ gives a smooth projective variety $Z_{\bf i}$, called a Bott-Samelson variety, and we can associate to ${\bf m} = (m_1, \ldots, m_r) \in \z_{\ge 0} ^r$ a line bundle $\mathcal{L}_{{\bf i}, {\bf m}}$ on $Z_{\bf i}$; note that we need not assume that $\mathcal{L}_{{\bf i}, {\bf m}}$ is very ample. We consider a specific local coordinate $(t_1, \ldots, t_r)$ in $Z_{\bf i}$ (see \S\S 2.3), and identify the function field $\c(Z_{\bf i})$ with the rational function field $\c(t_1, \ldots, t_r)$. Define a valuation $v_{\bf i}$ on $\c(Z_{\bf i})$ to be the highest term valuation on $\c(t_1, \ldots, t_r)$ with respect to the lexicographic order on $\z^r$ (see Example \ref{highest term valuation}). We then take a specific section of $\mathcal{L}_{{\bf i}, {\bf m}}$. Note that the space $H^0(Z_{\bf i}, \mathcal{L}_{{\bf i}, {\bf m}})$ of global sections has a natural $B$-module structure (see \S\S 2.2). For a dominant integral weight $\lambda$, let us denote by $V(\lambda)$ the irreducible highest weight $G$-module with highest weight $\lambda$. Then, the dual $B$-module $V_{{\bf i}, {\bf m}} := H^0(Z_{\bf i}, \mathcal{L}_{{\bf i}, {\bf m}})^\ast$ is realized as a $B$-submodule of $V(m_1 \varpi_{i_1}) \otimes \cdots \otimes V(m_r \varpi_{i_r})$, where $\varpi_i$, $i \in I$, denote the fundamental weights; this $B$-submodule $V_{{\bf i}, {\bf m}}$ is called a generalized Demazure module. We regard $H^0(Z_{\bf i}, \mathcal{L}_{{\bf i}, {\bf m}}) = V_{{\bf i}, {\bf m}} ^\ast$ as a quotient $B$-module of $(V(m_1 \varpi_{i_1}) \otimes \cdots \otimes V(m_r \varpi_{i_r}))^\ast$, and denote by $\tau_{{\bf i}, {\bf m}} \in H^0(Z_{\bf i}, \mathcal{L}_{{\bf i}, {\bf m}})$ the image of the lowest weight vector in $(V(m_1 \varpi_{i_1}) \otimes \cdots \otimes V(m_r \varpi_{i_r}))^\ast$. In this setting, we study the Newton-Okounkov body $\Delta(Z_{\bf i}, \mathcal{L}_{{\bf i}, {\bf m}}, v_{\bf i}, \tau_{{\bf i}, {\bf m}})$.

Let $U_q(\mathfrak{g})$ denote the quantized enveloping algebra of $\mathfrak{g}$, and $\mathcal{B}(\lambda)$ the crystal basis of the irreducible highest weight $U_q(\mathfrak{g})$-module $V_q(\lambda)$ with highest weight $\lambda$. About the generalized Demazure module $V_{{\bf i}, {\bf m}}$, Lakshmibai-Littelmann-Magyar (\cite{LLM}) introduced a certain subset $\mathcal{B}_{{\bf i}, {\bf m}} \subset \mathcal{B}(m_1 \varpi_{i_1}) \otimes \cdots \otimes \mathcal{B}(m_r \varpi_{i_r})$, called a generalized Demazure crystal, which gives the character of $V_{{\bf i}, {\bf m}}$; also, they constructed an explicit basis of $H^0(Z_{\bf i}, \mathcal{L}_{{\bf i}, {\bf m}})$ parameterized by $\mathcal{B}_{{\bf i}, {\bf m}}$, called a standard monomial basis. In this paper, we give a different basis of $H^0(Z_{\bf i}, \mathcal{L}_{{\bf i}, {\bf m}})$ parameterized by $\mathcal{B}_{{\bf i}, {\bf m}}$, which can be thought of as a perfect basis (see \cite[Definition 5.30]{BK} and \cite[Definition 2.5]{KOP} for the definition). In \S\S 5.1, we extend the notion of string parameterization (resp., string polytope) to the generalized Demazure crystal $\mathcal{B}_{{\bf i}, {\bf m}}$, which we denote by $\Omega_{\bf i}$ (resp., $\Delta_{{\bf i}, {\bf m}}$); see Appendix B for some examples of the generalized string polytope $\Delta_{{\bf i}, {\bf m}}$. Some properties of a usual string parameterization are naturally extended to this parameterization $\Omega_{\bf i}$ for $\mathcal{B}_{{\bf i}, {\bf m}}$. \S\S 5.2 is devoted to the study of these properties; for instance, we give a system of piecewise-linear inequalities defining $\Delta_{{\bf i}, {\bf m}}$, and show that $\Delta_{{\bf i}, {\bf m}}$ is a finite union of rational convex polytopes. 

In Section 6, we construct an upper global basis of a tensor product module, following \cite{Lus} and \cite{Kas5}. By specializing this basis at $q=1$, we obtain a specific basis $\{G^{\rm up}(b^\ast) \mid b \in \mathcal{B}(m_1 \varpi_{i_1}) \otimes \cdots \otimes \mathcal{B}(m_r \varpi_{i_r})\}$ of $(V(m_1 \varpi_{i_1}) \otimes \cdots \otimes V(m_r \varpi_{i_r}))^\ast$. Let $G^{\rm up} _{{\bf i}, {\bf m}}(b) \in H^0(Z_{\bf i}, \mathcal{L}_{{\bf i}, {\bf m}})$ denote the image of $G^{\rm up}(b^\ast)$ under the quotient map $(V(m_1 \varpi_{i_1}) \otimes \cdots \otimes V(m_r \varpi_{i_r}))^\ast \twoheadrightarrow H^0(Z_{\bf i}, \mathcal{L}_{{\bf i}, {\bf m}})$. The following is the first main result of this paper.

\vspace{2mm}\begin{thm1}
Let ${\bf i} \in I^r$ be an arbitrary word, and ${\bf m} \in \z_{\ge 0} ^r$. 
\begin{enumerate}
\item[{\rm (1)}] The set $\{G^{\rm up} _{{\bf i}, {\bf m}} (b) \mid b \in \mathcal{B}_{{\bf i}, {\bf m}}\}$ forms a $\c$-basis of $H^0(Z_{\bf i}, \mathcal{L}_{{\bf i}, {\bf m}})$.
\item[{\rm (2)}] The generalized string parameterization $\Omega_{\bf i} (b)$ is equal to $- v_{\bf i} (G^{\rm up} _{{\bf i}, {\bf m}} (b)/\tau_{{\bf i}, {\bf m}})$ for all $b \in \mathcal{B}_{{\bf i}, {\bf m}}$.
\item[{\rm (3)}] The generalized string polytope $\Delta_{{\bf i}, {\bf m}}$ is identical to $-\Delta(Z_{\bf i}, \mathcal{L}_{{\bf i}, {\bf m}}, v_{\bf i}, \tau_{{\bf i}, {\bf m}})$.
\end{enumerate}
\end{thm1}\vspace{2mm}

Since $\Delta_{{\bf i}, {\bf m}}$ is a finite union of rational convex polytopes and $\Delta(Z_{\bf i}, \mathcal{L}_{{\bf i}, {\bf m}}, v_{\bf i}, \tau_{{\bf i}, {\bf m}})$ is a convex body, we obtain the following.

\vspace{2mm}\begin{cor1}
The generalized string polytope $\Delta_{{\bf i}, {\bf m}}$ and the Newton-Okounkov body $\Delta(Z_{\bf i}, \mathcal{L}_{{\bf i}, {\bf m}}, v_{\bf i}, \tau_{{\bf i}, {\bf m}})$ are both rational convex polytopes.
\end{cor1}\vspace{2mm}

Now let $s_i \in G$, $i \in I$, be lifts of the simple reflections in the Weyl group, and assume that ${\bf i} = (i_1, \ldots, i_r) \in I^r$ is a reduced word. In Section 8, we consider a specific local coordinate $(t_1 ^\prime, \ldots, t_r ^\prime)$ at the image of $(s_{i_1}, \ldots, s_{i_r})$ in $Z_{\bf i}$, and define a valuation $v_{\bf i} ^\prime$ on $\c(Z_{\bf i})$ to be the highest term valuation on $\c(t_1 ^\prime, \ldots, t_r ^\prime)$ with respect to the lexicographic order on $\z^r$; note that $v_{\bf i} ^\prime$ is not equal to $v_{\bf i}$ as a valuation on $\c(Z_{\bf i})$. Also, we take another section $\tau_{{\bf i}, {\bf m}} ^\prime \in H^0(Z_{\bf i}, \mathcal{L}_{{\bf i}, {\bf m}})$, and introduce another parameterization $\Omega_{\bf i} ^\prime$ for $\mathcal{B}_{{\bf i}, {\bf m}}$. Replacing $\Omega_{\bf i}$ by $\Omega_{\bf i} ^\prime$ in the definition of $\Delta_{{\bf i}, {\bf m}}$, we obtain another compact set $\Delta_{{\bf i}, {\bf m}} ^\prime$. The following is the second main result of this paper.

\vspace{2mm}\begin{thm2}
Let ${\bf i} \in I^r$ be an arbitrary reduced word.
\begin{enumerate}
\item[{\rm (1)}] For all ${\bf m} \in \z^r _{\ge 0}$ and $b \in \mathcal{B}_{{\bf i}, {\bf m}}$, the parameterization $\Omega_{\bf i} ^\prime(b)$ is equal to $-v_{\bf i} ^\prime(G^{\rm up} _{{\bf i}, {\bf m}} (b)/\tau_{{\bf i}, {\bf m}} ^\prime)$. Moreover, the set $\Delta_{{\bf i}, {\bf m}} ^\prime$ is identical to the Newton-Okounkov body $-\Delta(Z_{\bf i}, \mathcal{L}_{{\bf i}, {\bf m}}, v_{\bf i} ^\prime, \tau_{{\bf i}, {\bf m}} ^\prime)$.
\item[{\rm (2)}] There exist explicit unimodular $r \times r$-matrices $A$ and $B$ such that $\Delta_{{\bf i}, {\bf m}} = A \Delta_{{\bf i}, {\bf m}} ^\prime + B {\bf m}$ for all ${\bf m} \in \z^r _{\ge 0}$.
\item[{\rm (3)}] The set $\Delta_{{\bf i}, {\bf m}} ^\prime$ and the Newton-Okounkov body $\Delta(Z_{\bf i}, \mathcal{L}_{{\bf i}, {\bf m}}, v_{\bf i} ^\prime, \tau_{{\bf i}, {\bf m}} ^\prime)$ are both rational convex polytopes.
\end{enumerate}
\end{thm2}\vspace{2mm}

Finally, we should mention that Schmitz and Sepp\"{a}nen (\cite{SS}) also proved that some Newton-Okounkov bodies of $Z_{\bf i}$ are indeed rational convex polytopes. However, our approach in this paper is quite different from theirs.

\vspace{2mm}\begin{ack}\normalfont
The author is deeply indebted to Professor S. Naito for numerous helpful suggestions and fruitful discussions. The author would also like to thank M. Harada for teaching me the background of the theory of Newton-Okounkov bodies.
\end{ack}
\section{Newton-Okounkov bodies}
\subsection{Basic definitions}

First of all, we review the definition of Newton-Okounkov bodies, following \cite{HK}, \cite{Kav}, \cite{KK1}, and \cite{KK2}. Let $R$ be a $\c$-algebra without nonzero zero-divisors, and $<$ the lexicographic order on $\z^r$, $r \ge 1$, which is given by $(a_1, \ldots, a_r) < (a_1 ^\prime, \ldots, a_r ^\prime)$ if and only if there exists $1 \le k \le r$ such that $a_1 = a_1 ^\prime, \ldots, a_{k-1} = a_{k-1} ^\prime$, $a_k < a_k ^\prime$. 

\vspace{2mm}\begin{defi}\normalfont\label{def,val}
A map $v: R \setminus \{0\} \rightarrow \z^r$ is called a {\it valuation} on $R$ with values in $\z^r$ if the following hold: for every $\sigma, \tau \in R \setminus \{0\}$ and $c \in \c \setminus \{0\}$,
\begin{enumerate}
\item[{\rm (i)}]  $v(\sigma \cdot \tau) = v(\sigma) + v(\tau)$,
\item[{\rm (ii)}] $v(c \cdot \sigma) = v(\sigma)$, 
\item[{\rm (iii)}] $v (\sigma + \tau) \ge \min\{v(\sigma), v(\tau)\}$ unless $\sigma + \tau = 0$. 
\end{enumerate}
\end{defi}\vspace{2mm}

In this paper, we always assume that the $\z$-lattice $\z^r$ is equipped with the lexicographic order. The following is a fundamental property of valuations.

\vspace{2mm}\begin{prop}[{see, for instance, \cite[Proposition 1.8]{Kav}}]\label{prop1,val}
Let $v$ be a valuation on $R$. Assume that $\sigma_1, \ldots, \sigma_s \in R \setminus \{0\}$, and that $v(\sigma_1), \ldots, v(\sigma_s)$ are distinct.
\begin{enumerate}
\item[{\rm (1)}] The elements $\sigma_1, \ldots, \sigma_s$ are linearly independent over $\c$.
\item[{\rm (2)}] For $c_1, \ldots, c_s \in \c$ such that $\sigma := c_1 \sigma_1 + \cdots + c_s \sigma_s \neq 0$, \[v(\sigma) = \min\{v(\sigma_t) \mid 1 \le t \le s,\ c_t \neq 0 \}.\]
\end{enumerate}
\end{prop}\vspace{2mm}

For ${\bf a} \in \z^r$ and a valuation $v$ on $R$ with values in $\z^r$, we set $R_{\bf a} := \{\sigma \in R \mid \sigma = 0\ {\rm or}\ v(\sigma) \ge {\bf a}\}$; this is a $\c$-subspace of $R$. The {\it leaf} above ${\bf a} \in \z^r$ is defined to be the quotient space $\widehat{R}_{\bf a} := R_{\bf a}/\bigcup_{{\bf a}<{\bf b}} R_{\bf b}$. A valuation $v$ is said to have {\it one-dimensional leaves} if dim$(\widehat{R}_{\bf a}) = 0\ {\rm or}\ 1$ for all ${\bf a} \in \z^r$.

\vspace{2mm}\begin{prop}[{see, for instance, \cite[Proposition 1.9]{Kav}}]\label{prop2,val}
Let $v$ be a valuation on $R$ with one-dimensional leaves, and $H \subset R$ a finite-dimensional $\c$-subspace. Then, there exists a basis $\{\sigma_1, \ldots, \sigma_s \}$ of $H$ such that $v(\sigma_1), \ldots, v(\sigma_s)$ are distinct. In particular, the complex dimension of $H$ is equal to the number of distinct elements in $v(H \setminus \{0\})$.
\end{prop}

\vspace{2mm}\begin{ex}\normalfont\label{highest term valuation}
Let $\c(t_1, \ldots, t_r)$ denote the rational function field in $r$ variables. The lexicographic order on $\z^r$ induces a total order (denoted by the same symbol $<$) on the set of all monomials in the variables $t_1, \ldots, t_r$ as follows: $t_1 ^{a_1} \cdots t_r ^{a_r} < t_1 ^{a_1 ^\prime} \cdots t_r ^{a_r ^\prime}$ if and only if $(a_1, \ldots, a_r) < (a_1 ^\prime, \ldots, a_r ^\prime)$. Let us define a map $v: \c(t_1, \ldots, t_r) \setminus \{0\} \rightarrow \z^r$ by $v(f/g) := v(f) - v(g)$ for $f, g \in \c[t_1, \ldots, t_r] \setminus \{0\}$, and by \[v(f) := -(a_1, \ldots, a_r)\ {\rm for}\ f = c t_1 ^{a_1} \cdots t_r ^{a_r} + ({\rm lower\ terms}) \in \c[t_1, \ldots, t_r] \setminus \{0\},\] where $c \in \c \setminus \{0\}$, and by ``lower terms'', we mean a linear combination of monomials smaller than $t_1 ^{a_1} \cdots t_r ^{a_r}$ with respect to the total order $<$ above. It is obvious that $v$ is a valuation with one-dimensional leaves. This valuation $v$ is called the {\it highest term valuation} with respect to the lexicographic order $<$ on $\z^r$.
\end{ex}\vspace{2mm}

Let $X$ be a normal projective variety over $\c$ of complex dimension $r$, and $\mathcal{L}$ a very ample line bundle on $X$. The {\it ring} $R(\mathcal{L})$ {\it of sections} is the $\z_{\ge 0}$-graded $\c$-algebra obtained from $\mathcal{L}$ by \[R(\mathcal{L}) := \bigoplus_{k \ge 0} H^0 (X, \mathcal{L}^{\otimes k}).\] If we take a nonzero section $\tau \in H^0 (X, \mathcal{L})$, then the $\c$-vector space $R(\mathcal{L})_k := H^0 (X, \mathcal{L}^{\otimes k})$ can be regarded as a $\c$-subspace of the function field $\c(X)$ as follows: \[R(\mathcal{L})_k \hookrightarrow \c(X),\ \sigma \mapsto \sigma/\tau^k.\] Hence a valuation on $\c(X)$ induces a map from $R(\mathcal{L})_k \setminus \{0\}$ to $\z^r$.

\vspace{2mm}\begin{defi}\normalfont\label{Newton-Okounkov body}
Let $v$ be a valuation on $\c(X)$ with values in $\z^r$, and $\tau \in H^0 (X, \mathcal{L})$ a nonzero section. Assume that $v$ has one-dimensional leaves. Define a subset $S(X, \mathcal{L}, v, \tau) \subset \z_{>0} \times \z^r$ by \[S(X, \mathcal{L}, v, \tau) := \bigcup_{k>0} \{(k, v(\sigma / \tau^k)) \mid \sigma \in R(\mathcal{L})_k \setminus \{0\}\},\] and denote by $C(X, \mathcal{L}, v, \tau) \subset \r_{\ge 0} \times \r^r$ the smallest real closed cone containing $S(X, \mathcal{L}, v, \tau)$, that is, \[C(X, \mathcal{L}, v, \tau) := \overline{\{c \cdot (k, {\bf a}) \mid c \in \r_{>0}\ {\rm and}\ (k, {\bf a}) \in S(X, \mathcal{L}, v, \tau)\}},\] where $\overline{H}$ means the closure of $H \subset \r_{\ge 0} \times \r^r$ with respect to the Euclidean topology. Now let us define a subset $\Delta(X, \mathcal{L}, v, \tau) \subset \r^r$ by \[\Delta(X, \mathcal{L}, v, \tau) := \{{\bf a} \in \r^r \mid (1, {\bf a}) \in C(X, \mathcal{L}, v, \tau)\};\] this is called the {\it Newton-Okounkov body} associated to $\mathcal{L}$, $v$, and $\tau$.
\end{defi}\vspace{2mm}

From the definition of valuations, it is obvious that $S(X, \mathcal{L}, v, \tau)$ is a semigroup. Hence it follows that $C(X, \mathcal{L}, v, \tau)$ is a closed convex cone, and that $\Delta(X, \mathcal{L}, v, \tau)$ is a convex set. Moreover, we deduce from \cite[Theorem 2.30 and Corollary 3.2]{KK2} that $\Delta(X, \mathcal{L}, v, \tau)$ is a convex body, i.e., a compact convex set, of real dimension $r$. Note that the Newton-Okounkov body $\Delta(X, \mathcal{L}, v, \tau)$ is not a polytope in general, that is, it may not be the convex hull of a finite number of points; if the semigroup $S(X, \mathcal{L}, v, \tau)$ is finitely generated, then the Newton-Okounkov body $\Delta(X, \mathcal{L}, v, \tau)$ is a rational convex polytope, i.e., the convex hull of a finite number of rational points.

\vspace{2mm}\begin{rem}\normalfont
Since $\mathcal{L}$ is a very ample line bundle, we can take a closed immersion $\rho: X \rightarrow \p (H^0 (X, \mathcal{L})^\ast )$ such that $\mathcal{L} = \rho^\ast (\mathcal{O}(1))$. Denote by $R = \bigoplus_{k \ge 0} R_k$ the homogeneous coordinate ring of $X$ with respect to the closed immersion $\rho$. In many literatures including \cite{HK}, Newton-Okounkov bodies are defined by using $R$ instead of $R(\mathcal{L})$. However, since $X$ is normal, we deduce from \cite[Chapter I\hspace{-.1em}I,  Ex.\ 5.14]{Hart} that $R_k = R(\mathcal{L})_k$ for all $k \gg 0$. In addition, since $S(X, \mathcal{L}, v, \tau)$ is a semigroup, the real closed cone $C(X, \mathcal{L}, v, \tau)$ is identical to the smallest real closed cone containing \[\bigcup_{k>k^\prime} \{(k, v(\sigma / \tau^k)) \mid \sigma \in R(\mathcal{L})_k \setminus \{0\}\}\] for $k^\prime \gg 0$. Therefore, $R$ and $R(\mathcal{L})$ are interchangeable in the definition of Newton-Okounkov bodies.
\end{rem}\vspace{2mm}

\begin{rem}\normalfont\label{independence}
If we take another section $\tau^\prime \in H^0 (X, \mathcal{L}) \setminus \{0\}$, then $S(X, \mathcal{L}, v, \tau^\prime)$ is the shift of $S(X, \mathcal{L}, v, \tau)$ by $k v(\tau/\tau^\prime)$ in $\{k\} \times \z^r$. Hence it follows that $\Delta(X, \mathcal{L}, v, \tau^\prime) = \Delta(X, \mathcal{L}, v, \tau) + v(\tau/\tau^\prime)$. Thus, the Newton-Okounkov body $\Delta(X, \mathcal{L}, v, \tau)$ does not essentially depend on the choice of $\tau \in H^0 (X, \mathcal{L}) \setminus \{0\}$.
\end{rem}

\subsection{Bott-Samelson varieties}

Here, we recall the definition of Bott-Samelson varieties and generalized Demazure modules, following \cite{LLM}. Let $G$ be a connected reductive algebraic group over $\c$ of rank $n$, $\mathfrak{g}$ its Lie algebra, $W$ the Weyl group, and $I = \{1, \ldots, n\}$ an index set for the vertices of the Dynkin diagram. Choose a Borel subgroup $B \subset G$ and a maximal torus $T \subset B$. We consider an arbitrary word ${\bf i} =(i_1, \ldots, i_r) \in I^r$; note that we do not necessarily assume that ${\bf i}$ is a reduced word. 

\vspace{2mm}\begin{defi}\normalfont
For a word ${\bf i} =(i_1, \ldots, i_r) \in I^r$, define the {\it Bott-Samelson variety} $Z_{\bf i}$ by \[Z_{\bf i} := (P_{i_1} \times \cdots \times P_{i_r})/B^r,\] where $P_i$, $i \in I$, denote the minimal parabolic subgroups, and $B^r$ acts on $P_{i_1} \times \cdots \times P_{i_r}$ on the right by $(p_1, \ldots, p_r)\cdot(b_1, \ldots, b_r) := (p_1 b_1, b_1 ^{-1} p_2 b_2, \ldots, b_{r-1} ^{-1} p_r b_r)$ for $p_1 \in P_{i_1}, \ldots, p_r \in P_{i_r}$, and $b_1, \ldots, b_r \in B$. 
\end{defi}\vspace{2mm}

\begin{rem}\normalfont
If we denote by $\widetilde{G}$ the connected, simply-connected semisimple algebraic group with the same Cartan matrix as $G$, then $Z_{\bf i}$ is also isomorphic to the Bott-Samelson variety for $\widetilde{G}$ corresponding to the same word ${\bf i}$.
\end{rem}\vspace{2mm}

By this remark, we may assume without loss of generality that $G$ is a connected, simply-connected semisimple algebraic group. Note that $Z_{\bf i}$ is a nonsingular (and hence normal) projective variety of complex dimension $r$. Denote by $\mathfrak{t}$ the Lie algebra of $T$, and set $\mathfrak{t}^\ast:=\ {\rm Hom}_\c(\mathfrak{t}, \c)$. Let $\{\alpha_i \mid i \in I\} \subset \mathfrak{t}^\ast$ be the set of simple roots, $\{h_i \mid i \in I\} \subset \mathfrak{t}$ the set of simple coroots, $\{\varpi_i \mid i \in I\} \subset \mathfrak{t}^\ast$ the set of fundamental weights, and $P \subset \mathfrak{t}^\ast$ the weight lattice. We regard a weight $\lambda \in P$ as a character of $B$. For ${\bf m} = (m_1, \ldots, m_r) \in \z^r$, define a variety $\mathcal{L}_{{\bf i}, {\bf m}}$ by \[\mathcal{L}_{{\bf i}, {\bf m}} := (P_{i_1} \times \cdots \times P_{i_r} \times \c)/B^r,\] where $B^r$ acts on $P_{i_1} \times \cdots \times P_{i_r} \times \c$ on the right by \[(p_1, \ldots, p_r, c) \cdot (b_1, \ldots, b_r) := (p_1 b_1, b_1 ^{-1} p_2 b_2, \ldots, b_{r-1} ^{-1} p_r b_r, (m_1 \varpi_{i_1})(b_1) \cdots (m_r \varpi_{i_r})(b_r) c)\] for $p_1 \in P_{i_1}, \ldots, p_r \in P_{i_r}, c \in \c$, and $b_1, \ldots, b_r \in B$. Then, the variety $\mathcal{L}_{{\bf i}, {\bf m}}$ induces a line bundle (denoted by the same symbol $\mathcal{L}_{{\bf i}, {\bf m}}$) on $Z_{\bf i}$ by the canonical projection \[\mathcal{L}_{{\bf i}, {\bf m}} \twoheadrightarrow Z_{\bf i},\ (p_1, \ldots, p_r, c) \bmod B^r \mapsto (p_1, \ldots, p_r) \bmod B^r.\] 

\vspace{2mm}\begin{prop}[{\cite{LT}}] Denote by {\rm Pic}$(Z_{\bf i})$ the Picard group of $Z_{\bf i}$.
\begin{enumerate}
\item[{\rm (1)}] The map $\z^r \xrightarrow{\sim}\ {\rm Pic}(Z_{\bf i})$, ${\bf m} \mapsto \mathcal{L}_{{\bf i}, {\bf m}}$, is an isomorphism of groups.
\item[{\rm (2)}] The line bundle $\mathcal{L}_{{\bf i}, {\bf m}}$ is very ample if and only if $m_1, \ldots, m_r >0$.
\item[{\rm (3)}] The line bundle $\mathcal{L}_{{\bf i}, {\bf m}}$ is generated by global sections if and only if $m_1, \ldots, m_r \ge 0$. 
\end{enumerate}
\end{prop}\vspace{2mm}

Let us define left actions of $P_{i_1}$ on $Z_{\bf i}$ and on $\mathcal{L}_{{\bf i}, {\bf m}}$ by
\begin{align*}
p \cdot ((p_1, \ldots, p_r) \bmod B^r) &:= (p p_1, p_2, \ldots, p_r) \bmod B^r,\\
p \cdot ((p_1, \ldots, p_r, c) \bmod B^r) &:= (p p_1, p_2, \ldots, p_r, c) \bmod B^r
\end{align*}
for $p, p_1 \in P_{i_1}$, $p_2 \in P_{i_2}, \ldots, p_r \in P_{i_r}$, and $c \in \c$. Since the projection $\mathcal{L}_{{\bf i}, {\bf m}} \twoheadrightarrow Z_{\bf i}$ is compatible with these actions, it follows that the space $H^0(Z_{\bf i}, \mathcal{L}_{{\bf i}, {\bf m}})$ of global sections has a natural $P_{i_1}$-module structure. In the following throughout this paper, we assume that $m_1, \ldots, m_r \ge 0$. The $P_{i_1}$-module $H^0(Z_{\bf i}, \mathcal{L}_{{\bf i}, {\bf m}})$ can be described more algebraically as follows. For a dominant integral weight $\lambda$, let $V(\lambda)$ denote the irreducible highest weight $G$-module with highest weight $\lambda$, and $v_\lambda \in V(\lambda)$ the highest weight vector. If we define a morphism $\Psi_{{\bf i}, {\bf m}}: Z_{\bf i} \rightarrow \p(V(m_1 \varpi_{i_1}) \otimes \cdots \otimes V(m_r \varpi_{i_r}))$ by \[\Psi_{{\bf i}, {\bf m}} ((p_1, \ldots, p_r) \bmod B^r) := \c(p_1 v_{m_1 \varpi_{i_1}} \otimes p_1 p_2 v_{m_2 \varpi_{i_2}} \otimes \cdots \otimes p_1 p_2 \cdots p_r v_{m_r \varpi_{i_r}}),\] then we have $\Psi_{{\bf i}, {\bf m}} ^\ast (\mathcal{O}(1)) = \mathcal{L}_{{\bf i}, {\bf m}}$. Hence the morphism $\Psi_{{\bf i}, {\bf m}}$ induces a surjection \[\Psi_{{\bf i}, {\bf m}} ^\ast: H^0(\p(V(m_1 \varpi_{i_1}) \otimes \cdots \otimes V(m_r \varpi_{i_r})), \mathcal{O}(1)) \twoheadrightarrow H^0(Z_{\bf i}, \mathcal{L}_{{\bf i}, {\bf m}}).\] For an arbitrary finite-dimensional $G$-module $V$ over $\c$, we remark that the space $H^0(\p(V), \mathcal{O}(1))$ of global sections is identified with the dual $G$-module $V^\ast :=\ {\rm Hom}_\c (V, \c)$. Therefore, the surjection $\Psi_{{\bf i}, {\bf m}} ^\ast$ is regarded as a $P_{i_1}$-module homomorphism from $(V(m_1 \varpi_{i_1}) \otimes \cdots \otimes V(m_r \varpi_{i_r}))^\ast$ onto the space $H^0(Z_{\bf i}, \mathcal{L}_{{\bf i}, {\bf m}})$ of global sections. Let us denote by $E_i , F_i, h_i \in \mathfrak{g}$, $i \in I$, the Chevalley generators.

\vspace{2mm}\begin{prop}[{\cite[Theorem 6]{LLM}}] Define a $P_{i_1}$-submodule $V_{{\bf i}, {\bf m}} \subset V(m_1 \varpi_{i_1}) \otimes \cdots \otimes V(m_r \varpi_{i_r})$ by \[V_{{\bf i}, {\bf m}} := \sum_{a_1, \ldots, a_r \ge 0} \c F_{i_1} ^{a_1} (v_{m_1 \varpi_{i_1}} \otimes F_{i_2} ^{a_2} (v_{m_2 \varpi_{i_2}} \otimes \cdots \otimes F_{i_{r-1}} ^{a_{r-1}} (v_{m_{r-1} \varpi_{i_{r-1}}} \otimes F_{i_r} ^{a_r} v_{m_r \varpi_{i_r}})\cdots)).\]
Then, the surjection $\Psi_{{\bf i}, {\bf m}} ^\ast$ induces an isomorphism of $P_{i_1}$-modules$:$ \[V_{{\bf i}, {\bf m}} ^\ast \xrightarrow{\sim} H^0 (Z_{\bf i}, \mathcal{L}_{{\bf i}, {\bf m}}).\]
\end{prop}\vspace{2mm}

The $P_{i_1}$-module $V_{{\bf i}, {\bf m}}$ is called a {\it generalized Demazure module}. As we will see in Section 4, this indeed generalizes the notion of Demazure module.

\subsection{Highest term valuations}

In this subsection, we introduce a specific valuation on $\c(Z_{\bf i})$, which we mainly use in this paper. Let $U$ (resp., $U^-$) denote the unipotent radical of the Borel subgroup $B$ (resp., the opposite Borel subgroup), and $U_i ^-:= \exp(\c F_i)$, $i \in I$, the opposite root subgroups. For the identity element $e \in G$, we can regard $U_{i_1} ^- \times \cdots \times U_{i_r} ^-$ as an affine open neighborhood of $(e, \ldots, e) \bmod B^r$ in $Z_{\bf i}$ by:
\begin{align*}
U_{i_1} ^- \times \cdots \times U_{i_r} ^- &\hookrightarrow Z_{\bf i},\\
(u_1, \ldots, u_r) &\mapsto (u_1, \ldots, u_r) \bmod B^r.
\end{align*}
By using the isomorphism of varieties \[\c^r \xrightarrow{\sim} U_{i_1} ^- \times \cdots \times U_{i_r} ^-,\ (t_1, \ldots, t_r) \mapsto (\exp(t_1 F_{i_1}), \ldots, \exp(t_r F_{i_r})),\] we identify the function field $\c(Z_{\bf i}) = \c(U_{i_1} ^- \times \cdots \times U_{i_r} ^-)$ with the rational function field $\c(t_1, \ldots, t_r)$. Define a valuation $v_{\bf i}$ on $\c(Z_{\bf i})$ to be the highest term valuation on $\c(t_1, \ldots, t_r)$ with respect to the lexicographic order on $\z^r$ (see Example \ref{highest term valuation}). We then take a specific section of $\mathcal{L}_{{\bf i}, {\bf m}}$. Let $\Phi_{{\bf i}, {\bf m}}: V_{{\bf i}, {\bf m}} \hookrightarrow V(m_1 \varpi_{i_1}) \otimes \cdots \otimes V(m_r \varpi_{i_r})$ be the inclusion map, and $\Phi_{{\bf i}, {\bf m}} ^\ast: (V(m_1 \varpi_{i_1}) \otimes \cdots \otimes V(m_r \varpi_{i_r}))^\ast \twoheadrightarrow V_{{\bf i}, {\bf m}} ^\ast = H^0 (Z_{\bf i}, \mathcal{L}_{{\bf i}, {\bf m}})$ the dual map. Also, we denote by $\tilde{\tau}_{{\bf i}, {\bf m}} \in (V(m_1 \varpi_{i_1}) \otimes \cdots \otimes V(m_r \varpi_{i_r}))^\ast$ the lowest weight vector such that $\tilde{\tau}_{{\bf i}, {\bf m}} (v_{m_1 \varpi_{i_1}} \otimes \cdots \otimes v_{m_r \varpi_{i_r}}) = 1$, and set $\tau_{{\bf i}, {\bf m}} := \Phi_{{\bf i}, {\bf m}} ^\ast (\tilde{\tau}_{{\bf i}, {\bf m}}) \in H^0 (Z_{\bf i}, \mathcal{L}_{{\bf i}, {\bf m}})$. In this setting, we will study the Newton-Okounkov body $\Delta(Z_{\bf i}, \mathcal{L}_{{\bf i}, {\bf m}}, v_{\bf i}, \tau_{{\bf i}, {\bf m}})$. Note that we do not necessarily assume that the line bundle $\mathcal{L}_{{\bf i}, {\bf m}}$ is very ample in this paper; hence the real dimension of $\Delta(Z_{\bf i}, \mathcal{L}_{{\bf i}, {\bf m}}, v_{\bf i}, \tau_{{\bf i}, {\bf m}})$ may be less than the complex dimension of $Z_{\bf i}$.

\vspace{2mm}\begin{rem}\normalfont
Since $\mathcal{L}_{{\bf i}, {\bf m}} ^{\otimes k} = \mathcal{L}_{{\bf i}, k{\bf m}}$ and $\tau_{{\bf i}, {\bf m}} ^k = \tau_{{\bf i}, {k {\bf m}}}$ in $H^0 (Z_{\bf i} , \mathcal{L}_{k{\bf m}})$ for all $k \in \z_{>0}$, it follows that \[S(Z_{\bf i}, \mathcal{L}_{{\bf i}, {\bf m}}, v_{\bf i}, \tau_{{\bf i}, {\bf m}}) = \bigcup_{k>0} \{ (k, v_{\bf i} (\sigma/\tau_{{\bf i}, k{\bf m}})) \mid \sigma \in H^0(Z_{\bf i}, \mathcal{L}_{{\bf i}, k{\bf m}}) \setminus \{0\}\}.\]
\end{rem}\vspace{2mm}

For $\sigma \in H^0 (Z_{\bf i}, \mathcal{L}_{{\bf i}, {\bf m}})$, the value $v_{\bf i} (\sigma/\tau_{{\bf i}, {\bf m}})$ can be described in terms of the Chevalley generators. In the rest of this section, we review this description, following \cite[Proposition 2.2]{Kav}. Let us first prove some lemmas.

\vspace{2mm}\begin{lem}\label{prop, lowest}
The section $\tau_{{\bf i}, {\bf m}}$ does not vanish on $U_{i_1} ^- \times \cdots \times U_{i_r} ^-\ (\hookrightarrow Z_{\bf i})$. In particular, the restriction of $\tau_{{\bf i}, {\bf m}} ^{-1}$ to $U_{i_1} ^- \times \cdots \times U_{i_r} ^-$ is an element of $H^0 (U_{i_1} ^- \times \cdots \times U_{i_r} ^-, \mathcal{L}_{{\bf i}, {\bf m}} ^{-1})$, and hence $\sigma/\tau_{{\bf i}, {\bf m}} \in \c[t_1, \ldots, t_r]$ $(= \c[U_{i_1} ^- \times \cdots \times U_{i_r} ^-])$ for all $\sigma \in H^0 (Z_{\bf i}, \mathcal{L}_{{\bf i}, {\bf m}})$.
\end{lem}

\begin{proof}
Note that the dual map $\Phi_{{\bf i}, {\bf m}} ^\ast: (V(m_1 \varpi_{i_1}) \otimes \cdots \otimes V(m_r \varpi_{i_r}))^\ast \twoheadrightarrow H^0 (Z_{\bf i}, \mathcal{L}_{{\bf i}, {\bf m}})$ is identical to the surjection \[\Psi_{{\bf i}, {\bf m}} ^\ast: H^0(\p(V(m_1 \varpi_{i_1}) \otimes \cdots \otimes V(m_r \varpi_{i_r})), \mathcal{O}(1)) \twoheadrightarrow H^0(Z_{\bf i}, \mathcal{L}_{{\bf i}, {\bf m}})\] defined in \S\S 2.2. For $u_1 \in U_{i_1} ^-, \ldots, u_r \in U_{i_r} ^-$, we see that 
\begin{align*}
&\Psi_{{\bf i}, {\bf m}}((u_1, \ldots, u_r) \bmod B^r) = \c(u_1 v_{m_1 \varpi_{i_1}} \otimes u_1 u_2 v_{m_2 \varpi_{i_2}} \otimes \cdots \otimes u_1 u_2 \cdots u_r v_{m_r \varpi_{i_r}}),\ {\rm and}\\ 
&u_1 v_{m_1 \varpi_{i_1}} \otimes u_1 u_2 v_{m_2 \varpi_{i_2}} \otimes \cdots \otimes u_1 u_2 \cdots u_r v_{m_r \varpi_{i_r}} = v_{m_1 \varpi_{i_1}} \otimes v_{m_2 \varpi_{i_2}} \otimes \cdots \otimes v_{m_r \varpi_{i_r}} + ({\rm other\ terms}), 
\end{align*}
where by ``other terms'', we mean a linear combination of weight vectors whose weight is not equal to $m_1 \varpi_{i_1} + \cdots + m_r \varpi_{i_r}$. Therefore, it follows from the definition of $\tilde{\tau}_{{\bf i}, {\bf m}}$ that \[\tilde{\tau}_{{\bf i}, {\bf m}} (u_1 v_{m_1 \varpi_{i_1}} \otimes u_1 u_2 v_{m_2 \varpi_{i_2}} \otimes \cdots \otimes u_1 u_2 \cdots u_r v_{m_r \varpi_{i_r}}) = 1,\] which implies the assertion of the lemma since $\tau_{{\bf i}, {\bf m}} = \Phi_{{\bf i}, {\bf m}} ^\ast (\tilde{\tau}_{{\bf i}, {\bf m}})$.
\end{proof}

We write ${\bf i}_{\ge s} :=(i_s, i_{s+1}, \ldots, i_r)$ and ${\bf m}_{\ge s} := (m_s, m_{s+1}, \ldots, m_r)$ for $s = 1,\ldots, r$. Then, the generalized Demazure module $V_{{\bf i}_{\ge s+1}, {\bf m}_{\ge s+1}}$ can be regarded as a $\c$-subspace of $V_{{\bf i}_{\ge s}, {\bf m}_{\ge s}}$ by: \[\iota_{s, s+1}: V_{{\bf i}_{\ge s+1}, {\bf m}_{\ge s+1}} \hookrightarrow V_{{\bf i}_{\ge s}, {\bf m}_{\ge s}},\ v \mapsto v_{m_s \varpi_{i_s}} \otimes v.\] Let $\iota_{s, s+1} ^\ast: H^0(Z_{{\bf i}_{\ge s}}, \mathcal{L}_{{\bf i}_{\ge s}, {\bf m}_{\ge s}}) \twoheadrightarrow H^0(Z_{{\bf i}_{\ge s+1}}, \mathcal{L}_{{\bf i}_{\ge s+1}, {\bf m}_{\ge s+1}})$ denote the dual map. Also, we obtain a sequence of subvarieties \[P_{i_r}/B = Z_{{\bf i}_{\ge r}} \subset Z_{{\bf i}_{\ge r-1}} \subset \cdots \subset Z_{{\bf i}_{\ge 1}} = Z_{\bf i},\] where the Bott-Samelson variety $Z_{{\bf i}_{\ge s+1}}$ is thought of as a closed subvariety of $Z_{{\bf i}_{\ge s}}$ by: \[Z_{{\bf i}_{\ge s+1}} \hookrightarrow Z_{{\bf i}_{\ge s}},\ (p_{s+1}, \ldots, p_r) \bmod B^{r-s} \mapsto (e, p_{s+1}, \ldots, p_r) \bmod B^{r-s+1}.\] Note that the open immersion $U_{i_1} ^- \times \cdots \times U_{i_r} ^- \hookrightarrow Z_{\bf i}$ induces an open immersion $U_{i_s} ^- \times \cdots \times U_{i_r} ^- \hookrightarrow Z_{{\bf i}_{\ge s}}$, and that the function field $\c(Z_{{\bf i}_{\ge s}})$ is identified with $\c(t_s, \ldots, t_r)$.

\vspace{2mm}\begin{lem}\label{restriction}
For $1 \le s \le r-1$ and $\sigma \in H^0(Z_{{\bf i}_{\ge s}}, \mathcal{L}_{{\bf i}_{\ge s}, {\bf m}_{\ge s}})$, \[(\sigma/\tau_{{\bf i}_{\ge s}, {\bf m}_{\ge s}})|_{t_s = 0} = \iota_{s, s+1} ^\ast(\sigma)/\tau_{{\bf i}_{\ge s+1}, {\bf m}_{\ge s+1}}\] in $\c[t_{s+1}, \ldots, t_r]\ (= \c[U_{i_{s+1}} ^- \times \cdots \times U_{i_r} ^-]);$ here, both sides of this equality are elements of $\c[t_{s+1}, \ldots, t_r]$ by Lemma \ref{prop, lowest}.
\end{lem}

\begin{proof}
We may assume that $s = 1$. Define an injection \[\tilde{\iota}_{1, 2}: V(m_2 \varpi_{i_2}) \otimes \cdots \otimes V(m_r \varpi_{i_r}) \hookrightarrow V(m_1 \varpi_{i_1}) \otimes (V(m_2 \varpi_{i_2}) \otimes \cdots \otimes V(m_r \varpi_{i_r}))\] by $\tilde{\iota}_{1, 2}(v) := v_{m_1 \varpi_{i_1}} \otimes v$ for $v \in V(m_2 \varpi_{i_2}) \otimes \cdots \otimes V(m_r \varpi_{i_r})$; note that the injection $\iota_{1, 2}$ is the restriction of $\tilde{\iota}_{1, 2}$. If we take $\tilde{\sigma} \in (V(m_1 \varpi_{i_1}) \otimes \cdots \otimes V(m_r \varpi_{i_r}))^\ast$ such that $\Phi_{{\bf i}, {\bf m}} ^\ast (\tilde{\sigma}) = \sigma$, then we deduce that 
\begin{align*}
((\sigma/\tau_{{\bf i}, {\bf m}})|_{t_1 = 0})(u_2, \ldots, u_r) &= (\sigma/\tau_{{\bf i}, {\bf m}})(1, u_2, \ldots, u_r)\quad({\rm since}\ \exp(0 \cdot F_{i_1}) = 1)\\
&= (\Phi_{{\bf i}, {\bf m}} ^\ast (\tilde{\sigma})/\Phi_{{\bf i}, {\bf m}} ^\ast (\tilde{\tau}_{{\bf i}, {\bf m}}))(1, u_2, \ldots, u_r)\\
&= (\Psi_{{\bf i}, {\bf m}} ^\ast (\tilde{\sigma})/\Psi_{{\bf i}, {\bf m}} ^\ast (\tilde{\tau}_{{\bf i}, {\bf m}}))(1, u_2, \ldots, u_r)\\
&= \tilde{\sigma}(v_{m_1 \varpi_{i_1}} \otimes v_{u_2, \ldots, u_r})/\tilde{\tau}_{{\bf i}, {\bf m}}(v_{m_1 \varpi_{i_1}} \otimes v_{u_2, \ldots, u_r})\\
&({\rm since}\ \Psi_{{\bf i}, {\bf m}} ((1, u_2, \ldots, u_r) \bmod B^r) = \c (v_{m_1 \varpi_{i_1}} \otimes v_{u_2, \ldots, u_r}))\\
&= \tilde{\sigma}(v_{m_1 \varpi_{i_1}} \otimes v_{u_2, \ldots, u_r})\\ 
&({\rm since}\ \tilde{\tau}_{{\bf i}, {\bf m}}(v_{m_1 \varpi_{i_1}} \otimes v_{u_2, \ldots, u_r}) = 1\ {\rm by\ the\ proof\ of\ Lemma}\ \ref{prop, lowest})
\end{align*}
for $u_2 \in U_{i_2} ^-, \ldots, u_r \in U_{i_r} ^-$, where we set \[v_{u_2, \ldots, u_r} := u_2 v_{m_2 \varpi_{i_2}} \otimes u_2 u_3 v_{m_3 \varpi_{i_3}} \otimes \cdots \otimes u_2 u_3 \cdots u_r v_{m_r \varpi_{i_r}}.\] Also, the equality $\Phi_{{\bf i}, {\bf m}} \circ \iota_{1, 2} = \tilde{\iota}_{1, 2} \circ \Phi_{{\bf i}_{\ge 2}, {\bf m}_{\ge 2}}$ implies that \[\iota_{1, 2} ^\ast (\sigma) = \iota_{1, 2} ^\ast \circ \Phi_{{\bf i}, {\bf m}} ^\ast (\tilde{\sigma}) = (\Phi_{{\bf i}, {\bf m}} \circ \iota_{1, 2})^\ast (\tilde{\sigma}) = (\tilde{\iota}_{1, 2} \circ \Phi_{{\bf i}_{\ge 2}, {\bf m}_{\ge 2}})^\ast (\tilde{\sigma}) = \Phi_{{\bf i}_{\ge 2}, {\bf m}_{\ge 2}} ^\ast \circ \tilde{\iota}_{1, 2} ^\ast (\tilde{\sigma}).\] Therefore, it follows that
\begin{align*}
(\iota_{1, 2} ^\ast(\sigma)/\tau_{{\bf i}_{\ge 2}, {\bf m}_{\ge 2}})(u_2, \ldots, u_r) &=  (\Phi_{{\bf i}_{\ge 2}, {\bf m}_{\ge 2}} ^\ast \circ \tilde{\iota}_{1, 2} ^\ast (\tilde{\sigma})/\Phi_{{\bf i}_{\ge 2}, {\bf m}_{\ge 2}} ^\ast (\tilde{\tau}_{{\bf i}_{\ge 2}, {\bf m}_{\ge 2}}))(u_2, \ldots, u_r)\\
&=  (\Psi_{{\bf i}_{\ge 2}, {\bf m}_{\ge 2}} ^\ast \circ \tilde{\iota}_{1, 2} ^\ast (\tilde{\sigma})/\Psi_{{\bf i}_{\ge 2}, {\bf m}_{\ge 2}} ^\ast (\tilde{\tau}_{{\bf i}_{\ge 2}, {\bf m}_{\ge 2}}))(u_2, \ldots, u_r)\\
&= \tilde{\iota}_{1, 2} ^\ast (\tilde{\sigma})(v_{u_2, \ldots, u_r})/\tilde{\tau}_{{\bf i}_{\ge 2}, {\bf m}_{\ge 2}}(v_{u_2, \ldots, u_r})\\ 
&({\rm since}\ \Psi_{{\bf i}_{\ge 2}, {\bf m}_{\ge 2}} ((u_2, \ldots, u_r) \bmod B^{r-1}) = \c v_{u_2, \ldots, u_r})\\
&= \tilde{\sigma}(v_{m_1 \varpi_{i_1}} \otimes v_{u_2, \ldots, u_r})\\
&({\rm since}\ \tilde{\iota}_{1, 2}(v_{u_2, \ldots, u_r}) = v_{m_1 \varpi_{i_1}} \otimes v_{u_2, \ldots, u_r}\ {\rm and}\ \tilde{\tau}_{{\bf i}_{\ge 2}, {\bf m}_{\ge 2}}(v_{u_2, \ldots, u_r}) = 1).
\end{align*}
From these, the assertion of the lemma follows immediately.
\end{proof}

\begin{prop}\label{val,Chevalley}
For $\sigma \in H^0(Z_{\bf i}, \mathcal{L}_{{\bf i}, {\bf m}})$, write $v_{\bf i} (\sigma/\tau_{{\bf i}, {\bf m}}) = -(a_1, \ldots, a_r)$. Then,
\begin{align*}
&a_1 = \max\{a \in \z_{\ge 0} \mid F_{i_1} ^a \sigma \neq 0\},\\
&a_2 = \max\{a \in \z_{\ge 0} \mid F_{i_2} ^a (\iota_{1, 2} ^\ast(F_{i_1} ^{a_1} \sigma)) \neq 0\},\\
&\ \vdots\\
&a_r = \max\{a \in \z_{\ge 0} \mid F_{i_r} ^a (\iota_{r-1, r} ^\ast(F_{i_{r-1}} ^{a_{r-1}}(\cdots(\iota_{2, 3} ^\ast(F_{i_2} ^{a_2}(\iota_{1, 2} ^\ast(F_{i_1} ^{a_1} \sigma))))\cdots))) \neq 0\}.
\end{align*}
\end{prop} 
 
\begin{proof}
Consider the left action of the opposite root subgroup $U_{i_k} ^-$ on $U_{i_k} ^- \times \cdots \times U_{i_r} ^-$ given by \[u \cdot (u_k, \ldots, u_r) := (u u_k, u_{k+1}, \ldots, u_r)\] for $u, u_k \in U_{i_k} ^-$, $u_{k+1} \in U_{i_{k+1}} ^-, \ldots, u_r \in U_{i_r} ^-$; this induces left actions of $U_{i_k} ^-$ and ${\rm Lie}(U_{i_k} ^-) = \c F_{i_k}$ on $\c[t_k, \ldots, t_r]$ $(= \c[U_{i_k} ^- \times \cdots \times U_{i_r} ^-])$, which are given by: 
\begin{align*}
&\exp(t F_{i_k}) \cdot f(t_k, \ldots, t_r) = f(t_k - t, \ldots, t_r),\ {\rm and\ hence}\\
&F_{i_k} \cdot f(t_k, \ldots, t_r) = -\frac{\partial}{\partial t_k} f(t_k, \ldots, t_r)
\end{align*} 
for $t \in \c$ and $f(t_k, \ldots, t_r) \in \c[t_k, \ldots, t_r]$. Also, it follows from the definition of $v_{\bf i}$ that $a_1$ is equal to the degree of $\sigma/\tau_{{\bf i}, {\bf m}}$ with respect to the variable $t_1$; here, Lemma \ref{prop, lowest} implies that $\sigma/\tau_{{\bf i}, {\bf m}} \in \c[t_1, \ldots, t_r]$. Therefore, we deduce that 
\begin{align*}
a_1 &= \max\{a \in \z_{\ge 0} \mid \frac{\partial^a}{\partial t_1 ^a} (\sigma/\tau_{{\bf i}, {\bf m}}) \neq 0\}\\
&= \max\{a \in \z_{\ge 0} \mid F_{i_1} ^a (\sigma/\tau_{{\bf i}, {\bf m}}) \neq 0\}.
\end{align*}
Moreover, because the section $\sigma$ is identical to $(\sigma/\tau_{{\bf i}, {\bf m}}) \cdot \tau_{{\bf i}, {\bf m}}$ in $H^0 (Z_{\bf i}, \mathcal{L}_{{\bf i}, {\bf m}})$, we have
\begin{equation}\label{Leibniz}
\begin{aligned}
F_{i_1} ^a \sigma &= F_{i_1} ^a ((\sigma/\tau_{{\bf i}, {\bf m}}) \cdot \tau_{{\bf i}, {\bf m}})\\
&= (F_{i_1} ^a (\sigma/\tau_{{\bf i}, {\bf m}})) \cdot \tau_{{\bf i}, {\bf m}}\quad({\rm since}\ F_{i_1}\tau_{{\bf i}, {\bf m}} = 0\ {\rm by\ the\ definition\ of}\ \tau_{{\bf i}, {\bf m}}).
\end{aligned}
\end{equation} 
From these, we deduce that \[a_1 = \max\{a \in \z_{\ge 0} \mid F_{i_1} ^a \sigma \neq 0\}.\] Also, because the polynomial $F_{i_1} ^{a_1} (\sigma/\tau_{{\bf i}, {\bf m}}) \in \c[t_1, \ldots, t_r]$ does not contain the variable $t_1$, the restriction $(F_{i_1} ^{a_1} (\sigma/\tau_{{\bf i}, {\bf m}}))|_{t_1 = 0}$ is identical to $F_{i_1} ^{a_1} (\sigma/\tau_{{\bf i}, {\bf m}}) \in \c[t_2, \ldots t_r]$ as a polynomial in the variables $t_2, \ldots, t_r$. Therefore, if $v_{{\bf i}_{\ge 2}}$ denotes the valuation on $\c(Z_{{\bf i}_{\ge 2}})$ defined to be the highest term valuation on $\c(t_2, \ldots, t_r)$ with respect to the lexicographic order on $\z^{r-1}$, then we see from the definition of $v_{\bf i}$ that \[v_{{\bf i}_{\ge 2}}((F_{i_1} ^{a_1} (\sigma/\tau_{{\bf i}, {\bf m}}))|_{t_1 = 0}) = -(a_2, \ldots, a_r).\] Moreover, we see that 
\begin{align*}
(F_{i_1} ^{a_1} (\sigma/\tau_{{\bf i}, {\bf m}}))|_{t_1 = 0} &= ((F_{i_1} ^{a_1}\sigma)/\tau_{{\bf i}, {\bf m}})|_{t_1 = 0}\quad({\rm by\ equation}\ (\ref{Leibniz}))\\
&= \iota_{1, 2} ^\ast(F_{i_1} ^{a_1}\sigma)/\tau_{{\bf i}_{\ge 2}, {\bf m}_{\ge 2}}\quad({\rm by\ Lemma}\ \ref{restriction}).
\end{align*} 
From these, we deduce that $v_{{\bf i}_{\ge 2}} (\iota_{1, 2} ^\ast(F_{i_1} ^{a_1}\sigma)/\tau_{{\bf i}_{\ge 2}, {\bf m}_{\ge 2}}) = -(a_2, \ldots, a_r)$. Repeating this argument, with $\sigma$ replaced by $\iota_{1, 2} ^\ast(F_{i_1} ^{a_1}\sigma)$, we conclude the assertion of the proposition.
\end{proof}

Remark that if we set $F_i ^{(a)} := F_i ^a/a!$ for $i \in I$ and $a \in \z_{\ge 0}$, then \[a_k = \max\{a \in \z_{\ge 0} \mid F_{i_k} ^{(a)} (\iota_{k-1, k} ^\ast(F_{i_{k-1}} ^{(a_{k-1})}(\cdots(\iota_{2, 3} ^\ast(F_{i_2} ^{(a_2)}(\iota_{1, 2} ^\ast(F_{i_1} ^{(a_1)} \sigma))))\cdots))) \neq 0\}\] for all $k = 1, \ldots, r$.

\section{Upper crystal bases and upper global bases}

In this section, we recall some basic facts about upper crystal bases and upper global bases, following \cite{Kas1}, \cite{Kas2}, and \cite{Kas3}. We denote by $\langle \cdot, \cdot \rangle: \mathfrak{t}^\ast \times \mathfrak{t} \rightarrow \c$ the canonical pairing, and define a symmetric bilinear form $(\cdot, \cdot)$ on $\mathfrak{t}^\ast$ by $2(\alpha_j, \alpha_i)/(\alpha_i, \alpha_i) = \langle \alpha_j, h_i \rangle$ for all $i, j \in I$, and by $(\alpha_i, \alpha_i) = 2$ for all short simple roots $\alpha_i$. We set
\begin{align*}
&q_i := q^{(\a_i, \a_i)/2}\ {\rm for}\ i \in I,\\
&[s]_i := \frac{q_i ^s - q_i ^{-s}}{q_i - q_i ^{-1}}\ {\rm for}\ i \in I,\ s \in \z,\\
&[0]_i ! := 1,\ {\rm and}\ [s]_i ! := [s]_i [s-1]_i \cdots [1]_i\ {\rm for}\ i \in I,\ s \in \z_{> 0},\\
&\genfrac{[}{]}{0pt}{}{s}{0}_i := 1\ {\rm for}\ s \in \z_{\ge 0},\ {\rm and}\ \genfrac{[}{]}{0pt}{}{s}{k}_i := \frac{[s]_i [s-1]_i \cdots [s - k +1]_i}{[k]_i [k-1]_i \cdots [1]_i}\ {\rm for}\ s, k \in \z_{> 0}\ {\rm with}\ k \le s.
\end{align*}

\vspace{2mm}\begin{defi}\normalfont
For a finite-dimensional semisimple Lie algebra $\mathfrak{g}$, the {\it quantized enveloping algebra} $U_q(\mathfrak{g})$ is the unital associative $\q(q)$-algebra with generators $e_i, f_i, t_i, t_i ^{-1}$, $i \in I$, and relations: for $i, j \in I$,
\begin{enumerate}
\item[(i)] $t_i t_i ^{-1}=1$ and $t_i t_j = t_j t_i$,
\item[(ii)] $t_i e_j t_i ^{-1} = q_i ^{c_{i, j}} e_j$ and $t_i f_j t_i ^{-1} = q_i ^{-c_{i, j}} f_j$,
\item[(iii)] $e_i f_i - f_i e_i = (t_i - t_i ^{-1})/(q_i - q_i ^{-1})$ and $e_i f_j - f_j e_i = 0$ if $i \neq j$,
\item[(iv)] $\sum_{s = 0} ^{1 - c_{i, j}}(-1)^s e_i^{(s)} e_j e_i ^{(1 - c_{i, j} - s)} = \sum_{s = 0} ^{1 - c_{i, j}}(-1)^s f_i^{(s)} f_j f_i ^{(1 - c_{i, j} - s)} = 0$ if $i \neq j$.
\end{enumerate}
Here, $(c_{i, j})_{i, j \in I} := (\langle \alpha_j, h_i \rangle)_{i, j \in I}$ is the Cartan matrix of $\mathfrak{g}$, and $e_i ^{(s)} := e_i ^s/[s]_i !$, $f_i ^{(s)} := f_i ^s/[s]_i !$ for $i \in I$, $s \in \z_{\ge 0}$. 
\end{defi}\vspace{2mm}

The algebra $U_q(\mathfrak{g})$ has the Hopf algebra structure given by the following coproduct $\Delta$, counit $\varepsilon$, and antipode $S$:
\begin{align*}
&\Delta (e_i) = e_i \otimes 1 + t_i \otimes e_i,\ \Delta (f_i) = f_i \otimes t_i ^{-1} + 1 \otimes f_i,\ \Delta (t_i) = t_i \otimes t_i,\\
&\varepsilon (e_i) = 0,\ \varepsilon (f_i) = 0,\ \varepsilon (t_i) = 1,\\
&S (e_i) = -t_i ^{-1} e_i,\ S (f_i) = -f_i t_i,\ S (t_i) = t_i^{-1}
\end{align*}
for $i \in I$. The coproduct $\Delta$ is identical to $\Delta_+$ in \cite{Kas2}, and to $\Delta$ in \cite{Lus}. In this paper, we always assume that $U_q(\mathfrak{g})$-modules are defined over $\q(q)$. For $\lambda \in P$ and a finite-dimensional $U_q(\mathfrak{g})$-module $V$, let $V_\lambda$ denote the corresponding weight space, i.e., \[V_\lambda:= \{v \in V \mid t_i v = q_i ^{\langle \lambda, h_i\rangle} v\ {\rm for\ all}\ i \in I\}.\] Also, we define operators $\tilde{e}_i$, $\tilde{f}_i$, $i \in I$, on $V$ as follows (see \cite[\S\S 3.1]{Kas3}): for $v \in\ {\rm Ker}\ e_i \cap V_\lambda$ and $0 \le k \le \langle \lambda, h_i \rangle$, 
\[\tilde{e}_i (f_i ^{(k)}v) := \frac{[\langle \lambda, h_i \rangle - k + 1]_i}{[k]_i} f_i ^{(k-1)}v,\ \tilde{f}_i (f_i ^{(k)}v) := \frac{[k + 1]_i}{[\langle \lambda, h_i \rangle - k]_i} f_i ^{(k+1)}v;\] here, we set $f_i ^{(-1)}v := 0$. These operators $\tilde{e}_i$, $\tilde{f}_i$, $i \in I$, are called the {\it upper Kashiwara operators}.

\vspace{2mm}\begin{defi}\normalfont
Denote by $A \subset \q(q)$ the $\q$-subalgebra of $\q(q)$ consisting of rational functions regular at $q = 0$. For a finite-dimensional $U_q(\mathfrak{g})$-module $V$, an {\it upper crystal basis} $(L, \mathcal{B})$ of $V$ is a pair of a free $A$-submodule $L \subset V$ and a $\q$ $(= A/qA)$ -basis $\mathcal{B}$ of $L/qL$ satisfying the following conditions:
\begin{enumerate}
\item[{\rm (i)}] $V \simeq \q(q) \otimes_A L$,
\item[{\rm (ii)}] $\tilde{e}_i L \subset L$ and $\tilde{f}_i L \subset L$ for $i \in I$ (in particular, $\tilde{e}_i$ and $\tilde{f}_i$ act on $L/qL$),
\item[{\rm (iii)}] $\tilde{e}_i \mathcal{B} \subset \mathcal{B} \cup \{0\}$ and $\tilde{f}_i \mathcal{B} \subset \mathcal{B} \cup \{0\}$ for $i \in I$,
\item[{\rm (iv)}] $L = \bigoplus_{\lambda \in P} L_\lambda$ and $\mathcal{B} = \coprod_{\lambda \in P} \mathcal{B}_\lambda$, where $L_\lambda := L \cap V_\lambda$ and $\mathcal{B}_\lambda := \mathcal{B} \cap (L_\lambda/q L_\lambda)$,
\item[{\rm (v)}] $b^\prime = \tilde{f}_i b$ if and only if $b = \tilde{e}_i b^\prime$ for $i \in I$ and $b, b^\prime \in \mathcal{B}$.
\end{enumerate}
\end{defi}\vspace{2mm}

The following is a fundamental property of an upper crystal basis.

\vspace{2mm}\begin{lem}[{see, for instance, \cite[Lemma 2.4.1 and equation (2.4.2)]{Kas2}}]\label{length of string}
Let $V$ be a finite-dimensional $U_q(\mathfrak{g})$-module, $(L, \mathcal{B})$ its upper crystal basis, and set \[\varepsilon_i(b) := \max\{a \in \z_{\ge 0} \mid \tilde{e}_i ^a b \neq 0\},\ \varphi_i (b) := \max\{a \in \z_{\ge 0} \mid \tilde{f}_i ^a b \neq 0\}\] for $i \in I$ and $b \in \mathcal{B}$. Then, it holds that \[\langle {\rm wt}(b), h_i \rangle = \varphi_i(b) - \varepsilon_i(b).\]
\end{lem}\vspace{2mm}

For a dominant integral weight $\lambda$, let $V_q(\lambda)$ denote the irreducible highest weight $U_q(\mathfrak{g})$-module with highest weight $\lambda$, and $v_{q, \lambda} \in V_q(\lambda)$ the highest weight vector. We define an $A$-submodule $L^{\rm up}(\lambda) \subset V_q(\lambda)$ and a subset $\mathcal{B}(\lambda) \subset L^{\rm up}(\lambda)/q L^{\rm up}(\lambda)$ by
\begin{align*}
L^{\rm up}(\lambda) &:= \sum_{\substack{l \in \z_{\ge 0},\\i_1, \ldots, i_l \in I}} A \tilde{f}_{i_1} \cdots \tilde{f}_{i_l} v_{q, \lambda},\\
\mathcal{B}(\lambda) &:= \{\tilde{f}_{i_1} \cdots \tilde{f}_{i_l} v_{q, \lambda} \bmod q L^{\rm up}(\lambda) \mid l \in \z_{\ge 0},\ i_1, \ldots, i_l \in I\} \setminus \{0\}.
\end{align*}
Then, it follows from \cite[Theorem 2]{Kas2} and \cite[Proposition 3.2.2]{Kas3} that $(L^{\rm up}(\lambda), \mathcal{B}(\lambda))$ is an upper crystal basis of $V_q(\lambda)$.

\vspace{2mm}\begin{defi}\normalfont
The {\it bar involution} $\overline{\vphantom{(}\cdot\vphantom{)}}: U_q(\mathfrak{g}) \rightarrow U_q(\mathfrak{g})$ is the $\q$-involution given by: \[\overline{e_i} = e_i,\ \overline{f_i} = f_i,\ \overline{t_i} = t_i ^{-1},\ \overline{q} = q^{-1}.\] Also, for a finite-dimensional $U_q(\mathfrak{g})$-module $V$, a $\q$-involution $\overline{\vphantom{(}\cdot\vphantom{)}}: V \rightarrow V$ is called a {\it bar involution} on $V$ if $\overline{u v} = \overline{u} \cdot \overline{v}$ for all $u \in U_q (\mathfrak{g})$ and $v \in V$.
\end{defi}\vspace{2mm}

\noindent Note that there exists a unique bar involution $\overline{\vphantom{(}\cdot\vphantom{)}}$ on $V_q(\lambda)$ such that $\overline{v_{q, \lambda}} = v_{q, \lambda}$. We now recall the definition of upper global bases. Let $V$ be a vector space over $\q(q)$, and $V^\q \subset V$ its $\q[q, q^{-1}]$-submodule. Then, $V^\q$ is called a {\it$\q[q, q^{-1}]$-form} of $V$ if $V \simeq V^\q \otimes_{\q[q, q^{-1}]} \q(q)$. Denote by $U_q ^\q(\mathfrak{g}) \subset U_q(\mathfrak{g})$ the $\q[q, q^{-1}]$-subalgebra of $U_q(\mathfrak{g})$ generated by $e_i ^{(k)}, f_i ^{(k)}, t_i, t_i ^{-1}$, $i \in I$, $k \in \z_{\ge 0}$, and by \[\genfrac{\{}{\}}{0pt}{}{t}{l}_i := \prod_{k=1} ^l \frac{q_i ^{1-k} t - q_i ^{k-1} t^{-1}}{q_i ^k - q_i ^{-k}},\ i \in I,\ l \in \z_{>0},\] for Laurent monomials $t$ in the variables $t_j$, $j \in I$. Also, let $V_{q, \q} ^{\rm up} (\lambda) \subset V_q(\lambda)$ denote the unique $\q[q, q^{-1}]$-submodule of $V_q (\lambda)$ satisfying the following conditions (see \cite[\S\S 4.2]{Kas3}):  
\begin{align*}
&V_{q, \q} ^{\rm up} (\lambda) \cap V_q (\lambda)_\lambda = \q[q, q^{-1}] v_{q, \lambda};\\ 
&\{v \in V_q (\lambda) \mid e_i ^{(k)} v \in V_{q, \q} ^{\rm up} (\lambda)\ {\rm for\ all}\ i \in I,\ k \ge 1\} = V_{q, \q} ^{\rm up} (\lambda) + \q(q) v_{q, \lambda}.
\end{align*}
Note that $U_q ^\q(\mathfrak{g})$ (resp., $V_{q, \q} ^{\rm up} (\lambda)$) is a $\q[q, q^{-1}]$-form of $U_q(\mathfrak{g})$ (resp., of $V_q(\lambda)$), and that $V_{q, \q} ^{\rm up} (\lambda)$ is a $U_q ^\q(\mathfrak{g})$-submodule. We regard $\c$ as a $\q[q, q^{-1}]$-module by the natural $\q$-algebra homomorphism $\q[q, q^{-1}] \rightarrow \c$, $q \mapsto 1$. Let $V$ be a finite-dimensional $U_q(\mathfrak{g})$-module, and $V^\q$ a $\q[q, q^{-1}]$-form of $V$ that is invariant under the action of $U_q ^\q(\mathfrak{g})$. Then, the $\c$-vector space $V^\q \otimes_{\q[q, q^{-1}]} \c$ has a $\mathfrak{g}$-module structure given by \[E_i(v \otimes c) := (e_i v) \otimes c,\ F_i(v \otimes c) := (f_i v) \otimes c,\ h_i(v \otimes c) := \left(\frac{t_i - t_i ^{-1}}{q_i - q_i ^{-1}}v\right) \otimes c\] for $i \in I$, $v \in V^\q$, and $c \in \c$. Note that $V_{q, \q} ^{\rm up} (\lambda) \otimes_{\q[q, q^{-1}]} \c$ is isomorphic to $V(\lambda)$ as a $\mathfrak{g}$-module (see, for instance, the proof of \cite[Lemma 5.14]{J2}).

\vspace{2mm}\begin{rem}\normalfont
The $\q[q, q^{-1}]$-form of $V_q(\lambda)$ used in \cite[Lemma 5.14]{J2} is not identical to our $\q[q, q^{-1}]$-form $V_{q, \q} ^{\rm up} (\lambda)$; note that these are dual to each other. However, the proof of \cite[Lemma 5.14]{J2} can also be applied to our $\q[q, q^{-1}]$-form $V_{q, \q} ^{\rm up} (\lambda)$.
\end{rem}

\vspace{2mm}\begin{defi}\normalfont
Let $V$ be a finite-dimensional $U_q(\mathfrak{g})$-module, $(L, \mathcal{B})$ its upper crystal basis, $\overline{\vphantom{(}\cdot\vphantom{)}}$ a bar involution on $V$, and $V^\q$ a $\q[q, q^{-1}]$-form of $V$ that is invariant under the action of $U_q ^\q(\mathfrak{g})$. Then, $(V^\q, L, \overline{L})$ is called a {\it balanced triple} if the natural $\q$-linear map $V^\q \cap L \cap \overline{L} \rightarrow L/q L$ is an isomorphism. If $G_q ^{\rm up}:L/q L \rightarrow V^\q \cap L \cap \overline{L}$ denotes the inverse map of this isomorphism, then $\{G_q ^{\rm up}(b) \mid b \in \mathcal{B}\}$ forms a $\q[q, q^{-1}]$-basis of $V^\q$; this is called the {\it upper global basis} of $V$ with respect to the balanced triple $(V^\q, L, \overline{L})$.
\end{defi}\vspace{2mm}

If we set $G^{\rm up}(b):= G_q ^{\rm up}(b) \otimes 1 \in V^\q \otimes_{\q[q, q^{-1}]} \c$ for $b \in \mathcal{B}$, the specialization of $G_q ^{\rm up}(b)$ at $q = 1$, then the set $\{G^{\rm up}(b) \mid b \in \mathcal{B}\}$ forms a $\c$-basis of the $\mathfrak{g}$-module $V^\q \otimes_{\q[q, q^{-1}]} \c$. The following is a fundamental property of an upper global basis.

\vspace{2mm}\begin{prop}[{see \cite[Proposition 5.3.1 and the remark following it]{Kas3}}]\label{property of upper global basis}
Let $V$ be a finite-dimensional $U_q(\mathfrak{g})$-module, $\{G_q ^{\rm up}(b) \mid b \in \mathcal{B}\}$ its upper global basis, and $\varepsilon_i$, $\varphi_i$ the maps defined in Lemma \ref{length of string}.
\begin{enumerate}
\item[{\rm (1)}] For all $i \in I$, $k \in \z_{\ge 0}$, and $b \in \mathcal{B}$, it holds that
\begin{align*}
&e_i ^{(k)} G^{\rm up} _q (b) \in \genfrac{[}{]}{0pt}{}{\varepsilon_i (b)}{k}_i G^{\rm up} _q (\tilde{e}_i ^k b) + \sum_{\substack{b^\prime \in \mathcal{B};\ {\rm wt}(b^\prime) = {\rm wt}(\tilde{e}_i ^k b),\\ \varepsilon_i (b^\prime) < \varepsilon_i (\tilde{e}_i ^k b)}} \z[q, q^{-1}] G^{\rm up} _q (b^\prime), \\
&f_i ^{(k)} G^{\rm up} _q (b) \in \genfrac{[}{]}{0pt}{}{\varphi_i (b)}{k}_i G^{\rm up} _q (\tilde{f}_i ^k b) + \sum_{\substack{b^\prime \in \mathcal{B};\ {\rm wt}(b^\prime) = {\rm wt}(\tilde{f}_i ^k b),\\ \varphi_i (b^\prime) < \varphi_i (\tilde{f}_i ^k b)}} \z[q, q^{-1}] G^{\rm up} _q (b^\prime).
\end{align*} 
\item[{\rm (2)}] Set $E_i ^{(k)} := E_i ^k/k!$ and $F_i ^{(k)} := F_i ^k/k!$ for $k \in \z_{\ge0}$. Then,
\begin{align*}
&E_i ^{(k)} G^{\rm up}(b) \in \genfrac{(}{)}{0pt}{}{\varepsilon_i (b)}{k} G^{\rm up}(\tilde{e}_i ^k b) + \sum_{\substack{b^\prime \in \mathcal{B};\ {\rm wt}(b^\prime) = {\rm wt}(\tilde{e}_i ^k b),\\ \varepsilon_i (b^\prime) < \varepsilon_i (\tilde{e}_i ^k b)}} \z G^{\rm up}(b^\prime), \\
&F_i ^{(k)} G^{\rm up}(b) \in \genfrac{(}{)}{0pt}{}{\varphi_i (b)}{k} G^{\rm up}(\tilde{f}_i ^k b) + \sum_{\substack{b^\prime \in \mathcal{B};\ {\rm wt}(b^\prime) = {\rm wt}(\tilde{f}_i ^k b),\\ \varphi_i (b^\prime) < \varphi_i (\tilde{f}_i ^k b)}} \z G^{\rm up}(b^\prime)
\end{align*}
for all $i \in I$, $k \in \z_{\ge 0}$, and $b \in \mathcal{B}$. Here, $\genfrac{(}{)}{0pt}{}{\varepsilon_i (b)}{k}, \genfrac{(}{)}{0pt}{}{\varphi_i (b)}{k}$ are the usual binomial coefficients. In particular, it holds that 
\begin{align*}
&E_i ^{(\varepsilon_i (b))} G^{\rm up} (b) = G^{\rm up} (\tilde{e}_i ^{\varepsilon_i (b)} b),\ \varepsilon_i (b) = \max\{k \in \z_{\ge 0} \mid E_i ^{(k)} G^{\rm up} (b) \neq 0\},\ {\it and}\\ 
&F_i ^{(\varphi_i (b))} G^{\rm up} (b) = G^{\rm up} (\tilde{f}_i ^{\varphi_i (b)} b),\ \varphi_i (b) = \max\{k \in \z_{\ge 0} \mid F_i ^{(k)} G^{\rm up} (b) \neq 0\}.
\end{align*}
\end{enumerate}
\end{prop}\vspace{2mm}

It follows from \cite[Lemma 4.2.1]{Kas3} that $(V_{q, \q} ^{\rm up} (\lambda), L^{\rm up}(\lambda), \overline{L^{\rm up}(\lambda)})$ is a balanced triple; hence we obtain an upper global basis $\{G_q ^{\rm up} (b) \mid b \in \mathcal{B}(\lambda)\} \subset V_q(\lambda)$.  

\section{String polytopes for Demazure modules}

Here we recall the main result of \cite{Kav}. Let us assume that ${\bf i} = (i_1, \ldots, i_r) \in I^r$ is a reduced word for $w \in W$. 

\vspace{2mm}\begin{defi}\normalfont
Let us denote by $X(w)$ for $w \in W$ the Zariski closure of $B \widetilde{w} B/B$ in $G/B$, where $\widetilde{w} \in G$ denotes a lift for $w$; note that the closed subvariety $X(w)$ is independent of the choice of a lift $\widetilde{w}$. The $X(w)$ is called the {\it Schubert variety} corresponding to $w \in W$.
\end{defi}\vspace{2mm}

It is well-known that the Schubert variety $X(w)$ is a normal projective variety of complex dimension $r$. For a dominant integral weight $\lambda$, we define a variety $\mathcal{L}_\lambda$ by \[\mathcal{L}_\lambda:= (G \times \c)/B,\] where $B$ acts on $G \times \c$ on the right as follows: $(g, c) \cdot b:= (g b, \lambda(b)c)$ for $g \in G$, $c \in \c$, and $b \in B$. Then, the variety $\mathcal{L}_\lambda$ induces a line bundle (denoted by the same symbol $\mathcal{L}_\lambda$) on $G/B$ by the canonical projection \[\mathcal{L}_\lambda \twoheadrightarrow G/B,\ (g, c) \bmod B \mapsto g \bmod B.\] Let us define left actions of $G$ on $G/B$ and on $\mathcal{L}_\lambda$ by
\begin{align*}
g \cdot (g^\prime \bmod B) &:= g g^\prime \bmod B,\\
g \cdot ((g^\prime, c) \bmod B) &:= (g g^\prime, c) \bmod B
\end{align*}
for $g, g^\prime \in G$ and $c \in \c$. Since the projection $\mathcal{L}_\lambda \twoheadrightarrow G/B$ is compatible with these actions, we know that the space $H^0(G/B, \mathcal{L}_\lambda)$ of global sections has a natural $G$-module structure. Note that $\mathcal{L}_\lambda$ induces a line bundle on $X(w)$, which we denote by the same symbol $\mathcal{L}_\lambda$. Since the Schubert variety $X(w)$ is a $P_{i_1}$-stable subvariety of $G/B$, the space $H^0(X(w), \mathcal{L}_\lambda)$ of global sections has a natural $P_{i_1}$-module structure. 

\vspace{2mm}\begin{defi}\normalfont
For $w \in W$ and a dominant integral weight $\lambda$, let $v_{w\lambda} \in V(\lambda)$ denote the extremal weight vector of weight $w\lambda$. Define a $P_{i_1}$-submodule $V_w(\lambda) \subset V(\lambda)$ by \[V_w(\lambda) := \sum_{b \in B} \c b v_{w\lambda}\ (= \sum_{p \in P_{i_1}} \c p v_{w\lambda});\] this is called the {\it Demazure module} corresponding to $w \in W$.
\end{defi}\vspace{2mm}

From the Borel-Weil theorem, we know that $H^0(G/B, \mathcal{L}_\lambda)$ is isomorphic to $V(\lambda)^\ast$ as a $G$-module, and that $H^0(X(w), \mathcal{L}_\lambda)$ is isomorphic to $V_w(\lambda)^\ast$ as a $P_{i_1}$-module.

\vspace{2mm}\begin{prop}[{see \cite[Chapters 13, 14]{J1}}]\label{resolution}
Let ${\bf i} = (i_1, \ldots, i_r) \in I^r$ be a reduced word for $w \in W$.
\begin{enumerate}
\item[{\rm (1)}] The product map \[\mu_{\bf i}:  Z_{\bf i} \rightarrow G/B,\ (p_1, \ldots, p_r) \bmod B^r \mapsto p_1 \cdots p_r \bmod B,\] induces a birational morphism onto the Schubert variety $X(w)$. 
\item[{\rm (2)}] For a dominant integral weight $\lambda = \sum_{i \in I} \lambda_i \varpi_i$, set $\lambda^\prime:= \sum_{i \in I \setminus \{i_1, \ldots, i_r\}} \lambda_i \varpi_i$, and define ${\bf m} = (m_1, \ldots, m_r) \in \z^r _{\ge 0}$ by \[m_k := 
\begin{cases}
\lambda_{i_k} &{\it if}\ i_{k^\prime} \neq i_k\ {\it for\ all}\ k < k^\prime \le r,\\
0 &{\it otherwise}
\end{cases}\] for $1 \le k \le r$. Then, the birational morphism $\mu_{\bf i}$ induces an isomorphism of $P_{i_1}$-modules$:$ \[H^0(X(w), \mathcal{L}_\lambda) \xrightarrow{\sim} H^0(Z_{\bf i}, \mu_{\bf i} ^\ast (\mathcal{L}_\lambda)) \simeq \c_{\lambda^\prime} ^\ast \otimes H^0(Z_{\bf i}, \mathcal{L}_{{\bf i}, {\bf m}}),\] where $\c_{\lambda^\prime}$ denotes the one-dimensional $B$-module induced by the weight $\lambda^\prime$, which is regarded as a $P_{i_1}$-module by the trivial action of $\exp(\c F_{i_1})$ $(\subset P_{i_1})$.
\end{enumerate}
\end{prop}\vspace{2mm}

Since Proposition \ref{resolution} (1) implies that $\c(X(w)) \simeq \c(Z_{\bf i})$, the valuation $v_{\bf i}$ can be regarded as a valuation on $\c(X(w))$. For a dominant integral weight $\lambda$, let $\tau_\lambda \in H^0(G/B, \mathcal{L}_\lambda) = V(\lambda)^\ast$ denote the lowest weight vector such that $\tau_\lambda(v_\lambda) = 1$. By restricting this section, we obtain a section in $H^0(X(w), \mathcal{L}_\lambda)$, which we denote by the same symbol $\tau_\lambda$. In addition, we take ${\bf m} \in \z_{\ge 0} ^r$ as in Proposition \ref{resolution} (2). Then, the Newton-Okounkov body $\Delta(X(w), \mathcal{L}_\lambda, v_{\bf i}, \tau_\lambda)$ is identical to the Newton-Okounkov body $\Delta(Z_{\bf i}, \mathcal{L}_{{\bf i}, {\bf m}}, v_{\bf i}, \tau_{{\bf i}, {\bf m}})$. The main result of \cite{Kav} states that this Newton-Okounkov body is also identical to the string polytope associated to the Demazure module $V_w(\lambda)$ and the reduced word ${\bf i}$. We now recall the definition of string polytopes for Demazure modules.

\vspace{2mm}\begin{defi}\normalfont
Let ${\bf i} = (i_1, \ldots, i_r) \in I^r$ be a reduced word for $w \in W$, and $\lambda$ a dominant integral weight. For the highest weight element $b_\lambda := v_{q, \lambda} \bmod q L^{\rm up}(\lambda) \in \mathcal{B}(\lambda)$, a subset \[\mathcal{B}_w(\lambda) := \{\tilde{f}_{i_1} ^{a_1} \cdots \tilde{f}_{i_r} ^{a_r} b_\lambda \mid a_1, \ldots, a_r \in \z_{\ge 0}\} \setminus \{0\} \subset \mathcal{B}(\lambda)\] is independent of the choice of a reduced word ${\bf i}$ (see \cite[Proposition 3.2.3]{Kas4}); this subset is called a {\it Demazure crystal}.
\end{defi}\vspace{2mm}

Note that the character of $V_w(\lambda)$ is equal to $\sum_{b \in \mathcal{B}_w(\lambda)} e^{{\rm wt}(b)}$, and that $\tilde{e}_i (\mathcal{B}_w(\lambda)) \subset \mathcal{B}_w(\lambda) \cup \{0\}$ for all $i \in I$ (see \cite[Proposition 3.2.3 (ii)]{Kas4}). 

\vspace{2mm}\begin{defi}\normalfont
Let ${\bf i} = (i_1, \ldots, i_r) \in I^r$ be a reduced word for $w \in W$, and $\lambda$ a dominant integral weight. 
\begin{enumerate}
\item[{\rm (1)}] For $b \in \mathcal{B}_w(\lambda)$, define $\Omega_{\bf i} (b) = (a_1, \ldots, a_r) \in \z_{\ge 0} ^r$ by
\begin{align*}
&a_1 := \max\{a \in \z_{\ge 0} \mid \tilde{e}_{i_1} ^a b \neq 0\},\\
&a_2 := \max\{a \in \z_{\ge 0} \mid \tilde{e}_{i_2} ^a \tilde{e}_{i_1} ^{a_1} b \neq 0\},\\
&\ \vdots\\
&a_r := \max\{a \in \z_{\ge 0} \mid \tilde{e}_{i_r} ^a \tilde{e}_{i_{r-1}} ^{a_{r-1}} \cdots \tilde{e}_{i_1} ^{a_1} b \neq 0\}.
\end{align*}
The $\Omega_{\bf i}(b)$ is called the {\it string parameterization} of $b$ with respect to the reduced word ${\bf i}$. The map $\Omega_{\bf i}: \mathcal{B}_w(\lambda) \rightarrow \z_{\ge 0} ^r$ is indeed an injection.
\item[{\rm (2)}] Define a subset $\mathcal{S}_{\bf i} ^{(\lambda, w)} \subset \z_{>0} \times \z^r$ by \[\mathcal{S}_{\bf i} ^{(\lambda, w)} := \bigcup_{k>0} \{(k, \Omega_{\bf i}(b)) \mid b \in \mathcal{B}_w (k\lambda)\},\] and denote by $\mathcal{C}_{\bf i} ^{(\lambda, w)} \subset \r_{\ge 0} \times \r^r$ the smallest real closed cone containing $\mathcal{S}_{\bf i} ^{(\lambda, w)}$. Also, let us define a subset $\Delta_{\bf i} ^{(\lambda, w)} \subset \r^r$ by \[\Delta_{\bf i} ^{(\lambda, w)} := \{{\bf a} \in \r^r \mid (1, {\bf a}) \in \mathcal{C}_{\bf i} ^{(\lambda, w)}\}.\] This subset $\Delta_{\bf i} ^{(\lambda, w)}$ is called the {\it string polytope} associated to $\mathcal{B}_w(\lambda)$ and ${\bf i}$ (see \cite[Section 1]{Lit} and \cite[Definition 3.5]{Kav}).
\end{enumerate}
\end{defi}\vspace{2mm}

\begin{lem}[{see \cite[Section 1]{Lit}}]
For $(a_1, \ldots, a_r) \in \Delta_{\bf i} ^{(\lambda, w)}$, it holds that \[0 \le a_r \le \langle \lambda, h_{i_r} \rangle, 0 \le a_{r-1} \le \langle \lambda -a_r \alpha_{i_r}, h_{i_{r-1}}\rangle, \ldots, 0 \le a_1 \le \langle\lambda -a_r \alpha_{i_r} - \cdots - a_2 \alpha_{i_2}, h_{i_1}\rangle.\] In particular, the string polytope $\Delta_{\bf i} ^{(\lambda, w)}$ is bounded, and hence compact.
\end{lem}\vspace{2mm}

\begin{prop}[{see \cite[\S\S 3.2 and Theorem 3.10]{BZ}}]
The real closed cone $\mathcal{C}_{\bf i} ^{(\lambda, w)}$ is a rational convex polyhedral cone, that is, there exists a finite number of rational points ${\bf a}_1, \ldots, {\bf a}_l \in \q_{\ge 0} \times \q^r$ such that $\mathcal{C}_{\bf i} ^{(\lambda, w)} = \r_{\ge 0}{\bf a}_1 + \cdots + \r_{\ge 0} {\bf a}_l$. Moreover, the equality $\mathcal{S}_{\bf i} ^{(\lambda, w)} = \mathcal{C}_{\bf i} ^{(\lambda, w)} \cap (\z_{>0} \times \z^r)$ holds. In particular, $\Delta_{\bf i} ^{(\lambda, w)}$ is a rational convex polytope, and the equality $\Omega_{\bf i} (\mathcal{B}_w(\lambda)) = \Delta_{\bf i} ^{(\lambda, w)} \cap \z^r$ holds.
\end{prop}\vspace{2mm}

In order to state the main result of \cite{Kav}, we recall the definition of dual crystals.

\vspace{2mm}\begin{defi}\normalfont
Let $V$ be a finite-dimensional $U_q(\mathfrak{g})$-module, and $(L, \mathcal{B})$ its upper crystal basis.
\begin{enumerate}
\item[{\rm (1)}] The {\it crystal graph} of $\mathcal{B}$ is the $I$-colored, directed graph with vertex set $\mathcal{B}$ whose directed edges are given by: $b \xrightarrow{i} b^\prime$ if and only if $b^\prime = \tilde{f}_i b$.
\item[{\rm (2)}] The {\it dual crystal} $\mathcal{B}^\ast$ is the $I$-colored, directed graph that is obtained from the crystal graph of $\mathcal{B}$ by reversing all directed edges. For $b \in \mathcal{B}$, denote by $b^\ast$ the vertex of $\mathcal{B}^\ast$ corresponding to $b$.
\end{enumerate}
\end{defi}\vspace{2mm}
\noindent
In this paper, we identify $\mathcal{B}$ with its crystal graph. If we have the irreducible decomposition $V = V_q(\lambda_1) \oplus \cdots \oplus V_q(\lambda_l)$, then the crystal graph of $\mathcal{B}$ is the disjoint union of the crystal graphs $\mathcal{B}(\lambda_1), \ldots, \mathcal{B}(\lambda_l)$. Also, if we write $\lambda^\ast := -w_0 \lambda \in P$ for a dominant integral weight $\lambda$, then the dual crystal $\mathcal{B}(\lambda)^\ast$ is identical to the crystal graph of $\mathcal{B}(\lambda^\ast)$. Now, for $w \in W$, let $\Phi_{\lambda, w}: V_w(\lambda) \hookrightarrow V(\lambda)$ denote the inclusion map, and $\Phi_{\lambda, w} ^\ast: V(\lambda)^\ast \twoheadrightarrow V_w(\lambda)^\ast = H^0(X(w), \mathcal{L}_\lambda)$ the dual map. If we set $G^{\rm up} _{\lambda, w} (b) := \Phi_{\lambda, w} ^\ast(G^{\rm up} (b^\ast))$ for $b \in \mathcal{B}_w (\lambda)$, then the set $\{G^{\rm up} _{\lambda, w} (b) \mid b \in \mathcal{B}_w (\lambda)\}$ forms a $\c$-basis of $H^0(X(w), \mathcal{L}_\lambda)$. The following is the main result of \cite{Kav}.

\vspace{2mm}\begin{prop}[{see \cite[Theorem 4.1, Corollary 4.2, and Remark 4.6]{Kav}}]\label{Kaveh}
Let ${\bf i} \in I^r$ be a reduced word for $w \in W$, and $\lambda$ a dominant integral weight. 
\begin{enumerate}
\item[{\rm (1)}] $\Omega_{\bf i} (b) = - v_{\bf i} (G^{\rm up} _{\lambda, w} (b)/\tau_\lambda)$ for all $b \in \mathcal{B}_w (\lambda)$.
\item[{\rm (2)}] Define the linear automorphism $\omega: \r \times \r^r \xrightarrow{\sim} \r \times \r^r$ by $\omega(k, {\bf a}) = (k, -{\bf a})$. Then, $\mathcal{S}_{\bf i} ^{(\lambda, w)} = \omega(S(X(w), \mathcal{L}_\lambda, v_{\bf i}, \tau_\lambda))$, $\mathcal{C}_{\bf i} ^{(\lambda, w)} = \omega(C(X(w), \mathcal{L}_\lambda, v_{\bf i}, \tau_\lambda))$, and $\Delta_{\bf i} ^{(\lambda, w)} = -\Delta(X(w), \mathcal{L}_\lambda, v_{\bf i}, \tau_\lambda)$.
\end{enumerate}
\end{prop}\vspace{2mm}

In the rest of this paper, we extend this result to $\Delta(Z_{\bf i}, \mathcal{L}_{{\bf i}, {\bf m}}, v_{\bf i}, \tau_{{\bf i}, {\bf m}})$ for an arbitrary ${\bf i} \in I^r$ and ${\bf m} \in \z_{\ge 0} ^r$.
\section{String polytopes for generalized Demazure modules}
\subsection{Basic definitions}
In this subsection, we introduce a generalization of string polytope. First, we recall some basic facts about the tensor product of upper crystal bases.

\vspace{2mm}\begin{prop}\label{tensor product of crystals}
Let $V_1, V_2$ be finite-dimensional $U_q(\mathfrak{g})$-modules, and $(L_1, \mathcal{B}_1), (L_2, \mathcal{B}_2)$ upper crystal bases of $V_1, V_2$, respectively.
\begin{enumerate}
\item[{\rm (1)}] For a $\q$-basis \[\mathcal{B}_1 \otimes \mathcal{B}_2 := \{b_1 \otimes b_2 \mid b_1 \in \mathcal{B}_1,\ b_2 \in \mathcal{B}_2\}\] of $L_1 /q L_1 \otimes L_2 /q L_2 \simeq (L_1 \otimes L_2)/q (L_1 \otimes L_2)$, the pair $(L_1 \otimes L_2, \mathcal{B}_1 \otimes \mathcal{B}_2)$ is an upper crystal basis of $V_1 \otimes V_2$.
\item[{\rm (2)}] For $i \in I$, $b_1 \in \mathcal{B}_1$, and $b_2 \in \mathcal{B}_2$,
\begin{align*}
&\tilde{e}_i(b_1 \otimes b_2) =
\begin{cases}
\tilde{e}_i b_1 \otimes b_2\ &{\it if}\ \varphi_i (b_1) \ge \varepsilon_i (b_2),\\
b_1 \otimes \tilde{e}_i b_2\ &{\it if}\ \varphi_i (b_1) < \varepsilon_i (b_2),
\end{cases}\\
&\tilde{f}_i(b_1 \otimes b_2) =
\begin{cases}
\tilde{f}_i b_1 \otimes b_2\ &{\it if}\ \varphi_i (b_1) > \varepsilon_i (b_2),\\
b_1 \otimes \tilde{f}_i b_2\ &{\it if}\ \varphi_i (b_1) \le \varepsilon_i (b_2),
\end{cases}\\
&\varepsilon_i(b_1 \otimes b_2) = \max\{\varepsilon_i(b_1),\ \varepsilon_i(b_2) - \langle{\rm wt}(b_1), h_i\rangle\},\\
&\varphi_i(b_1 \otimes b_2) = \max\{\varphi_i(b_2),\ \varphi_i(b_1) + \langle{\rm wt}(b_2), h_i\rangle\}.
\end{align*}
Here, $\varepsilon_i$ and $\varphi_i$ are the maps defined in Lemma \ref{length of string}.
\end{enumerate} 
\end{prop}\vspace{2mm}

We call $(L_1 \otimes L_2, \mathcal{B}_1 \otimes \mathcal{B}_2)$ the {\it tensor product} of $(L_1, \mathcal{B}_1)$ and $(L_2, \mathcal{B}_2)$. The following is easily seen from the tensor product rule for crystals (Proposition \ref{tensor product of crystals} (2)).

\vspace{2mm}\begin{cor}\label{tensor product corollary}
Let $V_1, V_2$ be finite-dimensional $U_q(\mathfrak{g})$-modules, and $(L_1, \mathcal{B}_1), (L_2, \mathcal{B}_2)$ upper crystal bases of $V_1, V_2$, respectively. For $i \in I$, $a \in \z_{\ge 0}$, $b_1 \in \mathcal{B}_1$, and $b_2 \in \mathcal{B}_2$,
\begin{enumerate}
\item[{\rm (1)}] it holds that \[\tilde{f}_i ^a(b_1 \otimes b_2) =
\begin{cases}
\tilde{f}_i ^a b_1 \otimes b_2\ &{\it if}\ a \le \varphi_i (b_1) - \varepsilon_i (b_2),\\
\tilde{f}_i ^{\varphi_i (b_1) - \varepsilon_i (b_2)}b_1 \otimes \tilde{f}_i ^{a - (\varphi_i (b_1) - \varepsilon_i (b_2))}b_2\ &{\it otherwise};
\end{cases}\]
\item[{\rm (2)}] if $\tilde{f}_i ^a b_1 \neq 0$ and $\tilde{e}_i b_2 = 0$, then \[\tilde{f}_i ^a(b_1 \otimes b_2) = \tilde{f}_i ^a b_1 \otimes b_2;\]
\item[{\rm (3)}] if $\tilde{f}_i b_1 = 0$, then \[\tilde{f}_i ^a(b_1 \otimes b_2) = b_1 \otimes \tilde{f}_i ^a b_2.\]
\end{enumerate}
\end{cor}\vspace{2mm}

\begin{prop}[{see \cite[Theorem 2, Theorem 5, Lemma 8, and Corollary 10]{LLM}}]\label{generalized Demazure crystal}
For an arbitrary word ${\bf i} = (i_1, \ldots, i_r) \in I^r$ and ${\bf m} = (m_1, \ldots, m_r) \in \z_{\ge 0} ^r$, denote by $\mathcal{B}_{{\bf i}, {\bf m}} \subset \mathcal{B}(m_1 \varpi_{i_1}) \otimes \cdots \otimes \mathcal{B}(m_r \varpi_{i_r})$ the subset
\[\{\tilde{f}_{i_1} ^{a_1} (b_{m_1 \varpi_{i_1}} \otimes \tilde{f}_{i_2} ^{a_2} (b_{m_2 \varpi_{i_2}} \otimes \cdots \otimes \tilde{f}_{i_{r-1}} ^{a_{r-1}} (b_{m_{r-1} \varpi_{i_{r-1}}} \otimes \tilde{f}_{i_r} ^{a_r} (b_{m_r \varpi_{i_r}}))\cdots)) \mid a_1, \ldots, a_r \in \z_{\ge 0}\} \setminus \{0\}.\]
\begin{enumerate}
\item[{\rm (1)}] $\tilde{e}_i (\mathcal{B}_{{\bf i}, {\bf m}}) \subset \mathcal{B}_{{\bf i}, {\bf m}} \cup \{0\}$ for all $i \in I$.
\item[{\rm (2)}] Let $\coprod_{1 \le k \le l} \mathcal{B}(\lambda_k)$ be the decomposition of the crystal graph of $\mathcal{B}(m_1 \varpi_{i_1}) \otimes \cdots \otimes \mathcal{B}(m_r \varpi_{i_r})$ into its connected components. If $\mathcal{B}(\lambda_k) \cap \mathcal{B}_{{\bf i}, {\bf m}} \neq \emptyset$ for $1 \le k \le l$, then there exists $w_k \in W$ such that $\mathcal{B}(\lambda_k) \cap \mathcal{B}_{{\bf i}, {\bf m}} = \mathcal{B}_{w_k} (\lambda_k)$.
\item[{\rm (3)}] The character of $V_{{\bf i}, {\bf m}}$ is equal to $\sum_{b \in \mathcal{B}_{{\bf i}, {\bf m}}} e^{{\rm wt}(b)}$.
\end{enumerate}
\end{prop}\vspace{2mm}

We call $\mathcal{B}_{{\bf i}, {\bf m}}$ a {\it generalized Demazure crystal}. In Section 7, we construct an explicit basis of $H^0(Z_{\bf i}, \mathcal{L}_{{\bf i}, {\bf m}})$ parameterized by $\mathcal{B}_{{\bf i}, {\bf m}}$, which can be regarded as a perfect basis (see \cite[Definition 5.30]{BK} and \cite[Definition 2.5]{KOP} for the definition). 

\vspace{2mm}\begin{rem}\normalfont\label{Schubert variety case}
Let ${\bf i}$ be a reduced word for $w \in W$, and $\lambda$ a dominant integral weight. If we take a dominant integral weight $\lambda^\prime$ and ${\bf m} \in \z_{\ge 0} ^r$ as in Proposition \ref{resolution} (2), then the crystal graph of $\mathcal{B}_w (\lambda)$ is identical to that of $b_{\lambda^\prime} \otimes \mathcal{B}_{{\bf i}, {\bf m}}$. Hence the notion of generalized Demazure crystal indeed generalizes that of Demazure crystal.
\end{rem}\vspace{2mm}

We extend the notion of string parameterizations for Demazure crystals to generalized Demazure crystals. 

\vspace{2mm}\begin{defi}\normalfont\label{generalized string parameterization}
Let ${\bf i} = (i_1, \ldots, i_r) \in I^r$ be an arbitrary word, and ${\bf m} = (m_1, \ldots, m_r) \in \z_{\ge 0} ^r$. Recall that ${\bf i}_{\ge s} = (i_s, \ldots, i_r)$ and ${\bf m}_{\ge s} = (m_s, \ldots, m_r)$ for $1 \le s \le r$. For $b \in \mathcal{B}_{{\bf i}, {\bf m}}$, define $\Omega_{\bf i} (b) = (a_1, \ldots, a_r) \in \z_{\ge 0} ^r$ and $b(s) \in \mathcal{B}_{{\bf i}_{\ge s}, {\bf m}_{\ge s}}$, $1 \le s \le r$, as follows. First, set $b(1):= b$ and $a_1:= \max\{a \in \z_{\ge 0} \mid \tilde{e}_{i_1} ^a b(1) \neq 0\}$. By the definition of $\mathcal{B}_{{\bf i}, {\bf m}}$, there exists $b(2) \in \mathcal{B}(m_2 \varpi_{i_2}) \otimes \cdots \otimes \mathcal{B}(m_r \varpi_{i_r})$ such that $\tilde{e}_{i_1} ^{a_1} b(1) = b_{m_1 \varpi_{i_1}} \otimes b(2)$. Then, it follows from Proposition \ref{generalized Demazure crystal} (1) that $b(2) \in \mathcal{B}_{{\bf i}_{\ge 2}, {\bf m}_{\ge 2}}$. Inductively, define $a_s \in \z_{\ge 0}$ and $b(s+1) \in \mathcal{B}_{{\bf i}_{\ge s+1}, {\bf m}_{\ge s+1}}$, $2 \le s \le r-1$, by $a_s:= \max\{a \in \z_{\ge 0} \mid \tilde{e}_{i_s} ^a b(s) \neq 0\}$, and by $\tilde{e}_{i_s} ^{a_s} b(s) = b_{m_s \varpi_{i_s}} \otimes b(s+1)$. Finally, set $a_r:= \max\{a \in \z_{\ge 0} \mid \tilde{e}_{i_r} ^a b(r) \neq 0\}$. The $\Omega_{\bf i} (b) = (a_1, \ldots, a_r)$ is called the {\it generalized string parameterization} of $b$ with respect to ${\bf i}$.
\end{defi}\vspace{2mm}

\begin{rem}\normalfont
In the situation of Remark \ref{Schubert variety case}, the generalized string parameterization $\Omega_{\bf i}$ is just the usual string parameterization with respect to the reduced word ${\bf i}$. 
\end{rem}\vspace{2mm}

\begin{prop}\label{generalized parameterization}
The following hold.
\begin{enumerate}
\item[{\rm (1)}] If $\Omega_{\bf i} (b) = (a_1, \ldots, a_r)$ for $b \in \mathcal{B}_{{\bf i}, {\bf m}}$, then \[b = \tilde{f}_{i_1} ^{a_1} (b_{m_1 \varpi_{i_1}} \otimes \tilde{f}_{i_2} ^{a_2} (b_{m_2 \varpi_{i_2}} \otimes \cdots \otimes \tilde{f}_{i_{r-1}} ^{a_{r-1}} (b_{m_{r-1} \varpi_{i_{r-1}}} \otimes \tilde{f}_{i_r} ^{a_r} b_{m_r \varpi_{i_r}})\cdots)).\]
\item[{\rm (2)}] If $b, b^\prime \in \mathcal{B}_{{\bf i}, {\bf m}}$ are such that $b \neq b^\prime$, then $\Omega_{\bf i} (b) \neq \Omega_{\bf i} (b ^\prime)$.
\end{enumerate}
\end{prop}

\begin{proof} 
If we write $\Omega_{\bf i} (b) = (a_1, \ldots, a_r)$ for $b \in \mathcal{B}_{{\bf i}, {\bf m}}$, then it follows from the definition of $\Omega_{\bf i}$ that 
\begin{align*}
b = b(1) &= \tilde{f}_{i_1} ^{a_1} (b_{m_1 \varpi_{i_1}} \otimes b(2))\\
&= \tilde{f}_{i_1} ^{a_1} (b_{m_1 \varpi_{i_1}} \otimes \tilde{f}_{i_2} ^{a_2} (b_{m_2 \varpi_{i_2}} \otimes b(3)))\\
&= \cdots = \tilde{f}_{i_1} ^{a_1} (b_{m_1 \varpi_{i_1}} \otimes \tilde{f}_{i_2} ^{a_2} (b_{m_2 \varpi_{i_2}} \otimes \cdots \otimes \tilde{f}_{i_{r-1}} ^{a_{r-1}} (b_{m_{r-1} \varpi_{i_{r-1}}} \otimes \tilde{f}_{i_r} ^{a_r} b_{m_r \varpi_{i_r}})\cdots)),
\end{align*} 
which implies part (1). From this, we see that $b \in \mathcal{B}_{{\bf i}, {\bf m}}$ can be reconstructed from the generalized string parameterization $\Omega_{\bf i} (b)$, which implies part (2). This proves the proposition.
\end{proof}

\begin{ex}\normalfont
Let $G = SL_3(\c)$, and write 
\begin{align*}
&\mathcal{B}(\varpi_1):\quad
\begin{ytableau}
1
\end{ytableau}
\xrightarrow{\tilde{f}_1}
\begin{ytableau}
2
\end{ytableau}
\xrightarrow{\tilde{f}_2}
\begin{ytableau}
3
\end{ytableau},\\
&\mathcal{B}(\varpi_2):\quad
\begin{ytableau}
1\\
2
\end{ytableau}
\xrightarrow{\tilde{f}_2}
\begin{ytableau}
1\\
3
\end{ytableau}
\xrightarrow{\tilde{f}_1}
\begin{ytableau}
2\\
3
\end{ytableau}
\end{align*}
by using tableaux. In what follows, we denote the tensor product $b \otimes b^\prime$ by the tableau made by placing $b^\prime$ directly on the right of $b$; for instance, we write as \[
\begin{ytableau}
1
\end{ytableau} \otimes
\begin{ytableau}
1\\
2
\end{ytableau} \otimes
\begin{ytableau}
2
\end{ytableau} =
\begin{ytableau}
1 & 1 & 2 \\
\none & 2
\end{ytableau}.\] For ${\bf i} = (1, 2, 1)$ and ${\bf m} = (1, 1, 1)$, we have $b =
\begin{ytableau}
2 & 1 & 2 \\
\none & 3
\end{ytableau} \in \mathcal{B}_{{\bf i}, {\bf m}}$. If we write $\Omega_{\bf i} (b) = (a_1, a_2, a_3)$, then it follows that
\begin{align*}
&a_1 = 1,\ \tilde{e}_{1} ^{a_1} b(1) = \tilde{e}_1 ^1 b(1) =
\begin{ytableau}
1 & 1 & 2 \\
\none & 3
\end{ytableau},\\
&a_2 = 1,\ \tilde{e}_{2} ^{a_2} b(2) = \tilde{e}_2 ^1 b(2) = 
\begin{ytableau}
1 & 2  \\
2
\end{ytableau},\\
&a_3 = 1,\ \tilde{e}_{1} ^{a_3} b(3) = \tilde{e}_{1} ^1 b(3) = 
\begin{ytableau}
1
\end{ytableau},
\end{align*}
and hence that $\Omega_{\bf i} (b) = (1, 1, 1)$. Moreover, all the elements of $\mathcal{B}_{{\bf i}, {\bf m}}$ are 
\begin{align*}
&\begin{ytableau}
1 & 1 & 1 \\
\none & 2
\end{ytableau},
\begin{ytableau}
2 & 1 & 1 \\
\none & 2
\end{ytableau},
\begin{ytableau}
2 & 1 & 2 \\
\none & 2
\end{ytableau},
\begin{ytableau}
1 & 1 & 1 \\
\none & 3
\end{ytableau},
\begin{ytableau}
2 & 1 & 1 \\
\none & 3
\end{ytableau},\\
&\begin{ytableau}
2 & 2 & 1 \\
\none & 3
\end{ytableau},
\begin{ytableau}
2 & 2 & 2 \\
\none & 3
\end{ytableau},
\begin{ytableau}
1 & 1 & 2 \\
\none & 2
\end{ytableau},
\begin{ytableau}
1 & 1 & 2 \\
\none & 3
\end{ytableau},
\begin{ytableau}
2 & 1 & 2 \\
\none & 3
\end{ytableau},\\
&\begin{ytableau}
1 & 1 & 3 \\
\none & 3
\end{ytableau},
\begin{ytableau}
2 & 1 & 3 \\
\none & 3
\end{ytableau},
\begin{ytableau}
2 & 2 & 3 \\
\none & 3
\end{ytableau},
\end{align*}
and the corresponding generalized string parameterizations are
\begin{align*}
&(0, 0, 0), (1, 0, 0), (2, 0, 0), (0, 1, 0), (1, 1, 0),\\
&(2, 1, 0), (3, 1, 0), (0, 0, 1), (0, 1, 1), (1, 1, 1),\\
&(0, 2, 1), (1, 2, 1), (2, 2, 1).
\end{align*}
\end{ex}\vspace{2mm}

\begin{ex}\normalfont
Let $G = Sp_4(\c)$, and write 
\begin{align*}
&\mathcal{B}(\varpi_1):\quad
\begin{ytableau}
1
\end{ytableau}
\xrightarrow{\tilde{f}_1}
\begin{ytableau}
2
\end{ytableau}
\xrightarrow{\tilde{f}_2}
\begin{ytableau}
\overline{2}
\end{ytableau}
\xrightarrow{\tilde{f}_1}
\begin{ytableau}
\overline{1}
\end{ytableau},\\
&\mathcal{B}(\varpi_2):\quad
\begin{ytableau}
1\\
2
\end{ytableau}
\xrightarrow{\tilde{f}_2}
\begin{ytableau}
1\\
\overline{2}
\end{ytableau}
\xrightarrow{\tilde{f}_1}
\begin{ytableau}
2\\
\overline{2}
\end{ytableau}
\xrightarrow{\tilde{f}_1}
\begin{ytableau}
2\\
\overline{1}
\end{ytableau}
\xrightarrow{\tilde{f}_2}
\begin{ytableau}
\overline{2}\\
\overline{1}
\end{ytableau}
\end{align*}
by using marked tableaux. For ${\bf i} = (1, 2, 1, 2)$ and ${\bf m} = (1, 1, 1, 1)$, all the elements of $\mathcal{B}_{{\bf i}, {\bf m}}$ are 
\begin{align*}
&\begin{ytableau}
1 & 1 & 1 & 1\\
\none & 2 &\none & 2
\end{ytableau},
\begin{ytableau}
2 & 1 & 1 & 1\\
\none & 2 &\none & 2
\end{ytableau},
\begin{ytableau}
2 & 1 & 2 & 1\\
\none & 2 &\none & 2
\end{ytableau},
\begin{ytableau}
1 & 1 & 1 & 1\\
\none & \overline{2} &\none & 2
\end{ytableau},
\begin{ytableau}
2 & 1 & 1 & 1\\
\none & \overline{2} &\none & 2
\end{ytableau},
\begin{ytableau}
2 & 2 & 1 & 1\\
\none & \overline{2} &\none & 2
\end{ytableau},
\begin{ytableau}
2 & 2 & 1 & 1\\
\none & \overline{1} &\none & 2
\end{ytableau},
\begin{ytableau}
2 & 2 & 2 & 1\\
\none & \overline{1} &\none & 2
\end{ytableau},\\
&\begin{ytableau}
1 & 1 & 1 & 1\\
\none & \overline{2} &\none & \overline{2}
\end{ytableau},
\begin{ytableau}
2 & 1 & 1 & 1\\
\none & \overline{2} &\none & \overline{2}
\end{ytableau},
\begin{ytableau}
2 & 2 & 1 & 1\\
\none & \overline{2} &\none & \overline{2}
\end{ytableau},
\begin{ytableau}
2 & 2 & 1 & 1\\
\none & \overline{1} &\none & \overline{2}
\end{ytableau},
\begin{ytableau}
2 & 2 & 2 & 1\\
\none & \overline{1} &\none & \overline{2}
\end{ytableau},
\begin{ytableau}
2 & 2 & 2 & 2\\
\none & \overline{1} &\none & \overline{2}
\end{ytableau},
\begin{ytableau}
2 & 2 & 2 & 2\\
\none & \overline{1} &\none & \overline{1}
\end{ytableau},
\begin{ytableau}
1 & 1 & 2 & 1\\
\none & 2 &\none & 2
\end{ytableau},\\
&\begin{ytableau}
1 & 1 & 2 & 1\\
\none & \overline{2} &\none & 2
\end{ytableau},
\begin{ytableau}
2 & 1 & 2 & 1\\
\none & \overline{2} &\none & 2
\end{ytableau},
\begin{ytableau}
2 & 2 & 2 & 1\\
\none & \overline{2} &\none & 2
\end{ytableau},
\begin{ytableau}
1 & 1 & \overline{2} & 1\\
\none & \overline{2} &\none & 2
\end{ytableau},
\begin{ytableau}
2 & 1 & \overline{2} & 1\\
\none & \overline{2} &\none & 2
\end{ytableau},
\begin{ytableau}
2 & 2 & \overline{2} & 1\\
\none & \overline{2} &\none & 2
\end{ytableau},
\begin{ytableau}
2 & 2 & \overline{2} & 1\\
\none & \overline{1} &\none & 2
\end{ytableau},
\begin{ytableau}
2 & 2 & \overline{1} & 1\\
\none & \overline{1} &\none & 2
\end{ytableau},\\
&\begin{ytableau}
1 & 1 & \overline{2} & 1\\
\none & \overline{2} &\none & \overline{2}
\end{ytableau},
\begin{ytableau}
2 & 1 & \overline{2} & 1\\
\none & \overline{2} &\none & \overline{2}
\end{ytableau},
\begin{ytableau}
2 & 2 & \overline{2} & 1\\
\none & \overline{2} &\none & \overline{2}
\end{ytableau},
\begin{ytableau}
2 & 2 & \overline{2} & 1\\
\none & \overline{1} &\none & \overline{2}
\end{ytableau},
\begin{ytableau}
2 & 2 & \overline{1} & 1\\
\none & \overline{1} &\none & \overline{2}
\end{ytableau},
\begin{ytableau}
2 & 2 & \overline{1} & 2\\
\none & \overline{1} &\none & \overline{2}
\end{ytableau},
\begin{ytableau}
2 & 2 & \overline{1} & 2\\
\none & \overline{1} &\none & \overline{1}
\end{ytableau},
\begin{ytableau}
1 & 1 & 1 & 1\\
\none & 2 &\none & \overline{2}
\end{ytableau},\\
&\begin{ytableau}
2 & 1 & 1 & 1\\
\none & 2 &\none & \overline{2}
\end{ytableau},
\begin{ytableau}
2 & 1 & 2 & 1\\
\none & 2 &\none & \overline{2}
\end{ytableau},
\begin{ytableau}
2 & 1 & 2 & 2\\
\none & 2 &\none & \overline{2}
\end{ytableau},
\begin{ytableau}
2 & 1 & 2 & 2\\
\none & 2 &\none & \overline{1}
\end{ytableau},
\begin{ytableau}
1 & 1 & 2 & 1\\
\none & 2 &\none & \overline{2}
\end{ytableau},
\begin{ytableau}
1 & 1 & 2 & 2\\
\none & 2 &\none & \overline{2}
\end{ytableau},
\begin{ytableau}
1 & 1 & 2 & 2\\
\none & 2 &\none & \overline{1}
\end{ytableau},
\begin{ytableau}
1 & 1 & 2 & 1\\
\none & \overline{2} &\none & \overline{2}
\end{ytableau},\\
&\begin{ytableau}
2 & 1 & 2 & 1\\
\none & \overline{2} &\none & \overline{2}
\end{ytableau},
\begin{ytableau}
2 & 2 & 2 & 1\\
\none & \overline{2} &\none & \overline{2}
\end{ytableau},
\begin{ytableau}
2 & 2 & 2 & 2\\
\none & \overline{2} &\none & \overline{2}
\end{ytableau},
\begin{ytableau}
2 & 2 & 2 & 2\\
\none & \overline{2} &\none & \overline{1}
\end{ytableau},
\begin{ytableau}
1 & 1 & 2 & 2\\
\none & \overline{2} &\none & \overline{2}
\end{ytableau},
\begin{ytableau}
2 & 1 & 2 & 2\\
\none & \overline{2} &\none & \overline{2}
\end{ytableau},
\begin{ytableau}
2 & 1 & 2 & 2\\
\none & \overline{2} &\none & \overline{1}
\end{ytableau},
\begin{ytableau}
1 & 1 & \overline{2} & 2\\
\none & \overline{2} &\none & \overline{2}
\end{ytableau},\\
&\begin{ytableau}
2 & 1 & \overline{2} & 2\\
\none & \overline{2} &\none & \overline{2}
\end{ytableau},
\begin{ytableau}
2 & 2 & \overline{2} & 2\\
\none & \overline{2} &\none & \overline{2}
\end{ytableau},
\begin{ytableau}
2 & 2 & \overline{2} & 2\\
\none & \overline{1} &\none & \overline{2}
\end{ytableau},
\begin{ytableau}
2 & 2 & \overline{2} & 2\\
\none & \overline{1} &\none & \overline{1}
\end{ytableau},
\begin{ytableau}
1 & 1 & 2 & 2\\
\none & \overline{2} &\none & \overline{1}
\end{ytableau},
\begin{ytableau}
1 & 1 & \overline{2} & 2\\
\none & \overline{2} &\none & \overline{1}
\end{ytableau},
\begin{ytableau}
2 & 1 & \overline{2} & 2\\
\none & \overline{2} &\none & \overline{1}
\end{ytableau},
\begin{ytableau}
2 & 2 & \overline{2} & 2\\
\none & \overline{2} &\none & \overline{1}
\end{ytableau},\\
&\begin{ytableau}
1 & 1 & \overline{2} & \overline{2}\\
\none & \overline{2} &\none & \overline{1}
\end{ytableau},
\begin{ytableau}
2 & 1 & \overline{2} & \overline{2}\\
\none & \overline{2} &\none & \overline{1}
\end{ytableau},
\begin{ytableau}
2 & 2 & \overline{2} & \overline{2}\\
\none & \overline{2} &\none & \overline{1}
\end{ytableau},
\begin{ytableau}
2 & 2 & \overline{2} & \overline{2}\\
\none & \overline{1} &\none & \overline{1}
\end{ytableau},
\begin{ytableau}
2 & 2 & \overline{1} & \overline{2}\\
\none & \overline{1} &\none & \overline{1}
\end{ytableau},
\end{align*}
and the corresponding generalized string parameterizations are
\begin{align*}
&(0, 0, 0, 0), (1, 0, 0, 0), (2, 0, 0, 0), (0, 1, 0, 0), (1, 1, 0, 0), (2, 1, 0, 0), (3, 1, 0, 0), (4, 1, 0, 0)\\
&(0, 2, 0, 0), (1, 2, 0, 0), (2, 2, 0, 0), (3, 2, 0, 0), (4, 2, 0, 0), (5, 2, 0, 0), (6, 2, 0, 0), (0, 0, 1, 0)\\
&(0, 1, 1, 0), (1, 1, 1, 0), (2, 1, 1, 0), (0, 2, 1, 0), (1, 2, 1, 0), (2, 2, 1, 0), (3, 2, 1, 0), (4, 2, 1, 0)\\
&(0, 3, 1, 0), (1, 3, 1, 0), (2, 3, 1, 0), (3, 3, 1, 0), (4, 3, 1, 0), (5, 3, 1, 0), (6, 3, 1, 0), (0, 0, 0, 1)\\
&(1, 0, 0, 1), (2, 0, 0, 1), (3, 0, 0, 1), (4, 0, 0, 1), (0, 0, 1, 1), (1, 0, 1, 1), (2, 0, 1, 1), (0, 1, 1, 1)\\
&(1, 1, 1, 1), (2, 1, 1, 1), (3, 1, 1, 1), (4, 1, 1, 1), (0, 1, 2, 1), (1, 1, 2, 1), (2, 1, 2, 1), (0, 2, 2, 1)\\
&(1, 2, 2, 1), (2, 2, 2, 1), (3, 2, 2, 1), (4, 2, 2, 1), (0, 1, 3, 1), (0, 2, 3, 1), (1, 2, 3, 1), (2, 2, 3, 1)\\
&(0, 3, 3, 1), (1, 3, 3, 1), (2, 3, 3, 1), (3, 3, 3, 1), (4, 3, 3, 1).
\end{align*}
\end{ex}\vspace{2mm}

\begin{defi}\normalfont\label{generalized string polytope}
For an arbitrary word ${\bf i} \in I^r$ and ${\bf m} \in \z_{\ge 0} ^r$, define a subset $\mathcal{S}_{{\bf i}, {\bf m}} \subset \z_{>0} \times \z^r$ by \[\mathcal{S}_{{\bf i}, {\bf m}} := \bigcup_{k>0} \{(k, \Omega_{\bf i}(b)) \mid b \in \mathcal{B}_{{\bf i}, k{\bf m}}\},\] and denote by $\mathcal{C}_{{\bf i}, {\bf m}} \subset \r_{\ge 0} \times \r^r$ the smallest real closed cone containing $\mathcal{S}_{{\bf i}, {\bf m}}$. Also, let us define a subset $\Delta_{{\bf i}, {\bf m}} \subset \r^r$ by \[\Delta_{{\bf i}, {\bf m}} := \{{\bf a} \in \r^r \mid (1, {\bf a}) \in \mathcal{C}_{{\bf i}, {\bf m}}\};\] this is called the {\it generalized string polytope} associated to ${\bf i}$ and ${\bf m}$.
\end{defi}\vspace{2mm}

\begin{rem}\normalfont
In the situation of Remark \ref{Schubert variety case}, the generalized string polytope $\Delta_{{\bf i}, {\bf m}}$ is just the usual string polytope $\Delta_{\bf i} ^{(\lambda, w)}$.
\end{rem}\vspace{2mm}

In Section 7, we prove that the generalized string polytope $\Delta_{{\bf i}, {\bf m}}$ is indeed a rational convex polytope. Also, in Appendix B, we give examples of generalized string polytopes.

\subsection{Some properties}

Here, we prove some properties of generalized string parameterizations and generalized string polytopes. For $\widetilde{\bf m} = (m_1, \ldots, m_{r-1}) \in \z^{r-1} _{\ge 0}$, ${\bf a} = (a_1, \ldots, a_r) \in \z^r _{\ge 0}$, and a dominant integral weight $\lambda$, consider the element $T_\lambda (\widetilde{\bf m}, {\bf a}) \in \mathcal{B}(m_1 \varpi_{i_1}) \otimes \cdots \otimes \mathcal{B}(m_{r-1} \varpi_{i_{r-1}}) \otimes \mathcal{B}(\lambda) \cup \{0\}$ given by \[T_\lambda (\widetilde{\bf m}, {\bf a}):= \tilde{f}_{i_1} ^{a_1} (b_{m_1 \varpi_{i_1}} \otimes \tilde{f}_{i_2} ^{a_2} (b_{m_2 \varpi_{i_2}} \otimes \cdots \otimes \tilde{f}_{i_{r-1}} ^{a_{r-1}}(b_{m_{r-1} \varpi_{i_{r-1}}} \otimes \tilde{f}_{i_r} ^{a_r} b_\lambda)\cdots)).\] For $\widetilde{\bf m}$, ${\bf a}$, and $\lambda$ such that $T_\lambda (\widetilde{\bf m}, {\bf a}) \neq 0$, we define a parameterization $\Omega_{\bf i} (T_\lambda (\widetilde{\bf m}, {\bf a}))$ in the same way as in Definition \ref{generalized string parameterization}, and also call it a generalized string parameterization. Note that it follows from an argument similar to that for Proposition \ref{generalized parameterization} (1) that $T_\lambda (\widetilde{\bf m}, {\bf a}) = T_\lambda (\widetilde{\bf m}, \Omega_{\bf i} (T_\lambda (\widetilde{\bf m}, {\bf a})))$. First we show that the generalized string parameterization $\Omega_{\bf i} (T_\lambda (\widetilde{\bf m}, {\bf a}))$ does not essentially depend on $\lambda$ (cf.~the discussion in \cite[Section 1]{Lit} for usual string parameterizations).

\vspace{2mm}\begin{lem}\label{lemma,retake}
If $T_\lambda (\widetilde{\bf m}, {\bf a}) = T_\lambda (\widetilde{\bf m}, {\bf b}) \neq 0$, then $T_{\lambda^\prime} (\widetilde{\bf m}, {\bf a}) = T_{\lambda^\prime} (\widetilde{\bf m}, {\bf b})$ for every dominant integral weight $\lambda^\prime$.
\end{lem}

\begin{proof}
If $T_{\lambda^\prime} (\widetilde{\bf m}, {\bf a})$ and $T_{\lambda^\prime} (\widetilde{\bf m}, {\bf b})$ are both $0$, then the assertion is obvious. Hence we may assume that $T_{\lambda^\prime}(\widetilde{\bf m}, {\bf b}) \neq 0$. Take dominant integral weights $\mu, \mu^\prime$ such that $\lambda + \mu = \lambda^\prime + \mu^\prime$. Since $T_{\lambda^\prime}(\widetilde{\bf m}, {\bf b}) \neq 0$, we see that $T_{\lambda^\prime + \mu^\prime}(\widetilde{\bf m}, {\bf b}) = T_{\lambda^\prime}(\widetilde{\bf m}, {\bf b}) \otimes b_{\mu^\prime}$ by Corollary \ref{tensor product corollary} (2). Moreover, we deduce that
\begin{align*}
T_{\lambda^\prime + \mu^\prime}(\widetilde{\bf m}, {\bf a}) &= T_{\lambda + \mu}(\widetilde{\bf m}, {\bf a})\\ 
&= T_\lambda(\widetilde{\bf m}, {\bf a}) \otimes b_\mu\quad({\rm by\ Corollary}\ \ref{tensor product corollary}\ (2)\ {\rm since}\ T_\lambda (\widetilde{\bf m}, {\bf a}) \neq 0)\\
&= T_\lambda(\widetilde{\bf m}, {\bf b}) \otimes b_\mu\quad({\rm by\ the\ assumption})\\ 
&= T_{\lambda + \mu}(\widetilde{\bf m}, {\bf b})\quad({\rm by\ Corollary}\ \ref{tensor product corollary}\ (2)\ {\rm since}\ T_\lambda (\widetilde{\bf m}, {\bf b}) \neq 0)\\
&= T_{\lambda^\prime + \mu^\prime}(\widetilde{\bf m}, {\bf b}).
\end{align*}
From these, we obtain $T_{\lambda^\prime + \mu^\prime}(\widetilde{\bf m}, {\bf a}) = T_{\lambda^\prime}(\widetilde{\bf m}, {\bf b}) \otimes b_{\mu^\prime}$. This implies that the element $T_{\lambda^\prime+\mu^\prime}(\widetilde{\bf m}, {\bf a})$ must be of the form $T_{\lambda^\prime}(\widetilde{\bf m}, {\bf a}) \otimes b_{\mu^\prime}$; hence it follows that $T_{\lambda^\prime}(\widetilde{\bf m}, {\bf a}) = T_{\lambda^\prime}(\widetilde{\bf m}, {\bf b})$. This proves the lemma. 
\end{proof}

\begin{prop}\label{invariance of parameterizations}
Let $\lambda, \lambda^\prime$ be arbitrary dominant integral weights. If $T_\lambda (\widetilde{\bf m}, {\bf a}) \neq 0$ and $T_{\lambda^\prime} (\widetilde{\bf m}, {\bf a}) \neq 0$, then $\Omega_{\bf i} (T_\lambda (\widetilde{\bf m}, {\bf a})) = \Omega_{\bf i} (T_{\lambda^\prime} (\widetilde{\bf m}, {\bf a}))$.
\end{prop}

\begin{proof}
If we write $\Omega_{\bf i}(T_\lambda(\widetilde{\bf m}, {\bf a})) = {\bf b} = (b_1, \ldots, b_r)$ and $\Omega_{\bf i}(T_{\lambda^\prime}(\widetilde{\bf m}, {\bf a})) = {\bf c} = (c_1, \ldots, c_r)$, then
\begin{align*}
T_\lambda(\widetilde{\bf m}, {\bf b}) &= T_\lambda(\widetilde{\bf m}, {\bf a}) \neq 0,\ {\rm and}\\
T_{\lambda^\prime}(\widetilde{\bf m}, {\bf c}) &= T_{\lambda^\prime}(\widetilde{\bf m}, {\bf a}) \neq 0.
\end{align*}
Therefore, it follows from Lemma \ref{lemma,retake} that 
\begin{align}
T_{\lambda^\prime}(\widetilde{\bf m}, {\bf b}) &= T_{\lambda^\prime}(\widetilde{\bf m}, {\bf a}),\ {\rm and}\label{equivalent1}\\
T_\lambda(\widetilde{\bf m}, {\bf c}) &= T_\lambda(\widetilde{\bf m}, {\bf a}).\label{equivalent2}
\end{align}
By the definition of $\Omega_{\bf i}$, we deduce that 
\begin{align*}
c_1 &= \max\{a \in \z_{\ge 0} \mid \tilde{e}_{i_1} ^a T_{\lambda^\prime}(\widetilde{\bf m}, {\bf a}) \neq 0\}\\
&= \max\{a \in \z_{\ge 0} \mid \tilde{e}_{i_1} ^a T_{\lambda^\prime}(\widetilde{\bf m}, {\bf b}) \neq 0\}\quad({\rm by\ equation}\ (\ref{equivalent1}))\\
&\ge b_1\quad({\rm by\ the\ definition\ of}\ T_{\lambda^\prime}(\widetilde{\bf m}, {\bf b})). 
\end{align*}
Similarly, equation (\ref{equivalent2}) implies that $c_1 \le b_1$. Hence we obtain $b_1 = c_1$. Also, if we set ${\bf b}_{\ge 2} := (b_2, \ldots, b_r)$, ${\bf c}_{\ge 2} := (c_2, \ldots, c_r)$, and $\widetilde{\bf m}_{\ge 2} := (m_2, \ldots, m_{r-1})$, then
\begin{align*}
b_{m_1 \varpi_{i_1}} \otimes T_\lambda(\widetilde{\bf m}_{\ge 2}, {\bf b}_{\ge 2}) &= \tilde{e}_{i_1} ^{b_1} T_\lambda(\widetilde{\bf m}, {\bf b})\quad({\rm by\ the\ definition\ of}\ T_\lambda(\widetilde{\bf m}, {\bf b}))\\
&= \tilde{e}_{i_1} ^{c_1} T_\lambda(\widetilde{\bf m}, {\bf b})\quad({\rm since}\ b_1 = c_1)\\
&= \tilde{e}_{i_1} ^{c_1} T_\lambda(\widetilde{\bf m}, {\bf c})\quad({\rm by\ equation}\ (\ref{equivalent2})\ {\rm since}\ T_\lambda(\widetilde{\bf m}, {\bf b}) = T_\lambda(\widetilde{\bf m}, {\bf a}))\\
&= b_{m_1 \varpi_{i_1}} \otimes T_\lambda(\widetilde{\bf m}_{\ge 2}, {\bf c}_{\ge 2})\quad({\rm by\ the\ definition\ of}\ T_\lambda(\widetilde{\bf m}, {\bf c})).
\end{align*}
From this, we deduce that $T_\lambda(\widetilde{\bf m}_{\ge 2}, {\bf b}_{\ge 2}) = T_\lambda(\widetilde{\bf m}_{\ge 2}, {\bf c}_{\ge 2})$. Similarly, it follows that $T_{\lambda^\prime}(\widetilde{\bf m}_{\ge 2}, {\bf b}_{\ge 2}) = T_{\lambda^\prime}(\widetilde{\bf m}_{\ge 2}, {\bf c}_{\ge 2})$. Note that $\Omega_{{\bf i}_{\ge 2}}(T_\lambda(\widetilde{\bf m}_{\ge 2}, {\bf b}_{\ge 2})) = {\bf b}_{\ge 2}$ and $\Omega_{{\bf i}_{\ge 2}}(T_{\lambda^\prime}(\widetilde{\bf m}_{\ge 2}, {\bf c}_{\ge 2})) = {\bf c}_{\ge 2}$ by the definition of $\Omega_{\bf i}$, where ${\bf i}_{\ge 2} := (i_2, \ldots, i_r)$; hence the equality $b_2 = c_2$ follows from the same argument as in the proof of $b_1 = c_1$. Repeating this argument, we conclude that ${\bf b} = {\bf c}$. This proves the proposition.
\end{proof}

We define a subset $\mathcal{S}_{\bf i} \subset \z_{\ge 0} ^{r-1} \times \z_{\ge 0} ^r$ by \[\mathcal{S}_{\bf i}:= \{(\widetilde{\bf m}, {\bf a}) \in \z_{\ge 0} ^{r-1} \times \z_{\ge 0} ^r \mid \Omega_{\bf i} (T_\lambda (\widetilde{\bf m}, {\bf a})) = {\bf a}\ {\rm for\ some}\ \lambda\ {\rm such\ that}\ T_\lambda (\widetilde{\bf m}, {\bf a}) \neq 0\},\] and denote by $\mathcal{C}_{\bf i} \subset \r_{\ge 0} ^{r-1} \times \r_{\ge 0} ^r$ the smallest real closed cone containing $\mathcal{S}_{\bf i}$. Note that by Proposition \ref{invariance of parameterizations}, we have 
\begin{align}\label{independent of lambda}
\mathcal{S}_{\bf i} = \{(\widetilde{\bf m}, {\bf a}) \in \z_{\ge 0} ^{r-1} \times \z_{\ge 0} ^r \mid \Omega_{\bf i} (T_\lambda (\widetilde{\bf m}, {\bf a})) = {\bf a}\ {\rm for\ all}\ \lambda\ {\rm such\ that}\ T_\lambda (\widetilde{\bf m}, {\bf a}) \neq 0\}.
\end{align}

\vspace{2mm}\begin{prop}\label{cutting}
For ${\bf m} = (m_1, \ldots, m_r) \in \z_{\ge 0} ^r$, set $\widetilde{\bf m} := (m_1, \ldots, m_{r-1})$. Then, $\mathcal{S}_{{\bf i}, {\bf m}}$ is identical to the set of $(k, {\bf a}) = (k, a_1, \ldots, a_r) \in \z_{>0} \times \z_{\ge 0} ^r$ satisfying the following conditions$:$
\begin{enumerate}
\item[{\rm (i)}] $(k \widetilde{\bf m}, {\bf a}) \in \mathcal{S}_{\bf i};$
\item[{\rm (ii)}] $a_j \le k m_j + \sum_{j < s \le r} \delta_{i_j, i_s} k m_{s} - \sum_{j < s \le r} c_{i_j, i_s} a_s$ for all $1 \le j \le r$.
\end{enumerate}
Here, $\delta_{i_j, i_s}$ $($resp., $(c_{i, j})_{i, j \in I})$ denotes the Kronecker delta $($resp., the Cartan matrix of $\mathfrak{g})$.
\end{prop}

\begin{proof}
We take $(k, {\bf a}) = (k, a_1, \ldots, a_r) \in \z_{>0} \times \z_{\ge 0} ^r$, and show that $(k, {\bf a}) \in \mathcal{S}_{{\bf i}, {\bf m}}$ if and only if (i), (ii) hold. We prove the ``only if'' part. By the definition of $\mathcal{S}_{{\bf i}, {\bf m}}$, if $(k, {\bf a}) \in \mathcal{S}_{{\bf i}, {\bf m}}$, then it follows that 
\begin{equation}\label{nonzero}
\begin{aligned}
&T_{km_r \varpi_{i_r}}(k \widetilde{\bf m}, {\bf a}) \neq 0,\ {\rm and}\\
&\Omega_{\bf i} (T_{km_r \varpi_{i_r}}(k \widetilde{\bf m}, {\bf a})) = {\bf a}, 
\end{aligned}
\end{equation}
which implies (i) (we take $km_r \varpi_{i_r}$ as a weight $\lambda$ in the definition of $\mathcal{S}_{\bf i}$). If we set $\widetilde{\bf m}_{\ge 2} := (m_2, \ldots, m_{r-1})$ and ${\bf a}_{\ge 2} := (a_2, \ldots, a_r)$, then we see from the definitions that \[T_{km_r \varpi_{i_r}}(k \widetilde{\bf m}, {\bf a}) = \tilde{f}_{i_1} ^{a_1}(b_{km_1 \varpi_{i_1}} \otimes T_{km_r \varpi_{i_r}}(k \widetilde{\bf m}_{\ge 2}, {\bf a}_{\ge 2})).\] Also, the equality $\Omega_{\bf i} (T_{km_r \varpi_{i_r}}(k \widetilde{\bf m}, {\bf a})) = {\bf a}$ implies that \[\tilde{e}_{i_1}(b_{km_1 \varpi_{i_1}} \otimes T_{km_r \varpi_{i_r}}(k \widetilde{\bf m}_{\ge 2}, {\bf a}_{\ge 2})) = 0.\] Therefore, we deduce that
\begin{align*}
a_1 &\le \varphi_{i_1}(b_{km_1 \varpi_{i_1}} \otimes T_{km_r \varpi_{i_r}}(k \widetilde{\bf m}_{\ge 2}, {\bf a}_{\ge 2}))\quad({\rm since}\ T_{km_r \varpi_{i_r}}(k \widetilde{\bf m}, {\bf a}) \neq 0\ {\rm by}\ (\ref{nonzero}))\\ 
&= \langle {\rm wt}(b_{km_1 \varpi_{i_1}} \otimes T_{km_r \varpi_{i_r}}(k \widetilde{\bf m}_{\ge 2}, {\bf a}_{\ge 2})), h_{i_1} \rangle\quad({\rm by\ Lemma}\ \ref{length of string})\\
&= \langle km_1 \varpi_{i_1} + \cdots + km_r \varpi_{i_r} -a_2 \a_{i_2} - \cdots - a_r \a_{i_r}, h_{i_1}\rangle\quad({\rm by\ the\ definition\ of}\ T_{km_r \varpi_{i_r}}(k \widetilde{\bf m}_{\ge 2}, {\bf a}_{\ge 2}))\\
&= km_1 + \sum_{1 < s \le r} \delta_{i_1, i_s}km_{s} - \sum_{1 < s \le r} c_{i_1, i_s}a_s.
\end{align*}
Repeating this argument, with $T_{km_r \varpi_{i_r}}(k \widetilde{\bf m}, {\bf a})$ replaced by $T_{km_r \varpi_{i_r}}(k \widetilde{\bf m}_{\ge 2}, {\bf a}_{\ge 2})$, we conclude (ii). Thus we have proved the ``only if'' part. Then, by reversing the arguments above, we deduce the ``if'' part. This completes the proof of the proposition.
\end{proof}

Next we give a system of piecewise-linear inequalities defining $\mathcal{S}_{\bf i}$. We set ${\bf a} ^{(r)} = (a_1 ^{(r)}, \ldots, a_r ^{(r)}) := (a_1, \ldots, a_r)$, and then define ${\bf a}^{(j-1)} = (a_1 ^{(j-1)}, \ldots, a_{j-1} ^{(j-1)})$ for $1 < j \le r$ and $\Psi ^{j, k} _{\bf i}(\widetilde{\bf m}, {\bf a})$, $1 \le k < j$, for $1 < j \le r$ inductively by 
\begin{align*}
&\Psi ^{j, k} _{\bf i}(\widetilde{\bf m}, {\bf a}) := \begin{cases}
\max\{a_l ^{(j)} - \sum_{k < s \le l} c_{i_j, i_s} a_s ^{(j)} + \sum_{k \le s<l} \delta_{i_j, i_s} m_s \mid k < l \le j,\ i_l = i_j\ (= i_k)\} &{\rm if}\ i_k = i_j,\\
a_k ^{(j)} &{\rm otherwise},
\end{cases}\\
&a_k ^{(j-1)} := \min\{a_k ^{(j)}, \Psi ^{j, k} _{\bf i}(\widetilde{\bf m}, {\bf a})\}.
\end{align*}
We can regard $\Psi ^{j, k} _{\bf i}(\widetilde{\bf m}, {\bf a})$ as a piecewise-linear function of $m_1, \ldots, m_{r-1}$ and $a_1, \ldots, a_r$.

\vspace{2mm}\begin{prop}\label{inequality}
For an arbitrary word ${\bf i} \in I^r$, \[\mathcal{S} _{\bf i} = \{(\widetilde{\bf m}, {\bf a}) \in \z ^{r-1} _{\ge 0} \times \z ^r _{\ge 0} \mid \Psi ^{j, k} _{\bf i}(\widetilde{\bf m}, {\bf a}) \ge 0\ {\it for\ all}\ 1 \le k < j \le r\}.\]
\end{prop}\vspace{2mm}

We will give a proof of this proposition in Appendix A; the proof is almost parallel to that of \cite[Lemma 1.6]{Lit}. From this explicit description of $\mathcal{S}_{\bf i}$, we obtain the following corollaries. Since $\Psi ^{j, k} _{\bf i}(\widetilde{\bf m}, {\bf a})$, $1 \le k < j \le r$, are piecewise-linear functions, it follows that the real closed cone $\mathcal{C}_{\bf i}$ is also given by $\Psi ^{j, k} _{\bf i}(\widetilde{\bf m}, {\bf a}) \ge 0$, $1 \le k < j \le r$, for $(\widetilde{\bf m}, {\bf a}) \in \r^{r-1} _{\ge 0} \times \r^r _{\ge 0}$. Thus, we obtain the following.

\vspace{2mm}\begin{cor}
The real closed cone $\mathcal{C} _{\bf i}$ is a finite union of rational convex polyhedral cones, and the equality $\mathcal{S} _{\bf i} = \mathcal{C} _{\bf i} \cap (\z^{r-1} _{\ge 0} \times \z^r _{\ge 0})$ holds.
\end{cor}\vspace{2mm}

Also, we know the following from Proposition \ref{cutting}.

\vspace{2mm}\begin{cor}\label{piecewise-linear inequalities 3}
For ${\bf m} = (m_1, \ldots, m_r) \in \z_{\ge 0} ^r$, set $\widetilde{\bf m} := (m_1, \ldots, m_{r-1})$. Then, $\mathcal{S}_{{\bf i}, {\bf m}}$ is identical to the set of $(k, {\bf a}) = (k, a_1, \ldots, a_r) \in \z_{>0} \times \z_{\ge 0} ^r$ satisfying the following conditions$:$
\begin{enumerate}
\item[{\rm (i)}] $\Psi ^{j, l} _{\bf i}(k \widetilde{\bf m}, {\bf a}) \ge 0$ for all $1 \le l < j \le r;$
\item[{\rm (ii)}] $a_j \le k m_j + \sum_{j < s \le r} \delta_{i_j, i_s} k m_{s} - \sum_{j < s \le r} c_{i_j, i_s} a_s$ for all $1 \le j \le r$.
\end{enumerate}
In particular, the real closed cone $\mathcal{C}_{{\bf i}, {\bf m}}$ is a finite union of rational convex polyhedral cones, and the equality $\mathcal{S}_{{\bf i}, {\bf m}} = \mathcal{C}_{{\bf i}, {\bf m}} \cap (\z_{>0} \times \z^r _{\ge 0})$ holds.
\end{cor}\vspace{2mm}

\begin{cor}\label{computation of generalized string polytopes}
For ${\bf m} = (m_1, \ldots, m_r) \in \z_{\ge 0} ^r$, set $\widetilde{\bf m} := (m_1, \ldots, m_{r-1})$. Then, the generalized string polytope $\Delta_{{\bf i}, {\bf m}}$ is identical to the set of ${\bf a} = (a_1, \ldots, a_r) \in \r_{\ge 0} ^r$ satisfying the following conditions$:$
\begin{enumerate}
\item[{\rm (i)}] $\Psi ^{j, l} _{\bf i}(\widetilde{\bf m}, {\bf a}) \ge 0$ for all $1 \le l < j \le r;$
\item[{\rm (ii)}] $a_j \le m_j + \sum_{j < s \le r} \delta_{i_j, i_s} m_{s} - \sum_{j < s \le r} c_{i_j, i_s} a_s$ for all $1 \le j \le r$.
\end{enumerate}
\end{cor}\vspace{2mm}

\begin{cor}\label{finite union of polytopes}
The generalized string polytope $\Delta_{{\bf i}, {\bf m}}$ is a finite union of rational convex polytopes, and the equality $\Omega_{\bf i}(\mathcal{B}_{{\bf i}, {\bf m}}) = \Delta_{{\bf i}, {\bf m}} \cap \z^r$ holds.
\end{cor}

\begin{proof}
From condition (ii) in Corollary \ref{computation of generalized string polytopes}, we deduce that the generalized string polytope $\Delta_{{\bf i}, {\bf m}}$ is bounded, and hence compact. Also, by Corollary \ref{computation of generalized string polytopes}, it follows that the generalized string polytope $\Delta_{{\bf i}, {\bf m}}$ is given by a finite number of piecewise-linear inequalities, which implies the first assertion of the corollary. Now the second assertion is an immediate consequence of Corollary \ref{piecewise-linear inequalities 3}. This proves the corollary.
\end{proof}

\begin{cor}
For $i \in I$, $b \in \mathcal{B}_{{\bf i}, {\bf m}}$, and $1 \le j \le r$ such that $i_j = i$, set \[\widetilde{\Psi}_i ^{(j)} (b) := \max\{a_l ^{(j)} - \sum_{1 \le s \le l} c_{i_j, i_s} a_s ^{(j)} + \sum_{1 \le s < l} \delta_{i_j, i_s} m_s \mid 1 \le l \le j,\ i_l = i_j\ (= i)\},\] where ${\bf a} = (a_1, \ldots, a_r) := \Omega_{\bf i}(b)$, and ${\bf a}^{(j)} = (a_1 ^{(j)}, \ldots, a_j ^{(j)})$ is defined as above. Then, \[\varepsilon_i (b) = 
\begin{cases}
\max\{\widetilde{\Psi}_i ^{(j)} (b) \mid 1 \le j \le r,\ i_j = i\} &{\it if}\ \widetilde{\Psi}_i ^{(j)} (b) \ge 0\ {\it for\ some}\ 1 \le j \le r,\ i_j = i,\\
0 &{\it otherwise}.
\end{cases}\]
\end{cor}

\begin{proof}
For $i \in I$ and $m \in \z_{\ge 0}$, consider the generalized string parameterization $\Omega_{(i, {\bf i})}$ for $\mathcal{B}_{(i, {\bf i}), (m, {\bf m})}$, where $(i, {\bf i}) := (i, i_1, \ldots, i_r)$ and $(m, {\bf m}) := (m, m_1, \ldots, m_r)$. If we write $\Omega_{\bf i}(b) = {\bf a} = (a_1, \ldots, a_r)$ for $b \in \mathcal{B}_{{\bf i}, {\bf m}}$, then we see from the definition of $\Omega_{(i, {\bf i})}$ that $\Omega_{(i, {\bf i})} (b_{m \varpi_i} \otimes b) = (0, {\bf a})$ if and only if $\varepsilon_i (b_{m \varpi_i} \otimes b) = 0$, where $(0, {\bf a}) := (0, a_1, \ldots, a_r)$. Now equation (\ref{independent of lambda}) implies that $\Omega_{(i, {\bf i})}(b_{m \varpi_i} \otimes b) = (0, {\bf a})$ if and only if $((m, m_1, \ldots, m_{r-1}), (0, {\bf a})) \in \mathcal{S}_{(i, {\bf i})}$. Applying Proposition \ref{inequality} to $\mathcal{S}_{(i, {\bf i})}$, this is also equivalent to $\widetilde{\Psi}_i ^{(j)} (b) \le m$ for all $1 \le j \le r$ such that $i_j = i$. Also, it follows from the assertion in Proposition \ref{tensor product of crystals} (2) for $\varepsilon_i(b_1 \otimes b_2)$ that $\varepsilon_i (b_{m \varpi_i} \otimes b) = 0$ if and only if $\varepsilon_i(b) \le m$. From these, we conclude that $\max\{\widetilde{\Psi}_i ^{(j)} (b) \mid 1 \le j \le r,\ i_j = i\} \le m$ if and only if $\varepsilon_i(b) \le m$. This proves the corollary.
\end{proof}

\section{Upper global bases of tensor product modules}

In this section, we recall some basic facts about an upper global basis of a tensor product module, following \cite{Lus} and \cite{Kas5}. Define $\overline{\Delta}: U_q (\mathfrak{g}) \rightarrow U_q (\mathfrak{g}) \otimes U_q (\mathfrak{g})$ by $\overline{\Delta} (u) := \overline{\Delta (\overline{u})}$ for $u \in U_q (\mathfrak{g})$, where we denote by $\overline{\vphantom{(}\cdot\vphantom{)}}$ the $\q$-involution $\overline{\vphantom{(}\cdot\vphantom{)}}  \otimes \overline{\vphantom{(}\cdot\vphantom{)}}: U_q (\mathfrak{g}) \otimes U_q (\mathfrak{g}) \rightarrow U_q (\mathfrak{g}) \otimes U_q (\mathfrak{g})$. Also, let us write \[U_q ^\pm (\mathfrak{g})_\nu := \{u \in U_q ^\pm (\mathfrak{g}) \mid t_i u t_i ^{-1} = q_i ^{\langle\nu, h_i\rangle} u\ {\rm for\ all}\ i \in I\}\] for $\nu \in \sum_{i \in I} \z \alpha_i$.

\vspace{2mm}\begin{lem}[{see \cite[Theorem 4.1.2 (a)]{Lus}}]
Set $Q_{\ge 0} := \sum_{i \in I} \z_{\ge 0} \alpha_i$. There exists a unique family of elements $\{\Theta_\nu \in U_q ^-(\mathfrak{g})_{-\nu} \otimes U_q ^+(\mathfrak{g})_\nu \mid \nu \in Q_{\ge 0}\}$ satisfying the following conditions$:$
\begin{enumerate}
\item[{\rm (i)}] $\Theta_0 = 1\otimes 1;$
\item[{\rm (ii)}] for all $u \in U_q (\mathfrak{g})$ and finite-dimensional $U_q(\mathfrak{g})$-modules $V_1, V_2$, the equality $\Delta(u) \circ \Theta = \Theta \circ \overline{\Delta}(u)$ holds as endomorphisms of $V_1 \otimes V_2$, where $\Theta := \sum_{\nu \in Q_{\ge 0}} \Theta_\nu$.
\end{enumerate}
\end{lem}\vspace{2mm}

The $\Theta$ above is called the {\it quasi-R-matrix}. By using $\Theta$, we can construct a bar involution on $V_1 \otimes V_2$.

\vspace{2mm}\begin{prop}[{see \cite[\S\S 4.1]{Lus}}]\label{bar involution}
For $i = 1, 2$, let $V_i$ be a finite-dimensional $U_q (\mathfrak{g})$-module, and $\overline{\vphantom{(}\cdot\vphantom{)}}$ its bar involution. Then, a $\q$-linear map $\overline{\vphantom{(}\cdot\vphantom{)}} := \Theta \circ (\overline{\vphantom{(}\cdot\vphantom{)}} \otimes \overline{\vphantom{(}\cdot\vphantom{)}})$ is a bar involution on $V_1 \otimes V_2$.
\end{prop}\vspace{2mm}

For finite-dimensional $U_q (\mathfrak{g})$-modules $V_1, V_2$, and $V_3$, we can construct bar involutions on $(V_1 \otimes V_2) \otimes V_3$ and $V_1 \otimes (V_2 \otimes V_3)$ by using Proposition \ref{bar involution} repeatedly. These bar involutions coincide through the natural isomorphism $(V_1 \otimes V_2) \otimes V_3 \simeq V_1 \otimes (V_2 \otimes V_3)$ (see \cite[\S\S 27.3]{Lus}). Hence, for finite-dimensional $U_q (\mathfrak{g})$-modules $V_1, \ldots, V_r$, we obtain a (unique) bar involution on $V_1 \otimes \cdots \otimes V_r$ by applying the proposition repeatedly.

\vspace{2mm}\begin{prop}[{see \cite[\S\S 24.2 and \S\S 27.3]{Lus} and \cite[Section 8]{Kas5}}]\label{global basis of tensor product}
For $i = 1, 2$, let $V_i$ be a finite-dimensional $U_q (\mathfrak{g})$-module, $(L_i, \mathcal{B}_i)$ its upper crystal basis, $\overline{\vphantom{(}\cdot\vphantom{)}}$ its bar involution, and $(V_i ^\q, L_i, \overline{L_i})$ a balanced triple. Assume that $\overline{v} = v$ for all $v \in V_i ^\q \cap L_i \cap \overline{L_i}$.
\begin{enumerate}
\item[{\rm (1)}] For the bar involution $\overline{\vphantom{(}\cdot\vphantom{)}}$ on $V_1 \otimes V_2$ defined in Proposition \ref{bar involution}, $(V_1 ^\q \otimes V_2 ^\q, L_1 \otimes L_2, \overline{L_1 \otimes L_2})$ is a balanced triple.
\item[{\rm (2)}] For the corresponding upper global basis $\{G^{\rm up} _q (b_1 \otimes b_2) \mid b_1 \in \mathcal{B}_1,\ b_2 \in \mathcal{B}_2\}$, it holds that $\overline{G^{\rm up} _q (b_1 \otimes b_2)} = G^{\rm up} _q (b_1 \otimes b_2)$, and that \[G^{\rm up} _q (b_1 \otimes b_2) \in G^{\rm up} _q (b_1) \otimes G^{\rm up} _q (b_2) + \sum_{\substack{{\rm wt}(b_1) - {\rm wt}(b_1 ^\prime) \in Q_{>0},\\{\rm wt}(b_2 ^\prime) - {\rm wt}(b_2) \in Q_{>0}}} q \z[q] G^{\rm up} _q (b_1 ^\prime) \otimes G^{\rm up} _q (b_2 ^\prime),\] where $Q_{>0} := Q_{\ge 0} \setminus \{0\}$.
\end{enumerate}
\end{prop}

\begin{proof}
We see from the deinition $\overline{\vphantom{(}\cdot\vphantom{)}} = \Theta \circ (\overline{\vphantom{(}\cdot\vphantom{)}} \otimes \overline{\vphantom{(}\cdot\vphantom{)}})$ that \[\overline{G^{\rm up} _q (b_1) \otimes G^{\rm up} _q (b_2)} = \Theta(\overline{G^{\rm up} _q (b_1)} \otimes \overline{G^{\rm up} _q (b_2)}) = \Theta(G^{\rm up} _q (b_1) \otimes G^{\rm up} _q (b_2))\quad({\rm by\ the\ assumption}).\] Also, it follows from \cite[\S\S 24.1]{Lus} that \[\Theta(G^{\rm up} _q (b_1) \otimes G^{\rm up} _q (b_2)) \in G^{\rm up} _q (b_1) \otimes G^{\rm up} _q (b_2) + \sum_{\substack{{\rm wt}(b_1) - {\rm wt}(b_1 ^\prime) \in Q_{>0},\\{\rm wt}(b_2 ^\prime) - {\rm wt}(b_2) \in Q_{>0}}} \z[q, q^{-1}] G^{\rm up} _q (b_1 ^\prime) \otimes G^{\rm up} _q (b_2 ^\prime).\] Therefore, by arguments similar to those in \cite[\S\S 24.2 and \S\S 27.3]{Lus}, there exists a unique family of elements $\{\pi_{b_1 \otimes b_2, b_1 ^\prime \otimes b_2 ^\prime} \in q\z[q] \mid {\rm wt}(b_1) - {\rm wt}(b_1 ^\prime) \in Q_{>0},\ {\rm wt}(b_2 ^\prime) - {\rm wt}(b_2) \in Q_{>0}\}$ such that the element \[G^{\rm up} _q(b_1) \diamondsuit G^{\rm up} _q(b_2) := G^{\rm up} _q(b_1) \otimes G^{\rm up} _q(b_2) + \sum_{\substack{{\rm wt}(b_1) - {\rm wt}(b_1 ^\prime) \in Q_{>0},\\{\rm wt}(b_2 ^\prime) - {\rm wt}(b_2) \in Q_{>0}}} \pi_{b_1 \otimes b_2, b_1 ^\prime \otimes b_2 ^\prime} G^{\rm up} _q(b_1 ^\prime) \otimes G^{\rm up} _q(b_2 ^\prime)\] is invariant under the bar involution $\overline{\vphantom{(}\cdot\vphantom{)}}$. Since $\pi_{b_1 \otimes b_2, b_1 ^\prime \otimes b_2 ^\prime} \in q\z[q]$, the set $\{G^{\rm up} _q(b_1) \diamondsuit G^{\rm up} _q(b_2) \mid b_1 \in \mathcal{B}_1,\ b_2 \in \mathcal{B}_2\}$ forms an $A$-basis (resp., a $\q[q, q^{-1}]$-basis) of $L_1 \otimes L_2$ (resp., $V_1 ^\q \otimes V_2 ^\q$). Moreover, we obtain 
\begin{align*}
G^{\rm up} _q(b_1) \diamondsuit G^{\rm up} _q(b_2) &= \overline{G^{\rm up} _q(b_1) \diamondsuit G^{\rm up} _q(b_2)}\\ 
&\in \overline{L_1 \otimes L_2}\quad({\rm since}\ G^{\rm up} _q(b_1) \diamondsuit G^{\rm up} _q(b_2) \in L_1 \otimes L_2), 
\end{align*}
which implies that the set above also forms a $\q$-basis of $(V_1 ^\q \otimes V_2 ^\q) \cap (L_1 \otimes L_2) \cap (\overline{L_1 \otimes L_2})$. Also, by the natural $\q$-linear map $(V_1 ^\q \otimes V_2 ^\q) \cap (L_1 \otimes L_2) \cap (\overline{L_1 \otimes L_2}) \rightarrow (L_1 \otimes L_2)/ q(L_1 \otimes L_2)$, the element $G^{\rm up} _q(b_1) \diamondsuit G^{\rm up} _q(b_2)$ is sent to $b_1 \otimes b_2$. Since $\{b_1 \otimes b_2 \mid b_1 \in \mathcal{B}_1,\ b_2 \in \mathcal{B}_2\}$ forms a $\q$-basis of $(L_1 \otimes L_2)/ q(L_1 \otimes L_2)$, we deduce that this $\q$-linear map is an isomorphism; hence we obtain part (1). Now part (2) follows immediately from the equality $G^{\rm up} _q (b_1 \otimes b_2) = G^{\rm up} _q(b_1) \diamondsuit G^{\rm up} _q(b_2)$. This proves the proposition.
\end{proof}

For finite-dimensional $U_q (\mathfrak{g})$-modules $V_1, \ldots, V_r$, the comment following Proposition \ref{bar involution} implies that we obtain a (unique) upper global basis of $V_1 \otimes \cdots \otimes V_r$ by using Proposition \ref{global basis of tensor product} repeatedly. In this paper, we consider the upper global basis of $V_q(\lambda_1) \otimes \cdots \otimes V_q(\lambda_r)$ constructed in this way, where we take $(V_{q, \q} ^{\rm up} (\lambda_k), L^{\rm up}(\lambda_k), \overline{L^{\rm up}(\lambda_k)})$ as a balanced triple for each module $V_q(\lambda_k)$.

\section{Main result}

Now let us construct a basis of $H^0(Z_{\bf i}, \mathcal{L}_{{\bf i}, {\bf m}})$ that has a property similar to Proposition \ref{Kaveh} (1). For dominant integral weights $\lambda_1, \ldots, \lambda_r$, the dual $G$-module $(V(\lambda_1) \otimes \cdots \otimes V(\lambda_r))^\ast$ is isomorphic to $V(\lambda_r)^\ast \otimes \cdots \otimes V(\lambda_1)^\ast \simeq V(\lambda_r ^\ast) \otimes \cdots \otimes V(\lambda_1 ^\ast)$. Also, by identifying $(b_1 \otimes \cdots \otimes b_r)^\ast \in (\mathcal{B}(\lambda_1) \otimes \cdots \otimes \mathcal{B}(\lambda_r))^\ast$ with $b_r ^\ast \otimes \cdots \otimes b_1 ^\ast \in \mathcal{B}(\lambda_r ^\ast) \otimes \cdots \otimes \mathcal{B}(\lambda_1 ^\ast)$, we have an isomorphism of crystals $(\mathcal{B}(\lambda_1) \otimes \cdots \otimes \mathcal{B}(\lambda_r))^\ast \simeq \mathcal{B}(\lambda_r ^\ast) \otimes \cdots \otimes \mathcal{B}(\lambda_1 ^\ast)$. Through these identifications, the argument in Section 6 yields the specialization of an upper global basis $\{G^{\rm up}(b^\ast) \mid b \in \mathcal{B}(\lambda_1) \otimes \cdots \otimes \mathcal{B}(\lambda_r)\} \subset (V(\lambda_1) \otimes \cdots \otimes V(\lambda_r))^\ast$ at $q = 1$. Recalling that $\Phi_{{\bf i}, {\bf m}} ^\ast: (V(m_1 \varpi_{i_1}) \otimes \cdots \otimes V(m_r \varpi_{i_r}))^\ast \twoheadrightarrow V_{{\bf i}, {\bf m}} ^\ast = H^0(Z_{\bf i}, \mathcal{L}_{{\bf i}, {\bf m}})$ is the dual of the inclusion map $\Phi_{{\bf i}, {\bf m}}: V_{{\bf i}, {\bf m}} \hookrightarrow V(m_1 \varpi_{i_1}) \otimes \cdots \otimes V(m_r \varpi_{i_r})$, we set $G^{\rm up} _{{\bf i}, {\bf m}} (b) := \Phi_{{\bf i}, {\bf m}} ^\ast (G^{\rm up} (b^\ast)) \in H^0(Z_{\bf i}, \mathcal{L}_{{\bf i}, {\bf m}})$ for $b \in \mathcal{B}_{{\bf i}, {\bf m}}$. The following is the main result of this paper.

\vspace{2mm}\begin{thm}\label{main result}
Let ${\bf i} \in I^r$ be an arbitrary word, and ${\bf m} \in \z_{\ge 0} ^r$. Then, the set $\{G^{\rm up} _{{\bf i}, {\bf m}} (b) \mid b \in \mathcal{B}_{{\bf i}, {\bf m}}\}$ forms a $\c$-basis of $H^0(Z_{\bf i}, \mathcal{L}_{{\bf i}, {\bf m}})$, and it holds that $\Omega_{\bf i} (b) = - v_{\bf i} (G^{\rm up} _{{\bf i}, {\bf m}} (b)/\tau_{{\bf i}, {\bf m}})$ for all $b \in \mathcal{B}_{{\bf i}, {\bf m}}$.
\end{thm}\vspace{2mm}

Before proving Theorem \ref{main result}, we give some corollaries. The following is easily seen from the definitions.

\vspace{2mm}\begin{cor}
Let $\omega: \r \times \r^r \xrightarrow{\sim} \r \times \r^r$ denote the linear automorphism given by $\omega(k, {\bf a}) = (k, -{\bf a})$. Then, $\mathcal{S}_{{\bf i}, {\bf m}} = \omega(S(Z_{\bf i}, \mathcal{L}_{{\bf i}, {\bf m}}, v_{\bf i}, \tau_{{\bf i}, {\bf m}}))$, $\mathcal{C}_{{\bf i}, {\bf m}} = \omega(C(Z_{\bf i}, \mathcal{L}_{{\bf i}, {\bf m}}, v_{\bf i}, \tau_{{\bf i}, {\bf m}}))$, and $\Delta_{{\bf i}, {\bf m}} = -\Delta(Z_{\bf i}, \mathcal{L}_{{\bf i}, {\bf m}}, v_{\bf i}, \tau_{{\bf i}, {\bf m}})$.
\end{cor}\vspace{2mm}

\begin{cor}\label{finitely generated semigroup}
The following hold.
\begin{enumerate}
\item[{\rm (1)}] The sets $\mathcal{S}_{{\bf i}, {\bf m}}$ and $S(Z_{\bf i}, \mathcal{L}_{{\bf i}, {\bf m}}, v_{\bf i}, \tau_{{\bf i}, {\bf m}})$ are both finitely generated semigroups.
\item[{\rm (2)}] The real closed cones $\mathcal{C}_{{\bf i}, {\bf m}}$ and $C(Z_{\bf i}, \mathcal{L}_{{\bf i}, {\bf m}}, v_{\bf i}, \tau_{{\bf i}, {\bf m}})$ are both rational convex polyhedral cones, and the equality $S(Z_{\bf i}, \mathcal{L}_{{\bf i}, {\bf m}}, v_{\bf i}, \tau_{{\bf i}, {\bf m}}) = C(Z_{\bf i}, \mathcal{L}_{{\bf i}, {\bf m}}, v_{\bf i}, \tau_{{\bf i}, {\bf m}}) \cap (\z_{>0} \times \z^r)$ holds.
\item[{\rm (3)}] The generalized string polytope $\Delta_{{\bf i}, {\bf m}}$ and the Newton-Okounkov body $\Delta(Z_{\bf i}, \mathcal{L}_{{\bf i}, {\bf m}}, v_{\bf i}, \tau_{{\bf i}, {\bf m}})$ are both rational convex polytopes, and the equality $\Omega_{\bf i}(\mathcal{B}_{{\bf i}, {\bf m}}) = - \Delta(Z_{\bf i}, \mathcal{L}_{{\bf i}, {\bf m}}, v_{\bf i}, \tau_{{\bf i}, {\bf m}}) \cap \z^r$ holds.
\end{enumerate}
\end{cor}

\begin{proof}
Since $C(Z_{\bf i}, \mathcal{L}_{{\bf i}, {\bf m}}, v_{\bf i}, \tau_{{\bf i}, {\bf m}})$ is a convex cone (resp., $\Delta(Z_{\bf i}, \mathcal{L}_{{\bf i}, {\bf m}}, v_{\bf i}, \tau_{{\bf i}, {\bf m}})$ is a convex set), part (2) (resp., part (3)) is an immediate consequence of Corollary \ref{piecewise-linear inequalities 3} (resp., Corollary \ref{finite union of polytopes}). Also, part (1) follows from part (2) and Gordan's Lemma (see, for instance, \cite[Proposition 1.2.17]{CLS}). This proves the corollary.
\end{proof}

Similarly, we obtain the following.

\vspace{2mm}\begin{cor}
The set $\mathcal{S}_{\bf i}$ is a finitely generated semigroup, and the real closed cone $\mathcal{C}_{\bf i}$ is a rational convex polyhedral cone.
\end{cor}

\begin{proof}[Proof of Theorem \ref{main result}.]
Define an injection \[\tilde{\iota}_{s, s+1}: V(m_{s+1} \varpi_{i_{s+1}}) \otimes \cdots \otimes V(m_r \varpi_{i_r}) \hookrightarrow V(m_s \varpi_{i_s}) \otimes (V(m_{s+1} \varpi_{i_{s+1}}) \otimes \cdots \otimes V(m_r \varpi_{i_r}))\] by $\tilde{\iota}_{s, s+1}(v) := v_{m_s \varpi_{i_s}} \otimes v$ for $s = 1, \ldots, r-1$, and denote by $\tilde{\iota}_{s, s+1} ^\ast: (V(m_s \varpi_{i_s}) \otimes \cdots \otimes V(m_r \varpi_{i_r}))^\ast \twoheadrightarrow (V(m_{s+1} \varpi_{i_{s+1}}) \otimes \cdots \otimes V(m_r \varpi_{i_r}))^\ast$ the dual map. Recall that $\iota_{s, s+1}: V_{{\bf i}_{\ge s+1}, {\bf m}_{\ge s+1}} \hookrightarrow V_{{\bf i}_{\ge s}, {\bf m}_{\ge s}}$ is the restriction of $\tilde{\iota}_{s, s+1}$, and that $\iota_{s, s+1} ^\ast:H^0(Z_{{\bf i}_{\ge s}}, \mathcal{L}_{{\bf i}_{\ge s}, {\bf m}_{\ge s}}) \twoheadrightarrow H^0(Z_{{\bf i}_{\ge s+1}}, \mathcal{L}_{{\bf i}_{\ge s+1}, {\bf m}_{\ge s+1}})$ denotes its dual. For $b \in \mathcal{B}_{{\bf i}, {\bf m}}$, we write $\Omega_{\bf i} (b) = (a_1, \ldots, a_r)$, $- v_{\bf i} (G^{\rm up} _{{\bf i}, {\bf m}} (b)/\tau_{{\bf i}, {\bf m}}) = (a_1 ^\prime, \ldots, a_r ^\prime)$, and set 
\begin{align*}
&a'' _1 := \max\{a \in \z_{\ge 0} \mid F_{i_1} ^{(a)} G^{\rm up} (b^\ast) \neq 0\},\\
&a'' _2 := \max\{a \in \z_{\ge 0} \mid F_{i_2} ^{(a)} (\tilde{\iota}_{1, 2} ^\ast(F_{i_1} ^{(a'' _1)} G^{\rm up} (b^\ast))) \neq 0\},\\
&\ \vdots\\
&a'' _r := \max\{a \in \z_{\ge 0} \mid F_{i_r} ^{(a)} (\tilde{\iota}_{r-1, r} ^\ast(F_{i_{r-1}} ^{(a'' _{r-1})}(\cdots(\tilde{\iota}_{2, 3} ^\ast(F_{i_2} ^{(a'' _2)}(\tilde{\iota}_{1, 2} ^\ast(F_{i_1} ^{(a'' _1)} G^{\rm up} (b^\ast)))))\cdots))) \neq 0\}.
\end{align*}
We will prove that $(a_1, \ldots, a_r) = (a'' _1, \ldots, a'' _r)$, and that $(a_1 ^\prime, \ldots, a_r ^\prime) = (a'' _1, \ldots, a'' _r)$. First, it follows that 
\begin{align*}
a'' _1 &= \max\{a \in \z_{\ge 0} \mid \tilde{f}_{i_1} ^a b^\ast \neq 0\}\quad({\rm by\ the\ assertion\ in\ Proposition}\ \ref{property of upper global basis} (2)\ {\rm for}\ \varphi_i)\\
&= \max\{a \in \z_{\ge 0} \mid \tilde{e}_{i_1} ^a b \neq 0\}\quad({\rm by\ the\ definition\ of\ dual\ crystals})\\
&= a_1\quad({\rm by\ the\ definition\ of}\ \Omega_{\bf i}).
\end{align*}
Also, we deduce that
\begin{equation}\label{equations1}
\begin{aligned}
F_{i_1} ^{(a'' _1)} G^{\rm up} (b^\ast) &= G^{\rm up} (\tilde{f}_{i_1} ^{a'' _1} b^\ast)\\ 
&({\rm by\ the\ assertions\ in\ Proposition}\ \ref{property of upper global basis} (2)\ {\rm for}\ \varphi_i\ {\rm and}\ F_i ^{(\varphi_i(b))})\\
&= G^{\rm up} ((\tilde{e}_{i_1} ^{a'' _1} b)^\ast)\quad({\rm by\ the\ definition\ of\ dual\ crystals})\\
&= G^{\rm up} ((\tilde{e}_{i_1} ^{a_1} b)^\ast)\quad({\rm since}\ a'' _1 = a_1).
\end{aligned}
\end{equation}
For $b(2) \in \mathcal{B}_{{\bf i}_{\ge 2}, {\bf m}_{\ge 2}}$ defined in Definition \ref{generalized string parameterization}, we have $\tilde{e}_{i_1} ^{a_1} b = b_{m_1 \varpi_{i_1}} \otimes b(2)$. Let us identify $(\mathcal{B}(m_1 \varpi_{i_1}) \otimes (\mathcal{B}(m_2 \varpi_{i_2}) \otimes \cdots \otimes \mathcal{B}(m_r \varpi_{i_r})))^\ast$ with $(\mathcal{B}(m_2 \varpi_{i_2}) \otimes \cdots \otimes \mathcal{B}(m_r \varpi_{i_r}))^\ast \otimes \mathcal{B}(m_1 \varpi_{i_1})^\ast$ by $(b_1 \otimes b_2)^\ast \mapsto b_2 ^\ast \otimes b_1 ^\ast$. Then, Proposition \ref{global basis of tensor product} (2) implies that $G^{\rm up}((b_{m_1 \varpi_{i_1}} \otimes b(2))^\ast) = G^{\rm up}(b(2)^\ast \otimes b_{m_1 \varpi_{i_1}} ^\ast)$ is an element of \[G^{\rm up}(b(2)^\ast) \otimes G^{\rm up}(b_{m_1 \varpi_{i_1}} ^\ast) + \sum_{\substack{b_1 \in \mathcal{B}(m_1 \varpi_{i_1}) \setminus \{b_{m_1 \varpi_{i_1}}\},\\ b_2 \in \mathcal{B}(m_2 \varpi_{i_2}) \otimes \cdots \otimes \mathcal{B}(m_r \varpi_{i_r})}} \z G^{\rm up}(b_2 ^\ast) \otimes G^{\rm up}(b_1 ^\ast).\] Here, since $G^{\rm up}(b_1 ^\ast) (v_{m_1 \varpi_{i_1}}) = 0$ for $b_1 \neq b_{m_1 \varpi_{i_1}}$, we deduce that $\tilde{\iota}_{1, 2} ^\ast (G^{\rm up} (b_2 ^\ast) \otimes G^{\rm up}(b_1 ^\ast)) = 0$, and hence that 
\begin{align}\label{equations2}
\tilde{\iota}_{1, 2} ^\ast (G^{\rm up} ((\tilde{e}_{i_1} ^{a_1} b)^\ast)) = \tilde{\iota}_{1, 2} ^\ast (G^{\rm up}((b_{m_1 \varpi_{i_1}} \otimes b(2))^\ast)) = G^{\rm up}(b(2)^\ast).
\end{align}
From equations (\ref{equations1}) and (\ref{equations2}), it follows that 
\begin{align}\label{equations3}
\tilde{\iota}_{1, 2} ^\ast (F_{i_1} ^{(a'' _1)} G^{\rm up} (b^\ast)) = G^{\rm up}(b(2)^\ast).
\end{align} 
Therefore, we conclude by induction that $(a_1, \ldots, a_r) = (a'' _1, \ldots, a'' _r)$, and that 
\begin{equation}\label{equations4}
\begin{aligned}
&F_{i_r} ^{(a'' _r)} (\tilde{\iota}_{r-1, r} ^\ast(F_{i_{r-1}} ^{(a'' _{r-1})}(\cdots(\tilde{\iota}_{2, 3} ^\ast(F_{i_2} ^{(a'' _2)}(\tilde{\iota}_{1, 2} ^\ast(F_{i_1} ^{(a'' _1)} G^{\rm up} (b^\ast)))))\cdots)))\\ 
=\ &F_{i_r} ^{(a'' _r)} (\tilde{\iota}_{r-1, r} ^\ast(F_{i_{r-1}} ^{(a'' _{r-1})}(\cdots(\tilde{\iota}_{2, 3} ^\ast(F_{i_2} ^{(a'' _2)}(G^{\rm up}(b(2)^\ast))))\cdots)))\\
=\ &\cdots = F_{i_r} ^{(a'' _r)} G^{\rm up}(b(r)^\ast) = G^{\rm up} (b_{m_r \varpi_{i_r}} ^\ast).
\end{aligned}
\end{equation} 

Now recall that $\Phi_{{\bf i}_{\ge s}, {\bf m}_{\ge s}} ^\ast: (V(m_s \varpi_{i_s}) \otimes \cdots \otimes V(m_r \varpi_{i_r}))^\ast \twoheadrightarrow H^0(Z_{{\bf i}_{\ge s}}, \mathcal{L}_{{\bf i}_{\ge s}, {\bf m}_{\ge s}})$ is the dual of the inclusion map $\Phi_{{\bf i}_{\ge s}, {\bf m}_{\ge s}}: V_{{\bf i}_{\ge s}, {\bf m}_{\ge s}} \hookrightarrow V(m_s \varpi_{i_s}) \otimes \cdots \otimes V(m_r \varpi_{i_r})$. Since $\Phi_{{\bf i}_{\ge s}, {\bf m}_{\ge s}} \circ \iota_{s, s+1} = \tilde{\iota}_{s, s+1} \circ \Phi_{{\bf i}_{\ge s+1}, {\bf m}_{\ge s+1}}$, we have $\iota_{s, s+1} ^\ast \circ \Phi_{{\bf i}_{\ge s}, {\bf m}_{\ge s}} ^\ast = (\Phi_{{\bf i}_{\ge s}, {\bf m}_{\ge s}} \circ \iota_{s, s+1})^\ast = (\tilde{\iota}_{s, s+1} \circ \Phi_{{\bf i}_{\ge s+1}, {\bf m}_{\ge s+1}})^\ast = \Phi_{{\bf i}_{\ge s+1}, {\bf m}_{\ge s+1}} ^\ast \circ \tilde{\iota}_{s, s+1} ^\ast$. Therefore, we deduce that 
\begin{align*}
&F_{i_r} ^{(a'' _r)} (\iota_{r-1, r} ^\ast(F_{i_{r-1}} ^{(a'' _{r-1})}(\cdots(\iota_{2, 3} ^\ast(F_{i_2} ^{(a'' _2)}(\iota_{1, 2} ^\ast(F_{i_1} ^{(a'' _1)} G^{\rm up} _{{\bf i}, {\bf m}}(b)))))\cdots)))\\
=\ &F_{i_r} ^{(a'' _r)} (\iota_{r-1, r} ^\ast(F_{i_{r-1}} ^{(a'' _{r-1})}(\cdots(\iota_{2, 3} ^\ast(F_{i_2} ^{(a'' _2)}(\iota_{1, 2} ^\ast \circ \Phi_{{\bf i}, {\bf m}} ^\ast (F_{i_1} ^{(a'' _1)}G^{\rm up}(b^\ast)))))\cdots)))\\
=\ &F_{i_r} ^{(a'' _r)} (\iota_{r-1, r} ^\ast(F_{i_{r-1}} ^{(a'' _{r-1})}(\cdots(\iota_{2, 3} ^\ast \circ \Phi_{{\bf i}_{\ge 2}, {\bf m}_{\ge 2}} ^\ast (F_{i_2} ^{(a'' _2)}(\tilde{\iota}_{1, 2} ^\ast(F_{i_1} ^{(a'' _1)} G^{\rm up}(b^\ast)))))\cdots)))\\
= &\cdots = \Phi_{{\bf i}_{\ge r}, {\bf m}_{\ge r}} ^\ast(F_{i_r} ^{(a'' _r)} (\tilde{\iota}_{r-1, r} ^\ast(F_{i_{r-1}} ^{(a'' _{r-1})}(\cdots(\tilde{\iota}_{2, 3} ^\ast(F_{i_2} ^{(a'' _2)}(\tilde{\iota}_{1, 2} ^\ast(F_{i_1} ^{(a'' _1)} G^{\rm up} (b^\ast)))))\cdots))))\\ 
=\ &\Phi_{{\bf i}_{\ge r}, {\bf m}_{\ge r}} ^\ast(G^{\rm up} (b_{m_r \varpi_{i_r}} ^\ast))\quad({\rm by\ equation}\ (\ref{equations4}))\\ 
\neq\ &0\quad({\rm since}\ \Phi_{{\bf i}_{\ge r}, {\bf m}_{\ge r}} ^\ast(G^{\rm up} (b_{m_r \varpi_{i_r}} ^\ast))(v_{m_r \varpi_{i_r}}) = G^{\rm up} (b_{m_r \varpi_{i_r}} ^\ast)(v_{m_r \varpi_{i_r}}) = 1),
\end{align*}
and hence that $F_{i_s} ^{(a'' _s)} (\iota_{s-1, s} ^\ast(F_{i_{s-1}} ^{(a'' _{s-1})}(\cdots(\iota_{2, 3} ^\ast(F_{i_2} ^{(a'' _2)}(\iota_{1, 2} ^\ast(F_{i_1} ^{(a'' _1)} G^{\rm up} _{{\bf i}, {\bf m}}(b)))))\cdots))) \neq 0$ for all $1 \le s \le r$. Similarly, we see that
\begin{align*}
&F_{i_s} ^{(a'' _s +1)} (\iota_{s-1, s} ^\ast(F_{i_{s-1}} ^{(a'' _{s-1})}(\cdots(\iota_{2, 3} ^\ast(F_{i_2} ^{(a'' _2)}(\iota_{1, 2} ^\ast(F_{i_1} ^{(a'' _1)} G^{\rm up} _{{\bf i}, {\bf m}}(b)))))\cdots)))\\
=\ &\Phi_{{\bf i}_{\ge s}, {\bf m}_{\ge s}} ^\ast(F_{i_s} ^{(a'' _s +1)} (\tilde{\iota}_{s-1, s} ^\ast(F_{i_{s-1}} ^{(a'' _{s-1})}(\cdots(\tilde{\iota}_{2, 3} ^\ast(F_{i_2} ^{(a'' _2)}(\tilde{\iota}_{1, 2} ^\ast(F_{i_1} ^{(a'' _1)} G^{\rm up} (b^\ast)))))\cdots))))\\ 
=\ &\Phi_{{\bf i}_{\ge s}, {\bf m}_{\ge s}} ^\ast(0)\quad({\rm by\ the\ definition\ of}\ a'' _s)\\ 
=\ &0.
\end{align*}
From these, we obtain \[a'' _s = \max\{a \in \z_{\ge 0} \mid F_{i_s} ^{(a)} (\iota_{s-1, s} ^\ast(F_{i_{s-1}} ^{(a'' _{s-1})}(\cdots(\iota_{2, 3} ^\ast(F_{i_2} ^{(a'' _2)}(\iota_{1, 2} ^\ast(F_{i_1} ^{(a'' _1)} G^{\rm up} _{{\bf i}, {\bf m}}(b)))))\cdots))) \neq 0\}\] for all $1 \le s \le r$. Also, Proposition \ref{val,Chevalley} and the remark following it imply that \[a' _s = \max\{a \in \z_{\ge 0} \mid F_{i_s} ^{(a)} (\iota_{s-1, s} ^\ast(F_{i_{s-1}} ^{(a' _{s-1})}(\cdots(\iota_{2, 3} ^\ast(F_{i_2} ^{(a' _2)}(\iota_{1, 2} ^\ast(F_{i_1} ^{(a' _1)} G^{\rm up} _{{\bf i}, {\bf m}}(b)))))\cdots))) \neq 0\}\] for all $1 \le s \le r$. By using these, we conclude by induction that $(a_1 ^\prime, \ldots, a_r ^\prime) = (a'' _1, \ldots, a'' _r)$. Thus, we have proved that $\Omega_{\bf i} (b) = - v_{\bf i} (G^{\rm up} _{{\bf i}, {\bf m}} (b)/\tau_{{\bf i}, {\bf m}})$ for all $b \in \mathcal{B}_{{\bf i}, {\bf m}}$. 

Finally, it follows from Proposition \ref{generalized parameterization} (2) that $v_{\bf i} (G^{\rm up} _{{\bf i}, {\bf m}} (b)/\tau_{{\bf i}, {\bf m}})$, $b \in \mathcal{B}_{{\bf i}, {\bf m}}$, are all distinct. Therefore, Proposition \ref{prop1,val} (1) implies that $G^{\rm up} _{{\bf i}, {\bf m}} (b)$, $b \in \mathcal{B}_{{\bf i}, {\bf m}}$, are linearly independent. From this, we conclude that $\{G^{\rm up} _{{\bf i}, {\bf m}} (b) \mid b \in \mathcal{B}_{{\bf i}, {\bf m}}\}$ forms a $\c$-basis of $H^0 (Z_{\bf i}, \mathcal{L}_{{\bf i}, {\bf m}})$, since the complex dimension of $H^0 (Z_{\bf i}, \mathcal{L}_{{\bf i}, {\bf m}})$ is equal to the cardinality of $\mathcal{B}_{{\bf i}, {\bf m}}$ by Proposition \ref{generalized Demazure crystal} (3). This completes the proof of the theorem.
\end{proof}

\section{Similar results for another valuation}

In this section, we treat a certain valuation on $\c(Z_{\bf i})$ and a certain parameterization for $\mathcal{B}_{{\bf i}, {\bf m}}$, which are different from the $v_{\bf i}$ and $\Omega_{\bf i}$, respectively. Assume that ${\bf i} = (i_1, \ldots, i_r)$ is a reduced word for $w \in W$, and write $s_i := \exp(E_i) \exp(-F_i) \exp(E_i) \in P_i$ for $i \in I$. Then, we can regard $s_{i_1}U_{i_1}^- \times \cdots \times s_{i_r}U_{i_r}^-$ as an affine open neighborhood of $(s_{i_1}, \ldots, s_{i_r}) \bmod B^r$ in $Z_{\bf i}$ by 
\begin{align*}
s_{i_1}U_{i_1}^- \times \cdots \times s_{i_r}U_{i_r}^- \hookrightarrow Z_{\bf i},\ (s_{i_1} u_1, \ldots, s_{i_r} u_r) \mapsto (s_{i_1} u_1, \ldots, s_{i_r} u_r) \bmod B^r.
\end{align*} 
Consider the isomorphism of varieties $\c^r \xrightarrow{\sim} s_{i_1}U_{i_1}^- \times \cdots \times s_{i_r}U_{i_r}^-$ given by \[(t_1 ^\prime, \ldots, t_r ^\prime) \mapsto (s_{i_1}\exp(t_1 ^\prime F_{i_1}), \ldots, s_{i_r}\exp(t_r ^\prime F_{i_r}));\] using this isomorphism, we regard $\c(Z_{\bf i}) = \c(s_{i_1}U_{i_1}^- \times \cdots \times s_{i_r}U_{i_r}^-)$ as the rational function field $\c(t_1 ^\prime, \ldots, t_r ^\prime)$. Let us define a valuation $v_{\bf i} ^\prime$ on $\c(Z_{\bf i})$ to be the highest term valuation on $\c(t_1 ^\prime, \ldots, t_r ^\prime)$ with respect to the lexicographic order $<$ on $\z^r$ (see Example \ref{highest term valuation}). For $1 \le k \le r$, denote by $\widetilde{V}(m_k \varpi_{i_k})$ the $\c$-subspace of $V(m_k \varpi_{i_k})$ spanned by weight vectors whose weight is not equal to that of $s_{i_1} s_{i_2} \cdots s_{i_k} v_{m_k \varpi_{i_k}}$. Since $s_{i_1} s_{i_2} \cdots s_{i_k} v_{m_k \varpi_{i_k}}$ is an extremal weight vector, we deduce that \[V(m_k \varpi_{i_k}) = \widetilde{V}(m_k \varpi_{i_k}) \oplus \c (s_{i_1} s_{i_2} \cdots s_{i_k} v_{m_k \varpi_{i_k}}).\] Define $\tilde{\tau}_{{\bf i}, {\bf m}} ^\prime \in (V(m_1 \varpi_{i_1}) \otimes \cdots \otimes V(m_r \varpi_{i_r}))^\ast$ by \[\tilde{\tau}_{{\bf i}, {\bf m}} ^\prime (s_{i_1} v_{m_1 \varpi_{i_1}} \otimes s_{i_1} s_{i_2} v_{m_2 \varpi_{i_2}} \otimes \cdots \otimes  s_{i_1} s_{i_2} \cdots s_{i_r} v_{m_r \varpi_{i_r}}) = 1,\] and by $\tilde{\tau}_{{\bf i}, {\bf m}} ^\prime (v_1 \otimes \cdots \otimes v_r) = 0$ if there exists $1 \le k \le r$ such that $v_k \in \widetilde{V}(m_k \varpi_{i_k})$. Let us set $\tau_{{\bf i}, {\bf m}} ^\prime := \Phi_{{\bf i}, {\bf m}} ^\ast (\tilde{\tau}_{{\bf i}, {\bf m}} ^\prime) \in H^0(Z_{\bf i}, \mathcal{L}_{{\bf i}, {\bf m}})$. In this section, we study the Newton-Okounkov body $\Delta(Z_{\bf i}, \mathcal{L}_{{\bf i}, {\bf m}}, v_{\bf i} ^\prime, \tau_{{\bf i}, {\bf m}} ^\prime)$. Let us prove a lemma similar to Lemma \ref{prop, lowest}.

\vspace{2mm}\begin{lem}
\begin{enumerate}
\item[{\rm (1)}] The section $\tau_{{\bf i}, {\bf m}} ^\prime$ does not vanish on $s_{i_1}U_{i_1}^- \times \cdots \times s_{i_r}U_{i_r}^-$ $(\hookrightarrow Z_{\bf i})$. In particular, $\sigma/\tau_{{\bf i}, {\bf m}} ^\prime \in \c[t_1 ^\prime, \ldots, t_r ^\prime]$ for all $\sigma \in H^0(Z_{\bf i}, \mathcal{L}_{{\bf i}, {\bf m}})$.
\item[{\rm (2)}] It holds that $E_{i_1} \tau_{{\bf i}, {\bf m}} ^\prime = 0$ in $H^0(Z_{\bf i}, \mathcal{L}_{{\bf i}, {\bf m}})$.
\end{enumerate}
\end{lem}

\begin{proof}
Denote by $U_\beta$ the root subgroup corresponding to a root $\beta$. Let us set $\beta_1 := \a_{i_1}$, $\beta_2 := s_{i_1}(\a_{i_2}), \ldots, \beta_r := s_{i_1}s_{i_2} \cdots s_{i_{r-1}}(\a_{i_r})$. It is well-known that these roots $\beta_1, \ldots, \beta_r$ are positive (see \cite[\S\S 10.2]{Hum}). Recall that $\Phi_{{\bf i}, {\bf m}} ^\ast$ is identical to the surjection $\Psi_{{\bf i}, {\bf m}} ^\ast$ defined in \S\S 2.2, and that \[\Psi_{{\bf i}, {\bf m}}((s_{i_1}u_1, \ldots, s_{i_r}u_r) \bmod B^r) = \c(v' _{u_1} \otimes v' _{u_1, u_2} \otimes \cdots \otimes v' _{u_1, \ldots, u_r})\] for $u_1 \in U_{i_1} ^-, \ldots, u_r \in U_{i_r} ^-$, where we set $v' _{u_1, \ldots, u_k} := (s_{i_1} u_1)(s_{i_2} u_2) \cdots (s_{i_k} u_k) v_{m_k \varpi_{i_k}}$ for $1 \le k \le r$. Since $(s_{i_1} U_{i_1} ^-)(s_{i_2} U_{i_2} ^-) \cdots (s_{i_k} U_{i_k} ^-) = U_{\beta_1}U_{\beta_2} \cdots U_{\beta_k} s_{i_1}s_{i_2} \cdots s_{i_k}$, it follows that \[v' _{u_1, \ldots, u_k} \in s_{i_1} s_{i_2} \cdots s_{i_k} v_{m_k \varpi_{i_k}} + \widetilde{V}(m_k \varpi_{i_k}).\] Hence we deduce from the definition of $\tilde{\tau}_{{\bf i}, {\bf m}} ^\prime$ that $\tilde{\tau}_{{\bf i}, {\bf m}} ^\prime(v' _{u_1} \otimes v' _{u_1, u_2} \otimes \cdots \otimes v' _{u_1, \ldots, u_r}) = 1$, which implies part (1) since $\tau_{{\bf i}, {\bf m}} ^\prime = \Phi_{{\bf i}, {\bf m}} ^\ast (\tilde{\tau}_{{\bf i}, {\bf m}} ^\prime)$. Also, the element $v' _{u_1, \ldots, u_k}$ is a linear combination of weight vectors whose weight is an element of \[s_{i_1} s_{i_2} \cdots s_{i_k} (m_k \varpi_{i_k}) + \sum_{1 \le j \le k} \z_{\ge 0} \beta_j.\] Since $\beta_1, \ldots, \beta_k$ are positive roots, we conclude that $E_{i_1} v' _{u_1, \ldots, u_k} \in \widetilde{V}(m_k \varpi_{i_k})$. From this, it follows that 
\begin{align*}
&(E_{i_1} \tilde{\tau}_{{\bf i}, {\bf m}} ^\prime)(v' _{u_1} \otimes v' _{u_1, u_2} \otimes \cdots \otimes v' _{u_1, \ldots, u_r})\\ 
=\ &-\tilde{\tau}_{{\bf i}, {\bf m}} ^\prime(E_{i_1}(v' _{u_1} \otimes v' _{u_1, u_2} \otimes \cdots \otimes v' _{u_1, \ldots, u_r}))\\ 
=\ &-\sum_{1 \le k \le r} \tilde{\tau}_{{\bf i}, {\bf m}} ^\prime(v' _{u_1} \otimes \cdots \otimes E_{i_1}v' _{u_1, \ldots, u_k} \otimes \cdots \otimes v' _{u_1, \ldots, u_r}) = 0, 
\end{align*}
and hence that the section $E_{i_1} \tau_{{\bf i}, {\bf m}} ^\prime$ is identically zero on $s_{i_1}U_{i_1}^- \times \cdots \times s_{i_r}U_{i_r}^-$. Now, since $Z_{\bf i}$ is irreducible, we conclude part (2). This proves the lemma.
\end{proof}

For $b \in \mathcal{B}_{{\bf i}, {\bf m}}$, define $\Omega_{\bf i} ^\prime(b) = (a_1 ^\prime, \ldots, a_r ^\prime) \in \z_{\ge 0} ^r$ by \[a_k ^\prime := \max\{a \in \z_{\ge 0} \mid \tilde{f}_{i_k} ^a(b(k)) \neq 0\}\] for $k = 1, \ldots, r$, where $b(k) \in \mathcal{B}_{{\bf i}_{\ge k}, {\bf m}_{\ge k}}$ is the element defined in Definition \ref{generalized string parameterization}. Replacing $\Omega_{\bf i}$ by $\Omega_{\bf i} ^\prime$ in the definitions of $\mathcal{S}_{{\bf i}, {\bf m}}$, $\mathcal{C}_{{\bf i}, {\bf m}}$, and $\Delta_{{\bf i}, {\bf m}}$, we obtain $\mathcal{S}_{{\bf i}, {\bf m}} ^\prime \subset \z_{> 0} \times \z^r$, $\mathcal{C}_{{\bf i}, {\bf m}} ^\prime \subset \r_{\ge 0} \times \r^r$, and $\Delta_{{\bf i}, {\bf m}} ^\prime \subset \r^r$. The following is the main result of this section. 

\vspace{2mm}\begin{thm}\label{main result 2}
Let ${\bf i} \in I^r$ be an arbitrary reduced word.
\begin{enumerate}
\item[{\rm (1)}] For all ${\bf m} \in \z^r _{\ge 0}$ and $b \in \mathcal{B}_{{\bf i}, {\bf m}}$, it follows that $\Omega_{\bf i} ^\prime(b) = -v_{\bf i} ^\prime(G^{\rm up} _{{\bf i}, {\bf m}} (b)/\tau_{{\bf i}, {\bf m}} ^\prime)$. In particular, $\mathcal{S}_{{\bf i}, {\bf m}} ^\prime = \omega(S(Z_{\bf i}, \mathcal{L}_{{\bf i}, {\bf m}}, v_{\bf i} ^\prime, \tau_{{\bf i}, {\bf m}} ^\prime))$, $\mathcal{C}_{{\bf i}, {\bf m}} ^\prime = \omega(C(Z_{\bf i}, \mathcal{L}_{{\bf i}, {\bf m}}, v_{\bf i} ^\prime, \tau_{{\bf i}, {\bf m}} ^\prime))$, and $\Delta_{{\bf i}, {\bf m}} ^\prime = -\Delta(Z_{\bf i}, \mathcal{L}_{{\bf i}, {\bf m}}, v_{\bf i} ^\prime, \tau_{{\bf i}, {\bf m}} ^\prime)$, where $\omega: \r \times \r^r \xrightarrow{\sim} \r \times \r^r$ is the linear automorphism given by $\omega(k, {\bf a}) = (k, -{\bf a})$.
\item[{\rm (2)}] There exist explicit unimodular $r \times r$-matrices $A$ and $B$ such that $\mathcal{S}_{{\bf i}, {\bf m}} = \Xi_{A, B}(\mathcal{S}_{{\bf i}, {\bf m}} ^\prime)$, $\mathcal{C}_{{\bf i}, {\bf m}} = \Xi_{A, B}(\mathcal{C}_{{\bf i}, {\bf m}} ^\prime)$, and $\Delta_{{\bf i}, {\bf m}} = A \Delta_{{\bf i}, {\bf m}} ^\prime + B {\bf m}$ for all ${\bf m} \in \z^r _{\ge 0}$, where $\Xi_{A, B}: \r \times \r^r \xrightarrow{\sim} \r \times \r^r$ is the linear automorphism given by $\Xi_{A, B}(k, {\bf a}) = (k, A {\bf a} + k B {\bf m})$.
\end{enumerate}
\end{thm}\vspace{2mm}

The following is an immediate consequence of Corollary \ref{finitely generated semigroup}.

\vspace{2mm}\begin{cor}
The following hold.
\begin{enumerate}
\item[{\rm (1)}] The sets $\mathcal{S}_{{\bf i}, {\bf m}} ^\prime$ and $S(Z_{\bf i}, \mathcal{L}_{{\bf i}, {\bf m}}, v_{\bf i} ^\prime, \tau_{{\bf i}, {\bf m}} ^\prime)$ are both finitely generated semigroups.
\item[{\rm (2)}] The real closed cones $\mathcal{C}_{{\bf i}, {\bf m}} ^\prime$ and $C(Z_{\bf i}, \mathcal{L}_{{\bf i}, {\bf m}}, v_{\bf i} ^\prime, \tau_{{\bf i}, {\bf m}} ^\prime)$ are both rational convex polyhedral cones, and the equality $S(Z_{\bf i}, \mathcal{L}_{{\bf i}, {\bf m}}, v_{\bf i} ^\prime, \tau_{{\bf i}, {\bf m}} ^\prime) = C(Z_{\bf i}, \mathcal{L}_{{\bf i}, {\bf m}}, v_{\bf i} ^\prime, \tau_{{\bf i}, {\bf m}} ^\prime) \cap (\z_{>0} \times \z^r)$ holds.
\item[{\rm (3)}] The set $\Delta_{{\bf i}, {\bf m}} ^\prime$ and the Newton-Okounkov body $\Delta(Z_{\bf i}, \mathcal{L}_{{\bf i}, {\bf m}}, v_{\bf i} ^\prime, \tau_{{\bf i}, {\bf m}} ^\prime)$ are both rational convex polytopes, and the equality $\Omega_{\bf i} ^\prime(\mathcal{B}_{{\bf i}, {\bf m}}) = - \Delta(Z_{\bf i}, \mathcal{L}_{{\bf i}, {\bf m}}, v_{\bf i} ^\prime, \tau_{{\bf i}, {\bf m}} ^\prime) \cap \z^r$ holds.
\end{enumerate}
\end{cor}\vspace{2mm}

In order to prove Theorem \ref{main result 2}, we need some lemmas.

\vspace{2mm}\begin{lem}\label{another parameterization}
If $\Omega_{\bf i}(b) = (a_1, \ldots, a_r)$ and $\Omega_{\bf i} ^\prime(b) = (a_1 ^\prime, \ldots, a_r ^\prime)$ for $b \in \mathcal{B}_{{\bf i}, {\bf m}}$, then $a_k = m_k - a_k ^\prime + \sum_{k < j \le r} (\delta_{i_k, i_j} m_j - c_{i_k, i_j} a_j)$ for all $1 \le k \le r;$ here, $(c_{i, j})_{i, j \in I}$ denotes the Cartan matrix of $\mathfrak{g}$.
\end{lem}

\begin{proof}
Since $b_{m_k \varpi_{i_k}} \otimes b(k+1) = \tilde{e}_{i_k} ^{a_k}b(k)$, the number $a_k + a_k ^\prime$ $(= \varepsilon_{i_k} (b(k)) + \varphi_{i_k}(b(k)))$ is equal to 
\begin{align*}
\varepsilon_{i_k} (b_{m_k \varpi_{i_k}} \otimes b(k+1)) + \varphi_{i_k}(b_{m_k \varpi_{i_k}} \otimes b(k+1)) &= \langle {\rm wt}(b_{m_k \varpi_{i_k}} \otimes b(k+1)), h_{i_k}\rangle\\ 
&({\rm by\ Lemma}\ \ref{length of string}\ {\rm since}\ \tilde{e}_{i_k}(b_{m_k \varpi_{i_k}} \otimes b(k+1)) = 0).
\end{align*}
Now the proof of Proposition \ref{generalized parameterization} (1) implies that \[b(k+1) = \tilde{f}_{i_{k+1}} ^{a_{k+1}} (b_{m_{k+1} \varpi_{i_{k+1}}} \otimes \tilde{f}_{i_{k+2}} ^{a_{k+2}} (b_{m_{k+2} \varpi_{i_{k+2}}} \otimes \cdots \otimes \tilde{f}_{i_{r-1}} ^{a_{r-1}} (b_{m_{r-1} \varpi_{i_{r-1}}} \otimes \tilde{f}_{i_r} ^{a_r} (b_{m_r \varpi_{i_r}}))\cdots));\] hence it follows that \[\langle {\rm wt}(b_{m_k \varpi_{i_k}} \otimes b(k+1)), h_{i_k}\rangle = m_k + \sum_{k < j \le r} (\delta_{i_k, i_j} m_j - c_{i_k, i_j} a_j).\] From these, the assertion of the lemma follows immediately. 
\end{proof}

Let $\r^{2r} \rightarrow \r^r$, $(a_1 ^\prime, \ldots, a_r ^\prime, m_1, \ldots, m_r) \mapsto (a_1, \ldots, a_r)$, be the linear map given by 
\begin{align*}
&a_r := m_r - a_r ^\prime,\\ 
&a_{r-1} := m_{r-1} - a_{r-1} ^\prime + \delta_{i_{r-1}, i_r} m_r - c_{i_{r-1}, i_r} a_r,\\
&\ \vdots\\
&a_1 := m_1 - a_1 ^\prime + \sum_{1 < j \le r} \delta_{i_1, i_j} m_j - \sum_{1 < j \le r} c_{i_1, i_j} a_j,
\end{align*}
and let $A$ and $B$ be $r \times r$-matrices given by \[\begin{pmatrix}
a_1\\
\vdots\\
a_r
\end{pmatrix}
= A
\begin{pmatrix}
a_1 ^\prime\\
\vdots\\
a_r ^\prime
\end{pmatrix}
+ B
\begin{pmatrix}
m_1\\
\vdots\\
m_r
\end{pmatrix}.\] Then, $A$ (resp., $B$) is an upper triangular matrix with diagonal entries $-1$ (resp., $1$), and all the entries of $A$ and $B$ are polynomials in $c_{i, j}$, $i, j \in I$, with coefficients in $\z$. In particular, these matrices are unimodular. We see from Lemma \ref{another parameterization} that $\Omega_{\bf i} (b) = A \Omega_{\bf i} ^\prime (b) + B {\bf m}$ for all ${\bf m} \in \z_{\ge 0} ^r$ and $b \in \mathcal{B}_{{\bf i}, {\bf m}}$. This proves part (2) of Theorem \ref{main result 2}. Moreover, the generalized string parameterization $\Omega_{\bf i}(b)$ can be reconstructed from $\Omega_{\bf i} ^\prime(b)$. Hence we conclude the following from Proposition \ref{generalized parameterization} (2).

\vspace{2mm}\begin{lem}
If $b, b^\prime \in \mathcal{B}_{{\bf i}, {\bf m}}$ are such that $b \neq b^\prime$, then $\Omega_{\bf i} ^\prime(b) \neq \Omega_{\bf i} ^\prime(b^\prime)$.
\end{lem}

\begin{proof}[Proof of Theorem \ref{main result 2} (1).]
For ${\bf m} \in \z_{\ge 0} ^r$ and $b \in \mathcal{B}_{{\bf i}, {\bf m}}$, we write $\Omega_{\bf i}(b) = (a_1, \ldots, a_r)$, $\Omega_{\bf i} ^\prime(b) = (a_1 ^\prime, \ldots, a_r ^\prime)$ and $-v_{\bf i} ^\prime (G^{\rm up} _{{\bf i}, {\bf m}} (b)/\tau_{{\bf i}, {\bf m}} ^\prime) = (a'' _1, \ldots, a'' _r)$. Define an injection $\iota_{k, k+1} ^\prime: V_{{\bf i}_{\ge k+1}, {\bf m}_{\ge k+1}} \hookrightarrow V_{{\bf i}_{\ge k}, {\bf m}_{\ge k}}$ by $v \mapsto s_{i_k} (v_{m_k \varpi_{i_k}} \otimes v)$, and let $(\iota_{k, k+1} ^\prime)^\ast: H^0(Z_{{\bf i}_{\ge k}}, \mathcal{L}_{{\bf i}_{\ge k}, {\bf m}_{\ge k}}) \twoheadrightarrow H^0(Z_{{\bf i}_{\ge k+1}}, \mathcal{L}_{{\bf i}_{\ge k+1}, {\bf m}_{\ge k+1}})$ denote the dual map; note that $\iota_{k, k+1} ^\prime = s_{i_k} \circ \iota_{k, k+1}$, and hence that $(\iota_{k, k+1} ^\prime)^\ast = (s_{i_k} \circ \iota_{k, k+1})^\ast = \iota_{k, k+1} ^\ast \circ s_{i_k}$. If we set $U_i := \exp(\c E_i)$ for $i \in I$, then it follows from the equality $s_{i_1} U_{i_1} ^- = U_{i_1} s_{i_1}$ that the root subgroup $U_{i_1}$ acts on $s_{i_1}U_{i_1}^- \times \cdots \times s_{i_r}U_{i_r}^-$ on the left by \[u \cdot (s_{i_1} u_1, \ldots, s_{i_r} u_r) := (u s_{i_1} u_1, s_{i_2} u_2, \ldots, s_{i_r} u_r)\] for $u \in U_{i_1}$, $u_1 \in U_{i_1} ^-, \ldots, u_r \in U_{i_r} ^-$. This induces left actions of $U_{i_1}$ and ${\rm Lie}(U_{i_1}) = \c E_{i_1}$ on $\c[t_1 ^\prime, \ldots, t_r ^\prime]$ ($= \c[s_{i_1}U_{i_1}^- \times \cdots \times s_{i_r}U_{i_r}^-]$); since we have $\exp(t E_{i_1}) s_{i_1} \exp(t_1 ^\prime F_{i_1}) = s_{i_1} \exp((t + t_1 ^\prime) F_{i_1})$ for $t, t_1 ^\prime \in \c$, these actions are given by \begin{align*}
&\exp(t E_{i_1}) \cdot f(t_1 ^\prime, \ldots, t_r ^\prime) = f(t_1 ^\prime - t, \ldots, t_r ^\prime),\ {\rm and\ hence}\\
&E_{i_1} \cdot f(t_1 ^\prime, \ldots, t_r ^\prime) = - \frac{\partial}{\partial t_1 ^\prime} f(t_1 ^\prime, \ldots, t_r ^\prime)
\end{align*}  
for $t \in \c$ and $f(t_1 ^\prime, \ldots, t_r ^\prime) \in \c[t_1 ^\prime, \ldots, t_r ^\prime]$. Hence, by the same arguments as in the proof of Proposition \ref{val,Chevalley} and in the remark following it, we obtain that
\begin{align*}
&a'' _1 = \max\{a \in \z_{\ge 0} \mid E_{i_1} ^{(a)} G^{\rm up} _{{\bf i}, {\bf m}} (b) \neq 0\},\\
&a'' _2 = \max\{a \in \z_{\ge 0} \mid E_{i_2} ^{(a)} ((\iota_{1, 2} ^\prime)^\ast(E_{i_1} ^{(a'' _1)} G^{\rm up} _{{\bf i}, {\bf m}} (b))) \neq 0\},\\
&\ \vdots\\
&a'' _r = \max\{a \in \z_{\ge 0} \mid E_{i_r} ^{(a)} ((\iota_{r-1, r} ^\prime)^\ast(E_{i_{r-1}} ^{(a'' _{r-1})}(\cdots((\iota_{2, 3} ^\prime)^\ast(E_{i_2} ^{(a'' _2)}((\iota_{1, 2} ^\prime)^\ast(E_{i_1} ^{(a'' _1)} G^{\rm up} _{{\bf i}, {\bf m}} (b)))))\cdots))) \neq 0\}.
\end{align*}

Now let us define an injection \[\tilde{\iota}'_{k, k+1}: V(m_{k+1} \varpi_{i_{k+1}}) \otimes \cdots \otimes V(m_r \varpi_{i_r}) \hookrightarrow V(m_k \varpi_{i_k}) \otimes (V(m_{k+1} \varpi_{i_{k+1}}) \otimes \cdots \otimes V(m_r \varpi_{i_r}))\] by $\tilde{\iota}'_{k, k+1}(v) := s_{i_k}(v_{m_k \varpi_{i_k}} \otimes v)$ for $1 \le k \le r-1$, and denote by $(\tilde{\iota}'_{k, k+1})^\ast: (V(m_k \varpi_{i_k}) \otimes \cdots \otimes V(m_r \varpi_{i_r}))^\ast \twoheadrightarrow (V(m_{k+1} \varpi_{i_{k+1}}) \otimes \cdots \otimes V(m_r \varpi_{i_r}))^\ast$ the dual map. Note that $\iota'_{k, k+1}: V_{{\bf i}_{\ge k+1}, {\bf m}_{\ge k+1}} \hookrightarrow V_{{\bf i}_{\ge k}, {\bf m}_{\ge k}}$ is the restriction of $\tilde{\iota}'_{k, k+1}$. Here, we deduce that 
\begin{align*}
a' _1 &= \varphi_{i_1}(b)\quad({\rm by\ the\ definition\ of}\ \Omega_{\bf i} ^\prime)\\ 
&= \varepsilon_{i_1}(b^\ast)\quad({\rm by\ the\ definition\ of\ dual\ crystals});
\end{align*}
hence the assertion of Proposition \ref{property of upper global basis} (2) for $E_i ^{(\varepsilon_i(b))}$ implies that $E_{i_1} ^{(a' _1)} G^{\rm up} (b^\ast) = G^{\rm up} (\tilde{e}_{i_1} ^{a' _1}b^\ast)$. In addition, the assertion of Proposition \ref{property of upper global basis} (2) for $\varepsilon_i(b)$ implies that $E_{i_1} G^{\rm up} (\tilde{e}_{i_1} ^{a' _1}b^\ast) = 0$. Therefore, by the standard representation theory of $\mathfrak{sl}_2(\c)$ (see \cite[Section 7 and \S\S 21.2 (6)]{Hum}), we see that \[s_{i_1} G^{\rm up} (\tilde{e}_{i_1} ^{a' _1}b^\ast) = c F_{i_1} ^{(\varphi_{i_1}(\tilde{e}_{i_1} ^{a' _1}b^\ast))} G^{\rm up} (\tilde{e}_{i_1} ^{a' _1}b^\ast)\] for some $c \in \c \setminus \{0\}$; here, note that $\varphi_{i_1}(\tilde{e}_{i_1} ^{a' _1}b^\ast) = \max\{a \in \z_{\ge 0} \mid F_{i_1} ^{(a)} G^{\rm up} (\tilde{e}_{i_1} ^{a' _1}b^\ast) \neq 0\}$ (see the assertion in Proposition \ref{property of upper global basis} (2) for $\varphi_i(b)$). Also, it holds that
\begin{align*}
F_{i_1} ^{(\varphi_{i_1}(\tilde{e}_{i_1} ^{a' _1}b^\ast))} G^{\rm up} (\tilde{e}_{i_1} ^{a' _1}b^\ast) &= G^{\rm up} (\tilde{f}_{i_1} ^{\varphi_{i_1}(b^\ast)} b^\ast)\quad({\rm by\ the\ assertion\ in\ Proposition}\ \ref{property of upper global basis} (2)\ {\rm for}\ F_i ^{(\varphi_i(b))})\\
&= G^{\rm up} ((\tilde{e}_{i_1} ^{\varepsilon_{i_1}(b)} b)^\ast)\quad({\rm by\ the\ definition\ of\ dual\ crystals})\\
&= G^{\rm up} ((\tilde{e}_{i_1} ^{a_1} b)^\ast)\quad({\rm by\ the\ definition\ of}\ \Omega_{\bf i}).
\end{align*}
From these, it follows that 
\begin{align*}
(\tilde{\iota}_{1, 2} ^\prime)^\ast(E_{i_1} ^{(a_1 ^\prime)} G^{\rm up} (b^\ast)) &= \tilde{\iota}_{1, 2} ^\ast (c G^{\rm up} ((\tilde{e}_{i_1} ^{a_1} b)^\ast))\\ 
&= c G^{\rm up} (b(2)^\ast)\quad({\rm by\ equation}\ (\ref{equations2})\ {\rm in\ the\ proof\ of\ Theorem}\ \ref{main result}).
\end{align*} 
Therefore, by the same argument as in the proof of Theorem \ref{main result}, we conclude part (1) of Theorem \ref{main result 2}.
\end{proof}

\appendix
\section{Proof of Proposition \ref{inequality}}

In this appendix, we give the (postponed) proof of Proposition \ref{inequality}. We proceed by induction on $r$. If $r = 1$, then it is obvious from the definition that $\mathcal{S}_{\bf i} = \z_{\ge 0}$; hence the assertion is obvious. Let $r \ge 2$, and take $(\widetilde{\bf m}, {\bf a}) \in \z_{\ge 0} ^{r-1} \times \z_{\ge 0} ^r$. For the induction step, it suffices to prove that $(\widetilde{\bf m}, {\bf a}) \in \mathcal{S}_{\bf i}$ if and only if 
\begin{equation}\label{RHS}
(\widetilde{\bf m}_{\le r-2}, {\bf a}^{(r-1)}) \in \mathcal{S} _{{\bf i}_{\le r-1}} \ {\rm and}\ \Psi ^{r, k} _{\bf i}(\widetilde{\bf m}, {\bf a}) \ge 0 \ {\rm for\ all}\ 1 \le k \le r-1,
\end{equation}
where we set $\widetilde{\bf m}_{\le r-2} := (m_1, \ldots, m_{r-2})$ and ${\bf i}_{\le r-1} := (i_1, \ldots, i_{r-1})$. We prove the ``if'' part. Let us fix $(\widetilde{\bf m}, {\bf a}) \in \z_{\ge 0} ^{r-1} \times \z_{\ge 0} ^r$ satisfying (\ref{RHS}), and take a dominant weight $\lambda$ such that $\langle \lambda, h_i \rangle \gg 0$ for all $i \in I$. Then, by the tensor product rule for crystals (the assertion in Proposition \ref{tensor product of crystals} (2) for $\tilde{f}_i$), we see that $T_\lambda(\widetilde{\bf m}, {\bf a}) \neq 0$. Set $m := \langle \lambda, h_{i_r} \rangle$ and $\lambda^\prime := m_{r-1} \varpi_{i_{r-1}} + (\lambda - m \varpi_{i_r})$. Since $\langle \lambda - m \varpi_{i_r}, h_{i_r} \rangle = 0$, Lemma \ref{length of string} implies that $\varphi_{i_r}(b_{\lambda - m \varpi_{i_r}}) = 0$. Therefore, if we identify $b_\lambda$ with $b_{\lambda - m \varpi_{i_r}} \otimes b_{m \varpi_{i_r}}$, then it follows from Corollary \ref{tensor product corollary} (3) that $\tilde{f}_{i_r} ^{a_r} b_\lambda = b_{\lambda - m \varpi_{i_r}} \otimes \tilde{f}_{i_r} ^{a_r} b_{m \varpi_{i_r}}$; below, we will identify $b_{m_{r-1} \varpi_{i_{r-1}}} \otimes \tilde{f}_{i_r} ^{a_r} b_\lambda$ with $b_{\lambda^\prime} \otimes \tilde{f}_{i_r} ^{a_r} b_{m \varpi_{i_r}}$. Now we define ${\bf a}^\prime = (a_1 ^\prime, \ldots, a_{r-1} ^\prime) \in \z_{\ge 0} ^{r-1}$ inductively by 
\begin{align*}
&\tilde{f}_{i_{r-1}} ^{a_{r-1}}(b_{\lambda^\prime} \otimes \tilde{f}_{i_r} ^{a_r} b_{m \varpi_{i_r}})\\ 
=\ &\tilde{f}_{i_{r-1}} ^{a_{r-1} ^\prime} b_{\lambda^\prime} \otimes \tilde{f}_{i_{r-1}} ^{a_{r-1} - a_{r-1} ^\prime} \tilde{f}_{i_r} ^{a_r} b_{m \varpi_{i_r}},\\
&\tilde{f}_{i_{r-2}} ^{a_{r-2}}(b_{m_{r-2} \varpi_{i_{r-2}}} \otimes \tilde{f}_{i_{r-1}} ^{a_{r-1} ^\prime} b_{\lambda^\prime} \otimes \tilde{f}_{i_{r-1}} ^{a_{r-1} - a_{r-1} ^\prime} \tilde{f}_{i_r} ^{a_r} b_{m \varpi_{i_r}})\\ 
=\ &\tilde{f}_{i_{r-2}} ^{a'_{r-2}}(b_{m_{r-2} \varpi_{i_{r-2}}} \otimes \tilde{f}_{i_{r-1}} ^{a_{r-1} ^\prime} b_{\lambda^\prime}) \otimes \tilde{f}_{i_{r-2}} ^{a_{r-2} - a_{r-2} ^\prime} \tilde{f}_{i_{r-1}} ^{a_{r-1} - a_{r-1} ^\prime} \tilde{f}_{i_r} ^{a_r} b_{m \varpi_{i_r}},\\
&\ \vdots
\end{align*}
namely, 
\begin{align*}
&\tilde{f}_{i_k} ^{a_k} \left(b_{m_k \varpi_{i_k}} \otimes T_{\lambda^\prime}(\widetilde{\bf m}_{[k+1, r-2]}, {\bf a}^\prime _{\ge k+1}) \otimes \tilde{f}_{i_{k+1}} ^{a_{k+1} - a_{k+1} ^\prime} \cdots \tilde{f}_{i_{r-1}} ^{a_{r-1} - a_{r-1} ^\prime} \tilde{f}_{i_r} ^{a_r} b_{m \varpi_{i_r}}\right)\\ 
=\ &\tilde{f}_{i_k} ^{a' _k} (b_{m_k \varpi_{i_k}} \otimes T_{\lambda^\prime}(\widetilde{\bf m}_{[k+1, r-2]}, {\bf a}^\prime _{\ge k+1})) \otimes \tilde{f}_{i_k} ^{a_k - a_k ^\prime} \tilde{f}_{i_{k+1}} ^{a_{k+1} - a_{k+1} ^\prime} \cdots \tilde{f}_{i_{r-1}} ^{a_{r-1} - a_{r-1} ^\prime} \tilde{f}_{i_r} ^{a_r} b_{m \varpi_{i_r}};
\end{align*}
here we set $\widetilde{\bf m}_{[k+1, r-2]} := (m_{k+1}, \ldots, m_{r-2})$ and ${\bf a}^\prime _{\ge k+1} := (a' _{k+1}, \ldots, a' _{r-1})$. 

\vspace{2mm}\begin{lem}\label{different case}
For $1 \le k \le r-1$ with $i_k \neq i_r$, the following hold. 
\begin{enumerate}
\item[{\rm (1)}] $\tilde{e}_{i_k}\tilde{f}_{i_r} ^{a} b_{m \varpi_{i_r}} = 0$ for all $a \in \z_{\ge 0}$.
\item[{\rm (2)}] $a_k ^\prime = a_k = a_k ^{(r-1)}$.
\end{enumerate}
\end{lem}

\begin{proof}
For a dominant integral weight $\lambda$, we see from the definition that \[\mathcal{B}(\lambda) = \{\tilde{f}_{j_1} \cdots \tilde{f}_{j_l} b_\lambda \mid l \in \z_{\ge 0},\ j_1, \ldots, j_l \in I\} \setminus \{0\};\] hence it follows that ${\rm wt}(b) \in \lambda - Q_{\ge 0}$ for all $b \in \mathcal{B}(\lambda)$, where $Q_{\ge 0} := \sum_{i \in I} \z_{\ge 0} \alpha_i$. Now, since \[{\rm wt}(\tilde{e}_{i_k}\tilde{f}_{i_r} ^{a} b_{m \varpi_{i_r}}) = \alpha_{i_k} - a \alpha_{i_r} + m \varpi_{i_r} \notin m \varpi_{i_r} - Q_{\ge 0}\] for every $a \in \z_{\ge 0}$, we deduce that $\tilde{e}_{i_k}\tilde{f}_{i_r} ^a b_{m \varpi_{i_r}} = 0$. This proves part (1). In order to prove part (2), we proceed by descending induction on $k$. Take $1 \le k \le r-1$ with $i_k \neq i_r$, and assume that part (2) holds for all $k < j \le r-1$ with $i_j \neq i_r$; in particular, we have $a'_j = a_j$ for all $k < j \le r-1$ with $i_j \neq i_r$. By the definitions of $a'_{k+1}, \ldots, a'_{r-1}$, this implies that the element $T_\lambda(\widetilde{\bf m}_{\ge k+1}, {\bf a}_{\ge k+1})$ must be of the form \[T_{\lambda^\prime}(\widetilde{\bf m}_{[k+1, r-2]}, {\bf a}^\prime _{\ge k+1}) \otimes \tilde{f}_{i_r} ^{a_r + l_{k+1}} b_{m \varpi_{i_r}},\] where $\widetilde{\bf m}_{\ge k+1} := (m_{k+1}, \ldots, m_{r-1})$, ${\bf a}_{\ge k+1} := (a_{k+1}, \ldots, a_r)$, and $l_{k+1} := \sum_{k+1 \le j \le r-1;\ i_j = i_r} (a_j - a' _j)$. It follows from the definition of $\lambda^\prime$ that $\langle \lambda^\prime, h_{i_k} \rangle = m_{r-1} \langle \varpi_{i_{r-1}}, h_{i_k}\rangle + \langle \lambda, h_{i_k} \rangle \gg 0$, which implies that $\tilde{f}_{i_k} ^{a_k} (b_{m_k \varpi_{i_k}} \otimes T_{\lambda^\prime}(\widetilde{\bf m}_{[k+1, r-2]}, {\bf a}^\prime _{\ge k+1})) \neq 0$ by the tensor product rule for crystals (the assertion in Proposition \ref{tensor product of crystals} (2) for $\tilde{f}_i$). Since $\tilde{e}_{i_k}\tilde{f}_{i_r} ^{a_r + l_{k+1}} b_{m \varpi_{i_r}} = 0$ by part (1), Corollary \ref{tensor product corollary} (2) implies that $a' _k = a_k$. Also, it follows from the definition that $a_k ^{(r-1)} = a_k ^{(r)} = a_k$. From these, we conclude part (2). This proves the lemma. 
\end{proof}

\begin{lem}\label{description of l}
For all $1 \le k \le r-1$, \[T_\lambda(\widetilde{\bf m}_{\ge k}, {\bf a}_{\ge k}) = T_{\lambda^\prime}(\widetilde{\bf m}_{[k, r-2]}, {\bf a}^\prime _{\ge k}) \otimes \tilde{f}_{i_r} ^{a_r + l_k} b_{m \varpi_{i_r}},\] where $l_k := \sum_{k \le j \le r-1;\ i_j = i_r} (a_j - a' _j)$.
\end{lem}

\begin{proof}
We see from Lemma \ref{different case} (2) that $a'_j = a_j$ for all $k \le j \le r-1$ with $i_j \neq i_r$. This implies the lemma by the definitions of $a'_k, \ldots, a'_{r-1}$.
\end{proof}

We will show that 
\begin{equation}\label{goal}
a_k ^\prime = a_k ^{(r-1)}\ {\rm and}\ \tilde{e}_{i_k} (b_{m_k \varpi_{i_k}} \otimes T_\lambda(\widetilde{\bf m}_{\ge k+1}, {\bf a}_{\ge k+1})) = 0
\tag*{$(11)_k$}
\end{equation}
for $1 \le k \le r-1$ by descending induction on $k$. Note that if $k = r-1$, then the element $b_{m_k \varpi_{i_k}} \otimes T_\lambda(\widetilde{\bf m}_{\ge k+1}, {\bf a}_{\ge k+1})$ is identified with $b_{\lambda^\prime} \otimes \tilde{f}_{i_r} ^{a_r} b_{m \varpi_{i_r}}$. If $i_{r-1} \neq i_r$, then the first assertion of $(11)_{r-1}$ is just Lemma \ref{different case} (2). Also, it follows from Lemma \ref{different case} (1) that $\tilde{e}_{i_{r-1}} \tilde{f}_{i_r} ^{a_r} b_{m \varpi_{i_r}} = 0$, and hence that $\tilde{e}_{i_{r-1}} (b_{\lambda^\prime} \otimes \tilde{f}_{i_r} ^{a_r} b_{m \varpi_{i_r}}) = 0$. This proves $(11)_{r-1}$ when $i_{r-1} \neq i_r$. If $i_{r-1} = i_r$, then it holds that 
\begin{align*}
\Psi ^{r, r-1} _{\bf i}(\widetilde{\bf m}, {\bf a}) &= m_{r-1} - a_r\quad({\rm by\ the\ definition})\\
 &= \varphi_{i_{r-1}}(b_{\lambda'}) - \varepsilon_{i_{r-1}}(\tilde{f}_{i_r} ^{a_r} b_{m \varpi_{i_r}})\quad({\rm by\ Lemma}\ \ref{length of string}\ {\rm and\ the\ equality}\ i_{r-1} = i_r).
\end{align*} 
Therefore, Corollary \ref{tensor product corollary} (1) implies that
\begin{align*}
a_{r-1} ^\prime &= \min\{a_{r-1}, \Psi ^{r, r-1} _{\bf i}(\widetilde{\bf m}, {\bf a})\}\\ 
&= a_{r-1} ^{(r-1)}\quad({\rm by\ the\ definition\ of}\ a_{r-1} ^{(r-1)}). 
\end{align*}
Moreover, we deduce from the (assumed) inequality $\Psi ^{r, r-1} _{\bf i}(\widetilde{\bf m}, {\bf a}) \ge 0$ that $\varphi_{i_{r-1}}(b_{\lambda'}) \ge \varepsilon_{i_{r-1}}(\tilde{f}_{i_r} ^{a_r} b_{m \varpi_{i_r}})$; hence, by the tensor product rule for crystals (see Proposition \ref{tensor product of crystals} (2)), we obtain \[\tilde{e}_{i_{r-1}}(b_{\lambda^\prime} \otimes \tilde{f}_{i_r} ^{a_r} b_{m \varpi_{i_r}}) = \tilde{e}_{i_{r-1}} b_{\lambda^\prime} \otimes \tilde{f}_{i_r} ^{a_r} b_{m \varpi_{i_r}} = 0.\] This proves $(11)_{r-1}$ when $i_{r-1} = i_r$. Now assume that $k < r-1$, and that $(11)_j$ holds for every $k+1 \le j \le r-1$. Then, it follows that
\begin{equation}
\begin{aligned}
T_\lambda(\widetilde{\bf m}_{\ge k+1}, {\bf a}_{\ge k+1}) &= T_{\lambda^\prime}(\widetilde{\bf m}_{[k+1, r-2]}, {\bf a}^\prime _{\ge k+1}) \otimes \tilde{f}_{i_r} ^{a_r + l_{k+1}} b_{m \varpi_{i_r}}\quad({\rm by\ Lemma\ \ref{description of l}})\\
&= T_{\lambda^\prime}(\widetilde{\bf m}_{[k+1, r-2]}, {\bf a}^{(r-1)} _{\ge k+1}) \otimes \tilde{f}_{i_r} ^{a_r + l_{k+1}} b_{m \varpi_{i_r}}\quad({\rm by\ the\ induction\ hypothesis}).
\end{aligned}
\tag*{(12)}
\end{equation} 
We see from equation (\ref{independent of lambda}) and the assumption $(\widetilde{\bf m}_{\le r-2}, {\bf a}^{(r-1)}) \in \mathcal{S} _{{\bf i}_{\le r-1}}$ (see (\ref{RHS})) that \[\Omega_{{\bf i}_{\le r-1}}(T_{\lambda'}(\widetilde{\bf m}_{\le r-2}, {\bf a}^{(r-1)})) = {\bf a}^{(r-1)},\] which implies that
\begin{equation}\label{assumption}
\tilde{e}_{i_k}(b_{m_k \varpi_{i_k}} \otimes T_{\lambda^\prime}(\widetilde{\bf m}_{[k+1, r-2]}, {\bf a}^{(r-1)} _{\ge k+1})) = 0. 
\tag*{(13)}
\end{equation}
We first consider the case that $i_k \neq i_r$. Recall that, in this case, the first assertion of $(11)_k$ is just Lemma \ref{different case} (2). Moreover, we deduce that 
\begin{align*}
&\tilde{e}_{i_k}(b_{m_k \varpi_{i_k}} \otimes T_\lambda(\widetilde{\bf m}_{\ge k+1}, {\bf a}_{\ge k+1}))\\ 
=\ &\tilde{e}_{i_k}\left(b_{m_k \varpi_{i_k}} \otimes T_{\lambda^\prime}(\widetilde{\bf m}_{[k+1, r-2]}, {\bf a}^{(r-1)}_{\ge k+1}) \otimes \tilde{f}_{i_r} ^{a_r + l_{k+1}} b_{m \varpi_{i_r}}\right)\quad({\rm by\ equation}\ (13))\\ 
=\ &0\quad({\rm by\ equation\ \ref{assumption}\ and\ Lemma\ \ref{different case}\ (1)}). 
\end{align*}
This proves $(11)_k$ when $i_k \neq i_r$. We next consider the case that $i_k = i_r$. In this case, we have 
\begin{align*}
\varepsilon_{i_k}(\tilde{f}_{i_r} ^{a_r + l_{k+1}} b_{m \varpi_{i_r}}) &= a_r + l_{k+1}\quad({\rm since}\ i_k = i_r)\\ 
&= a_r + \sum_{k+1 \le j \le r-1;\ i_j = i_r} (a_j - a_j ^\prime)\quad({\rm by\ the\ definition\ of}\ l_{k+1})\\
&= a_r + \sum_{k+1 \le j \le r-1;\ i_j = i_r} (a_j - a_j ^{(r-1)})\quad({\rm by\ the\ induction\ hypothesis}).
\end{align*}
Also, equation \ref{assumption} and Lemma \ref{length of string} imply that
\begin{align*}
\varphi_{i_k}(b_{m_k \varpi_{i_k}} \otimes T_{\lambda^\prime}(\widetilde{\bf m}_{[k+1, r-2]}, {\bf a}^{(r-1)} _{\ge k+1})) &= \langle {\rm wt}(b_{m_k \varpi_{i_k}} \otimes T_{\lambda^\prime}(\widetilde{\bf m}_{[k+1, r-2]}, {\bf a}^{(r-1)} _{\ge k+1})), h_{i_k} \rangle\\
&= \langle m_k \varpi_{i_k} + \sum_{k+1 \le j \le r-2} m_j \varpi_{i_j} - \sum_{k+1 \le j \le r-1} a_j ^{(r-1)} \alpha_{i_j} + \lambda^\prime, h_{i_k} \rangle\\ 
&({\rm by\ the\ definition\ of}\ T_{\lambda^\prime}(\widetilde{\bf m}_{[k+1, r-2]}, {\bf a}^{(r-1)} _{\ge k+1}))\\
&= \sum_{k \le j \le r-1} m_j \delta_{i_k, i_j} - \sum_{k+1 \le j \le r-1}c_{i_k, i_j} a_j ^{(r-1)}\\ 
&({\rm since}\ \lambda^\prime = m_{r-1} \varpi_{i_{r-1}} + (\lambda - m \varpi_{i_r})\ {\rm and}\ \langle \lambda - m \varpi_{i_r}, h_{i_r} \rangle = 0).
\end{align*}
From these, we deduce that
\begin{align*}
&\varphi_{i_k}(b_{m_k \varpi_{i_k}} \otimes T_{\lambda^\prime}(\widetilde{\bf m}_{[k+1, r-2]}, {\bf a}^{(r-1)}_{\ge k+1})) - \varepsilon_{i_k}(\tilde{f}_{i_r} ^{a_r + l_{k+1}} b_{m \varpi_{i_r}})\\
=\ &\sum_{k \le j \le r-1} m_j \delta_{i_k, i_j} - \sum_{k+1 \le j \le r-1}c_{i_k, i_j} a_j ^{(r-1)} - (a_r + \sum_{k+1 \le j \le r-1;\ i_j = i_r} (a_j - a_j ^{(r-1)})).
\end{align*}
By using the equality $a_j ^{(r-1)} = \min\{a_j, \Psi_{\bf i} ^{r, j}(\widetilde{\bf m}, {\bf a})\}$, it is not hard to verify that this is equal to $\Psi ^{r, k} _{\bf i}(\widetilde{\bf m}, {\bf a})$. Therefore, we conclude from the definition of $a_k ^{(r-1)}$ and Corollary \ref{tensor product corollary} (1) that
\begin{align*}
&\tilde{f}_{i_k} ^{a_k}(b_{m_k \varpi_{i_k}} \otimes T_\lambda(\widetilde{\bf m}_{\ge k+1}, {\bf a}_{\ge k+1}))\\ 
=\ &\tilde{f}_{i_k} ^{a_k}\left(b_{m_k \varpi_{i_k}} \otimes T_{\lambda^\prime}(\widetilde{\bf m}_{[k+1, r-2]}, {\bf a}^{(r-1)}_{\ge k+1}) \otimes \tilde{f}_{i_r} ^{a_r + l_{k+1}} b_{m \varpi_{i_r}}\right)\quad({\rm by\ equation}\ (13))\\
=\ &\tilde{f}_{i_k} ^{a_k ^{(r-1)}}(b_{m_k \varpi_{i_k}} \otimes T_{\lambda^\prime}(\widetilde{\bf m}_{[k+1, r-2]}, {\bf a}^{(r-1)}_{\ge k+1})) \otimes \tilde{f}_{i_r} ^{a_r + l_{k+1} + a_k - a_k ^{(r-1)}} b_{m \varpi_{i_r}},
\end{align*}
from which it follows that $a_k ^\prime = a_k ^{(r-1)}$ by the definition of $a' _k$. Moreover, we deduce from the (assumed) inequality $\Psi ^{r, k} _{\bf i}(\widetilde{\bf m}, {\bf a}) \ge 0$ and the tensor product rule for crystals (see Proposition \ref{tensor product of crystals} (2)) that 
\begin{align*}
&\tilde{e}_{i_k}(b_{m_k \varpi_{i_k}} \otimes T_\lambda(\widetilde{\bf m}_{\ge k+1}, {\bf a}_{\ge k+1}))\\ 
=\ &\tilde{e}_{i_k}\left(b_{m_k \varpi_{i_k}} \otimes T_{\lambda^\prime}(\widetilde{\bf m}_{[k+1, r-2]}, {\bf a}^{(r-1)}_{\ge k+1}) \otimes \tilde{f}_{i_r} ^{a_r + l_{k+1}} b_{m \varpi_{i_r}}\right)\quad({\rm by\ equation}\ (13))\\
=\ &\tilde{e}_{i_k}(b_{m_k \varpi_{i_k}} \otimes T_{\lambda^\prime}(\widetilde{\bf m}_{[k+1, r-2]}, {\bf a}^{(r-1)}_{\ge k+1})) \otimes \tilde{f}_{i_r} ^{a_r + l_{k+1}} b_{m \varpi_{i_r}}\\
=\ &0\quad({\rm by\ equation}\ \ref{assumption}).
\end{align*}
Thus, we obtain \ref{goal} when $i_k = i_r$. This proves \ref{goal} for all $1 \le k \le r-1$. Note that the second assertion of \ref{goal} for $1 \le k \le r-1$ implies that $\Omega_{\bf i} (T_\lambda(\widetilde{\bf m}, {\bf a})) = {\bf a}$. This proves the ``if'' part. Finally, by reversing the arguments above, we can prove the ``only if'' part. This completes the proof of the proposition.

\section{Some examples of string polytopes for generalized Demazure modules}

Here, we compute the explicit forms of $\mathcal{S}_{\bf i}$ and $\Delta_{{\bf i}, {\bf m}}$ for some ${\bf i}$ and ${\bf m}$ by using Proposition \ref{inequality} and Corollary \ref{computation of generalized string polytopes}. First, we consider some reduced words for $w_0$.

\vspace{2mm}\begin{ex}\normalfont
If $G = SL_3 (\c)$ (of type $A_2$) and ${\bf i} = (1, 2, 1)$, ${\bf i}^\prime = (2, 1, 2)$, which are reduced words for $w_0$, then we have \[\mathcal{S}_{\bf i} = \mathcal{S}_{{\bf i}^\prime} = \{(m_1, m_2, a_1, a_2, a_3) \in \z^2 _{\ge 0} \times \z^3 _{\ge 0} \mid a_2 - a_3 + m_1 \ge 0\}.\] Also, for ${\bf m} = (m_1, m_2, m_3) \in \z_{\ge 0} ^3$, the generalized string polytope $\Delta_{{\bf i}, {\bf m}} = \Delta_{{\bf i}^\prime, {\bf m}}$ is identical to the following polytope: \[\{(a_1, a_2, a_3) \in \r^3 _{\ge 0} \mid a_3 \le m_3,\ a_3 - m_1 \le a_2 \le a_3 + m_2,\ a_1 \le a_2 - 2a_3 + m_1 + m_3\}.\] In particular, for ${\bf m} = (1, 1, 1)$, we have \[\Delta_{{\bf i}, {\bf m}} = \Delta_{{\bf i}^\prime, {\bf m}} = \{(a_1, a_2, a_3) \in \r^3 _{\ge 0} \mid a_3 \le 1,\ a_2 \le a_3+1,\ a_1 \le a_2-2a_3+2\};\] see the figure below.
\begin{center}
\includegraphics[width=2.8cm]{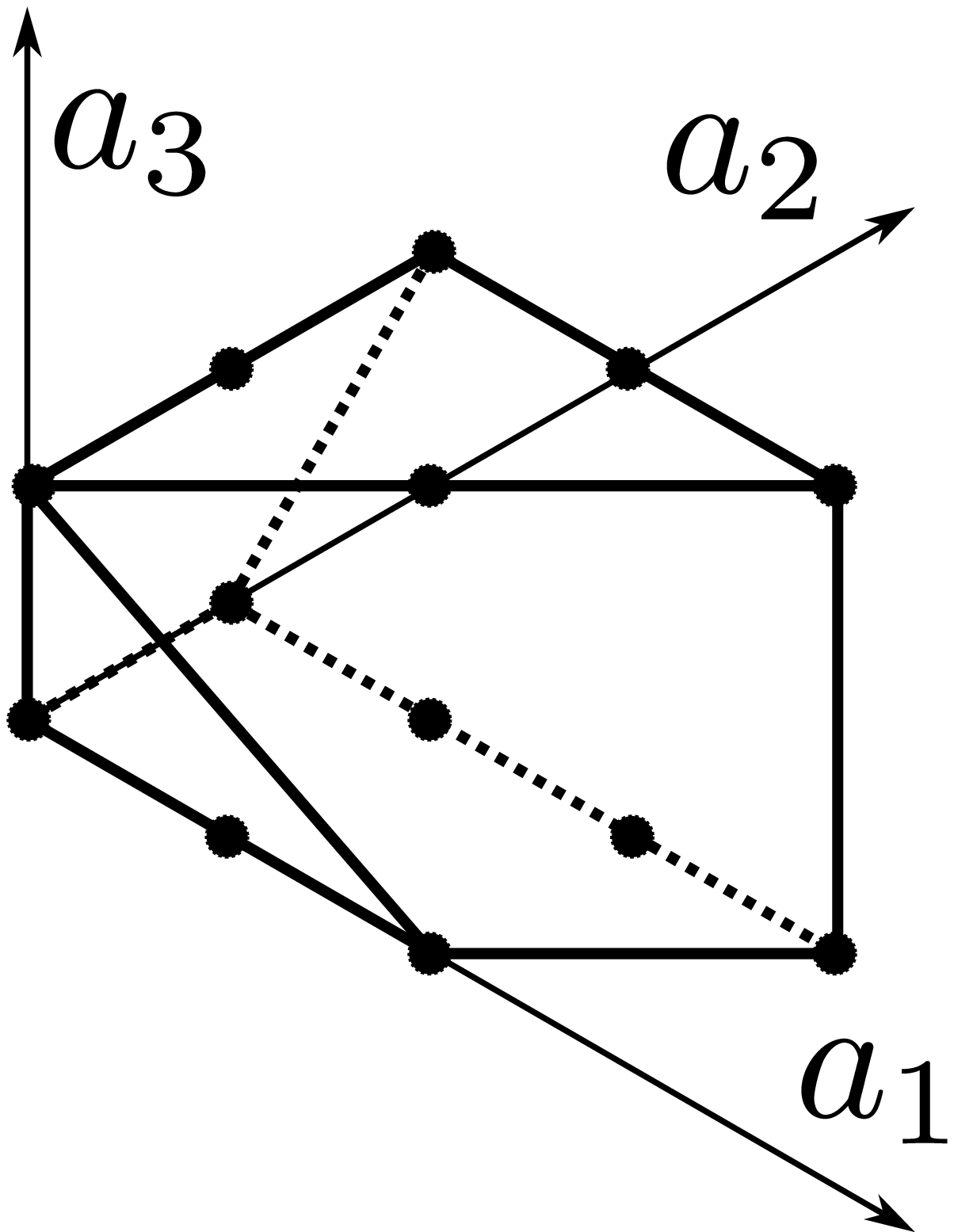}
\end{center}
\end{ex}\vspace{2mm}

\begin{ex}\normalfont
If $G = Sp_4 (\c)$ (of type $C_2$) and ${\bf i} = (1, 2, 1, 2)$, ${\bf i}^\prime = (2, 1, 2, 1)$, which are reduced words for $w_0$, then $\mathcal{S}_{\bf i}$ and $\mathcal{S}_{{\bf i}^\prime}$ are identical to the sets of $(m_1, \ldots, m_3, a_1, \ldots, a_4) \in \z^3 _{\ge 0} \times \z^4 _{\ge 0}$ satisfying the inequalities:
\begin{align*}
&a_3 - a_4 + m_2 \ge 0,\ 2 a_2 - a_3 + m_1 \ge 0,\ a_3 -2 a_4 + m_1 + 2 m_2 \ge 0,\ {\rm and}\\[2mm]
&2a_3 - a_4 + m_2 \ge 0,\ a_2 - a_3 + m_1 \ge 0,\ a_3 - a_4 + m_1 + m_2 \ge 0,\ {\rm respectively}.
\end{align*}
Also, for ${\bf m} = (m_1, \ldots, m_4) \in \z_{\ge 0} ^4$, the generalized string polytopes $\Delta_{{\bf i}, {\bf m}}$ and $\Delta_{{\bf i}^\prime, {\bf m}}$ are identical to the sets of $(a_1, \ldots, a_4) \in \r^4 _{\ge 0}$ satisfying the conditions:
\begin{align*}
&(m_1, \ldots, m_3, a_1, \ldots, a_4) \in \mathcal{S}_{\bf i},\\
&a_4 \le m_4,\ a_3 \le 2a_4 + m_3,\\
&a_2 \le a_3 - 2a_4 + m_2 + m_4,\ a_1 \le 2 a_2 - 2a_3 + 2a_4 + m_1 + m_3,\ {\rm and}\\[2mm]
&(m_1, \ldots, m_3, a_1, \ldots, a_4) \in \mathcal{S}_{{\bf i}^\prime},\\
&a_4 \le m_4,\ a_3 \le a_4 + m_3,\\
&a_2 \le 2a_3 - 2a_4 + m_2 + m_4,\ a_1 \le a_2 - 2a_3 + a_4 + m_1 + m_3,\ {\rm respectively}.
\end{align*}
In particular, for ${\bf m} = (1, 1, 1, 1)$, the generalized string polytopes $\Delta_{{\bf i}, {\bf m}}$ and $\Delta_{{\bf i}^\prime, {\bf m}}$ are given by the inequalities:
\begin{align*}
&0 \le a_4 \le 1,\ 0 \le a_3 \le 2a_4 + 1,\\
&\frac{1}{2} (a_3 - 1) \le a_2 \le a_3 - 2a_4 + 2,\ 0 \le a_1 \le 2 a_2 - 2a_3 + 2a_4 + 2,\ {\rm and}\\[2mm]
&0 \le a_4 \le 1,\ 0 \le a_3 \le a_4 + 1,\\
&a_3 - 1 \le a_2 \le 2a_3 - 2a_4 + 2,\ 0 \le a_1 \le a_2 - 2a_3 + a_4 + 2,\ {\rm respectively}.
\end{align*}
\end{ex}\vspace{2mm}

\begin{ex}\normalfont
If $G$ is of type $G_2$, and ${\bf i} = (1, 2, 1, 2, 1, 2)$, ${\bf i}^\prime = (2, 1, 2, 1, 2, 1)$, which are reduced words for $w_0$, then $\mathcal{S}_{\bf i}$ and $\mathcal{S}_{{\bf i}^\prime}$ are identical to the sets of $(m_1, \ldots m_5, a_1, \ldots, a_6) \in \z_{\ge 0} ^5 \times \z_{\ge 0} ^6$ satisfying the inequalities:
\begin{align*}
&3 a_2 - a_3 + m_1 \ge 0,\ a_3 - a_4 + m_2 \ge 0,\\
&3 a_4 - a_5 + m_3 \ge 0,\ a_5 - a_6 + m_4 \ge 0,\\ 
&2 a_3 - 3 a_4 + m_1 + 3 m_2 \ge 0,\ 2 a_4 - a_5 + m_2 + m_3 \ge 0,\\
&2 a_5 - 3 a_6 + m_3 + 3 m_4 \ge 0,\ 3 a_4 - 2 a_5 + m_1 + 3 m_2 + 2 m_3,\\
& a_5 - 2 a_6 + m_2 + m_3 + 2 m_4 \ge 0,\ a_5 - 3 a_6 + m_1 + 3 m_2 + 2 m_3 + 3 m_4 \ge 0,\ {\rm and}\\[2mm]
&a_2 - a_3 + m_1 \ge 0,\ 3a_3 - a_4 + m_2 \ge 0,\\
&a_4 - a_5 + m_3 \ge 0,\ 3a_5 - a_6 + m_4 \ge 0,\\ 
&2 a_3 - a_4 + m_1 + m_2 \ge 0,\ 2 a_4 - 3a_5 + m_2 + 3m_3 \ge 0,\\
&2 a_5 - a_6 + m_3 + m_4 \ge 0,\ a_4 - 2 a_5 + m_1 + m_2 + 2 m_3,\\
&3 a_5 - 2 a_6 + m_2 + 3m_3 + 2 m_4 \ge 0,\ a_5 - a_6 + m_1 + m_2 + 2 m_3 + m_4 \ge 0,\ {\rm respectively}.
\end{align*}
Also, for ${\bf m} = (m_1, \ldots, m_6) \in \z_{\ge 0} ^6$, the generalized string polytopes $\Delta_{{\bf i}, {\bf m}}$ and $\Delta_{{\bf i}^\prime, {\bf m}}$ are identical to the sets of $(a_1, \ldots, a_6) \in \r^6 _{\ge 0}$ satisfying the conditions:
\begin{align*}
&(m_1, \ldots m_5, a_1, \ldots, a_6) \in \mathcal{S}_{\bf i},\\
&a_6 \le m_6,\ a_5 \le 3a_6 + m_5,\\
&a_4 \le a_5 - 2a_6 + m_4 + m_6,\ a_3 \le 3 a_4 - 2a_5 + 3a_6 + m_3 + m_5,\\
&a_2 \le a_3 - 2 a_4 + a_5 - 2 a_6 + m_2 + m_4 + m_6,\\ 
&a_1 \le 3 a_2 - 2 a_3 + 3 a_4 - 2 a_5 + 3 a_6 + m_1 + m_3 + m_5,\ {\rm and}\\[2mm]
&(m_1, \ldots m_5, a_1, \ldots, a_6) \in \mathcal{S}_{{\bf i}^\prime},\\
&a_6 \le m_6,\ a_5 \le a_6 + m_5,\\
&a_4 \le 3a_5 - 2a_6 + m_4 + m_6,\ a_3 \le a_4 - 2a_5 + a_6 + m_3 + m_5,\\
&a_2 \le 3a_3 - 2 a_4 + 3a_5 - 2 a_6 + m_2 + m_4 + m_6,\\ 
&a_1 \le a_2 - 2 a_3 + a_4 - 2 a_5 + a_6 + m_1 + m_3 + m_5,\ {\rm respectively}.
\end{align*}
In particular, for ${\bf m} = (1, 1, 1, 1, 1, 1) \in \z_{\ge 0} ^6$, the generalized string polytopes $\Delta_{{\bf i}, {\bf m}}$ and $\Delta_{{\bf i}^\prime, {\bf m}}$ are given by the inequalities:
\begin{align*}
&0 \le a_6 \le 1,\ 0 \le a_5 \le 3 a_6 + 1,\\
&\frac{1}{3}(a_5 - 1) \le a_4 \le a_5 - 2 a_6 + 2,\\
&\max\{a_4 - 1, \frac{1}{2}(3 a_4 - 4)\} \le a_3 \le 3 a_4 - 2 a_5 + 3 a_6 + 2,\\
&\frac{1}{3}(a_3 - 1) \le a_2 \le a_3 - 2 a_4 + a_5 - 2 a_6 + 3,\\
&0 \le a_1 \le 3 a_2 - 2 a_3 + 3 a_4 - 2 a_5 + 3 a_6 + 3,\ {\rm and}\\[2mm]
&0 \le a_6 \le 1,\ 0 \le a_5 \le a_6 + 1,\\
&a_5 - 1 \le a_4 \le 3 a_5 - 2 a_6 + 2,\\
&\max\{\frac{1}{3}(a_4 - 1), \frac{1}{2}(a_4 - 2)\} \le a_3 \le a_4 - 2 a_5 + a_6 + 2,\\
&a_3 - 1 \le a_2 \le 3 a_3 - 2 a_4 + 3 a_5 - 2 a_6 + 3,\\
&0 \le a_1 \le a_2 - 2 a_3 + a_4 - 2 a_5 + a_6 + 3,\ {\rm respectively}.
\end{align*}
\end{ex}\vspace{2mm}

Next, we consider some nonreduced words.

\vspace{2mm}\begin{ex}\normalfont
If $G = SL_2 (\c)$ (of type $A_1$), and ${\bf i} = (1, 1, 1)$, which is a nonreduced word, then we have \[\mathcal{S}_{\bf i} = \{(m_1, m_2, a_1, a_2, a_3) \in \z^2 _{\ge 0} \times \z^3 _{\ge 0} \mid m_1 - a_2 \ge 0,\ m_2 - a_3 \ge 0\}.\] Also, for ${\bf m} = (m_1, m_2, m_3) \in \z^3 _{\ge 0}$, the generalized string polytope $\Delta_{{\bf i}, {\bf m}}$ is identical to the set of $(a_1, a_2, a_3) \in \r_{\ge 0} ^3$ satisfying the inequalities:
\begin{align*}
&a_3 \le \min\{m_2, m_3\},\ a_2 \le \min\{m_1, m_2 + m_3 - 2 a_3\},\\
&a_1 \le m_1 + m_2 + m_3 -2 a_2 - 2 a_3.
\end{align*}
In particular, for ${\bf m} = (1, 1, 1)$, we have \[\Delta_{{\bf i}, {\bf m}} = \{(a_1, a_2, a_3) \in \r^3 _{\ge 0} \mid a_3 \le 1,\ a_2 \le \min\{1, 2(1-a_3)\},\ a_1 \le 3 - 2a_2 - 2a_3\};\] see the figure below.
\begin{center}
\includegraphics[width=2.8cm]{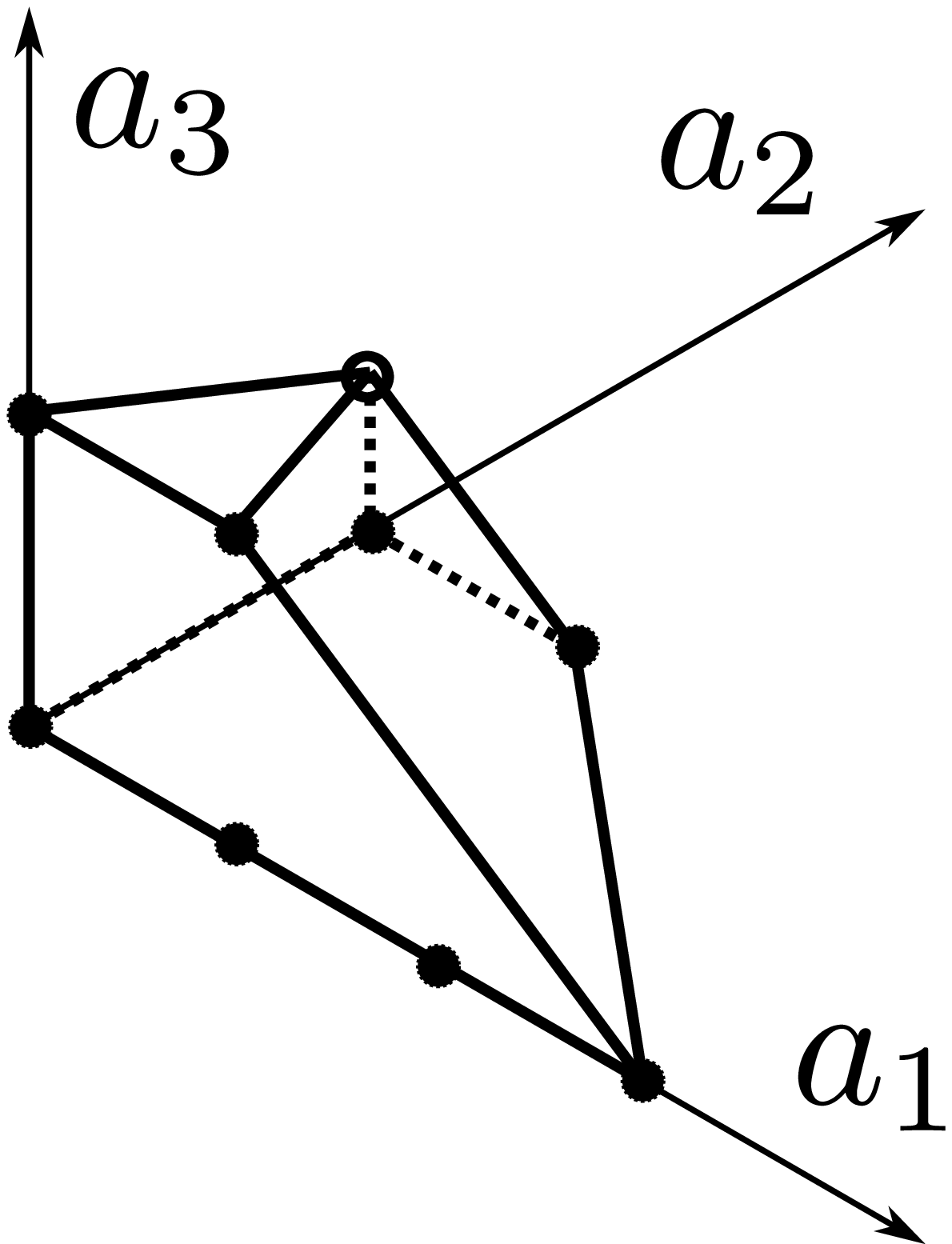}
\end{center}
Note that this is not a lattice polytope since it has $(0, 1, \frac{1}{2})$ as a vertex.
\end{ex}\vspace{2mm}

\begin{ex}\normalfont
If $G = SL_3 (\c)$ and ${\bf i} = (1, 2, 1, 2)$, which is a nonreduced word, then we have \[\mathcal{S}_{\bf i} = \{(m_1, \ldots, m_3, a_1, \ldots, a_4) \in \z^3 _{\ge 0} \times \z^4 _{\ge 0} \mid m_1 + a_2 - a_3 \ge 0,\ m_2 + a_3 - a_4 \ge 0,\ m_1 + m_2 - a_4 \ge 0\}.\] Also, for ${\bf m} = (m_1, \ldots, m_4) \in \z^4 _{\ge 0}$, the generalized string polytope $\Delta_{{\bf i}, {\bf m}}$ is identical to the set of $(a_1, \ldots, a_4) \in \r_{\ge 0} ^4$ satisfying the inequalities:
\begin{align*}
&a_4 \le \min\{m_1 + m_2, m_4\},\ a_4 - m_2 \le a_3 \le a_4 + m_3,\\
&a_3 - m_1 \le a_2 \le a_3 - 2 a_4 + m_2 + m_4,\ a_1 \le a_2 - 2 a_3 + a_4 + m_1 + m_3.
\end{align*}
\end{ex}\vspace{2mm}

\begin{ex}\normalfont
If $m_r \ge 1$, and $m_s$, $1 \le s \le r-1$, are sufficiently larger than $m_{s+1}, \ldots, m_r$, then we have $\Psi ^{j, l} _{\bf i}(\widetilde{\bf m}, {\bf a}) \ge 0$ for all ${\bf a} \in \r_{\ge 0} ^r$. Therefore, the generalized string polytope $\Delta_{{\bf i}, {\bf m}}$ is given by equalities (ii) in Corollary \ref{computation of generalized string polytopes}. This polytope has $2^r$ vertices and $2 r$ facets of dimension $r-1$. 
\end{ex}

\end{document}